\documentclass[a4paper,12pt]{amsart}

\title[Higher Fitting ideals and Sinnott's circular units]
 {On the higher Fitting ideals of Iwasawa modules 
of ideal class groups 
over real abelian fields}
\author{Tatsuya Ohshita}
\address{Department of Mathematics, Kyoto University,
Kyoto 606-8502, Japan}
\email{ohshita@math.kyoto-u.ac.jp}
\date{\today}

\usepackage{amsmath}
\usepackage{amsmath}
\usepackage{amssymb}
\usepackage{amscd}
\usepackage{mathrsfs}
\usepackage[all]{xy}
\usepackage[dvipdfm]{graphicx}
\newtheorem{thm}{Theorem}[section]
\newtheorem{prop}[thm]{Proposition}
\newtheorem{cor}[thm]{Corollary}
\newtheorem{lem}[thm]{Lemma}

\theoremstyle{definition}
\newtheorem{dfn}[thm]{Definition}
\newtheorem{exa}[thm]{Example}
\newtheorem{rem}[thm]{Remark}

\def\plim{\mathop{\underleftarrow{\mathrm{lim}}}\nolimits}
\def\Gal{\mathop{\mathrm{Gal}}\nolimits}
\def\cha{\mathop{\mathrm{char}}\nolimits}
\def\Fitt{\mathop{\mathrm{Fitt}}\nolimits}

\def\det{\mathop{\mathrm{det}}\nolimits}
\def\Im{\mathop{\mathrm{Im}}\nolimits}
\def\Ker{\mathop{\mathrm{Ker}}\nolimits}
\def\Coker{\mathop{\mathrm{Coker}}\nolimits}
\def\Hom{\mathop{\mathrm{Hom}}\nolimits}

\def\ann{\mathop{\mathrm{ann}}\nolimits}
\def\Frob{\text{\rm Fr}}

\newcommand{\mf}[1]{{\mathfrak{#1}}}
\newcommand{\mb}[1]{{\mathbf{#1}}}
\newcommand{\bb}[1]{{\mathbb{#1}}}
\newcommand{\mca}[1]{{\mathcal{#1}}}
\newcommand{\mrm}[1]{{\mathrm{#1}}}
\newfont{\bg}{cmr9 scaled\magstep4}
\newcommand{\bigzerol}{\smash{\lower1.0ex\hbox{\bg 0}}}

\begin{document}

\begin{abstract}
Kurihara established a refinement of the minus-part of the Iwasawa main conjecture for totally real number fields using the higher Fitting ideals 
in his paper \cite{Ku}. 
In this paper, by using Kurihara's methods and Mazur-Rubin theory, 
we study the higher Fitting ideals of the plus-part of 
Iwasawa modules associated the cyclotomic $\bb{Z}_p$-extension of 
abelian fields for an odd prime number $p$.
We define {\em the higher cyclotomic ideals} $\{ \mf{C}_i \}_{i \ge 0}$, which are ideals of the Iwasawa algebra defined by the Kolyvagin derivative classes of circular units, and prove that they give upper and lower bounds 
of the  higher Fitting ideals in some sense, and determine the pseudo-isomorphism classes
of the plus-part of Iwasawa modules. 
Our result can be regarded as an partial analogue of Kurihara's results and 
a refinement of the plus-part of 
the Iwasawa main conjecture for abelian fields.    
\end{abstract}

\maketitle

\section{Introduction}\label{notation}
The Iwasawa main conjecture describes 
the characteristic ideals of certain Iwasawa modules. 
The characteristic ideals are important invariants  
on the structure of finitely generated torsion Iwasawa modules, 
but they are not enough to determine the pseudo-isomorphism classes 
of Iwasawa modules (cf.\ \S 2) completely. 

The higher Fitting ideals have more detailed information 
on Iwasawa modules.  
For instance, the higher Fitting ideals determine 
the pseudo-isomorphism class and 
the least cardinality of generators of 
finitely generated torsion Iwasawa modules. 
(See Remark \ref{remark Fitt and generators} and 
Remark \ref{remark for p-isom class}.)
In \cite{Ku}, Kurihara proved that 
all the higher Fitting ideals of the minus-part of the Iwasawa modules 
associated to the cyclotomic $\bb{Z}_p$-extension of certain CM-fields 
coincide with the higher Stickelberger ideals, 
which are defined by analytic objects arising 
from $p$-adic $L$-functions (cf.\ \cite{Ku} Theorem 1.1). 
His result is a refinement of the minus-part of 
the Iwasawa main conjecture for totally real number fields. 

Here, we study the higher Fitting ideals of the plus-part of 
the Iwasawa modules by similar methods as in \cite{Ku}.  
In this paper, we construct a collection 
$\{\mf{C}_{i,\chi} \}_{i\ge 0}$ of ideals 
of the Iwasawa algebra $\Lambda_{\chi}$, 
which is an analogue of  
Kurihara's higher Stickelberger ideals, 
and prove that the ideals $\mf{C}_{i,\chi}$
give upper and lower bounds of 
the higher Fitting ideals of the plus-part in some sense. 
(In certain cases, the ideals $\mf{C}_{i,\chi}$ determine 
the pseudo-isomorphism class of the plus-part.)
The main tool in \cite{Ku} is the Kolyvagin system of Gauss sums. 
Instead, in this paper, 
we use the Euler system of circular units, 
so we can only treat the Iwasawa modules associated 
to the cyclotomic $\bb{Z}_p$-extension of subfields of cyclotomic fields. 

In order to state the main theorem of this paper, 
we set notations in this paper. 
We fix an odd prime number $p$.  
Let $\overline{\bb{Q}}$ be an algebraic closure of $\bb{Q}$ and 
$K$ a totally real subfield of $\overline{\bb{Q}}$, 
which is a finite abelian extension of $\bb{Q}$. 
We assume that the prime number $p$ is unramified in $K/\bb{Q}$. 
Let $\mu_n$ be the group of all $n$-th roots of unity 
contained in $\overline{\bb{Q}}$. 
For an integer $m$ with $m \ge 0$, 
let $F_m$ be the maximal totally real subfield of 
$K(\mu_{p^{m+1}})$ and $F_\infty := \bigcup_{m \ge 0}F_m$. 
We put $\Gamma_{m,n} :=\Gal(F_m/F_n)$ 
and $\Gamma_m:=\Gal(F_\infty/F_m)$. 
Especially, we write $\Gamma:=\Gamma_0$. 
We fix a topological generator $\gamma \in \Gamma_0$.
Let $\Lambda:=\bb{Z}_p[[\Gal(F_\infty/\bb{Q})]]=
\plim \bb{Z}_p[\Gal(F_m/\bb{Q})]$. 

Put $\Delta := \Gal(F_0/\bb{Q})=\Delta_0\times\Delta_p$, where $\Delta_0$ is the maximal subgroup of $\Delta$ whose order is prime to $p$, and $\Delta_p$ is the $p$-Sylow subgroup of $\Delta$. We denote $D_p$ the decomposition subgroup of $\Delta$ at $p$. (Note that $D_p$ is uniquely determined since $\Delta$ is abelian.) We put $\widehat{\Delta}:=\Hom(\Delta, \overline{\bb{Q}}_p^{\times})$. For any character $\chi \in \widehat{\Delta}$, we denote by $\mca{O}_{\chi}$ the $\bb{Z}_p[\Delta]$-algebra, which is isomorphic to $\bb{Z}_p[\Im\chi]$ as a $\bb{Z}_p$-algebra, and $\Delta$ acts on via $\chi$. 
We denote the $\Lambda$-algebra $\mca{O}_{\chi}[[\Gamma_0]]$ 
by $\Lambda_\chi$, and we identify $\Lambda_\chi$ with
$\mca{O}_\chi[[T]]$ by the isomorphism 
$\Lambda_\chi \simeq \mca{O}_\chi[[T]]$ of $\mca{O}_\chi$ algebras 
defined by $\gamma \mapsto 1+T$. 

For any $\Lambda$-module $M$, we put 
$M_{\chi}:=M\otimes_{\Lambda}\Lambda_{\chi}$. 
We define a $\Lambda$-module $X:= \plim A_m$, 
where $A_m:=A_{F_m}$ is the $p$-Sylow subgroup of 
the ideal class group of $F_m$ 
and the projective limit is taken with respect to the norm maps. 
It is well-known that the $\Lambda_\chi$-module $X_\chi$ is 
finitely generated and torsion.
In this paper, we study the higher Fitting ideals 
$\Fitt_{\Lambda_\chi,i}(X_\chi)$ of $X_\chi$ 
for any non-trivial character $\chi \in \widehat{\Delta}$.
Let $X_{\chi,\mrm{fin}}$ be the largest pseudo-null 
$\Lambda_\chi$-submodule of $X$, and $X'_\chi:=X_\chi/X_{\chi,\mrm{fin}}$.
We treat $X'_\chi$ instead of $X_\chi$ in order to apply 
Kurihara's Euler system arguments, 
which work for finitely generated torsion
$\Lambda$-modules 
whose structures are given by square matrices 
(cf.\ Lemma \ref{no pn} and 
\S \ref{Kurihara's Euler system argument}).

Comparing our setting with the minus part, which Kurihara studied, 
in the case of the plus-part, a problem lies in how to define the ideals which are substitutes for the higher Stickelberger ideals by Kurihara because we do not have elements as the Stickelberger elements in group rings of Galois groups.  
A key idea of this paper lies in the definition of ideal $\mf{C}_{i,\chi}$ 
of $\Lambda$, called {\em the higher cyclotomic ideals} 
for each $i \in \bb{Z}_{ \ge 0}$
which match Kurihara's arguments well.
We shall define these ideals in \S \ref{the section of cyclotomic ideals}, by using the Euler system of circular units (cf.\ Definition \ref{The $i$-th cyclotomic ideal}). Roughly speaking, first, we shall define the ideals 
$\mf{C}_{i,m,N,\chi}$  of the group ring 
$R_{m,N,\chi}:=\bb{Z}/p^N[\Gal(F_m/\bb{Q})]_\chi$ 
generated by images of certain Kolyvagin 
derivatives $\kappa_{m,N}(\xi)$ by {\em all $R_{m,N,\chi}$-homomorphisms} 
$\xymatrix{\big(F_m^\times/(F_m^\times)^{p^N} \big)_\chi \ar[r] 
& R_{m,N,\chi}}$, 
then we shall define $\mf{C}_{i,\chi}$ 
by the projective limit of them.

Let $I$ and $J$ be ideals of $\Lambda_{\chi}$. Then, we write $I \prec J$ if there exists a height two ideal $A$ of $\Lambda_{\chi}$
(called an ``error factor") satisfying $AI \subseteq J$. 
Note that for two ideals $I$ and $J$ of $\Lambda_\chi$, 
we have $I \prec J$ if and only if 
$I\Lambda_{\chi,\mf{p}}\subseteq J\Lambda_{\chi,\mf{p}}$
for all prime ideals $\mf{p}$ of height $1$, 
where we denote the localization of $\Lambda_\chi$ 
at $\mf{p}$ by $\Lambda_{\chi,\mf{p}}$.
We write $I \sim  J$ if $I \prec J$ and $J \prec I$. 
The relation $\sim$ is 
an equivalence relation on ideals of $\Lambda_\chi$.

The following theorem is a rough form of our main theorem in this paper. 
\begin{thm}
\label{Main theorem, Rough}
We assume that the extension degree of $K/\bb{Q}$ is prime to $p$. 
Let $\chi  \in \widehat \Delta$ be a character
satisfying $\chi(p)\ne 1$.
Then, we have 
$$ \Fitt_{\Lambda_\chi,i}(X_\chi)\sim \mf{C}_{i,\chi}$$
for any $i \in \bb{Z}_{\ge 0}$. 
Moreover, we have 
\begin{equation}\label{explicit estimates}
\ann_{\Lambda_\chi}(X_{\chi,\mrm{fin}}) 
\Fitt_{\Lambda_\chi,i}(X'_\chi)\subseteq \mf{C}_{i,\chi}
\end{equation}
for any $i \in \bb{Z}_{\ge 0}$. 
\end{thm}

\begin{rem}
Let $K/\bf{Q}$ and $\chi  \in \widehat \Delta$ be as 
in Theorem \ref{Main theorem, Rough}. 
By a property of the principal (the $0$-th)  Fitting ideals, 
we have
$$\Fitt_{\Lambda_\chi,0}(X_{\chi,\mrm{fin}})\subseteq 
\ann_{\Lambda_\chi}(X_{\chi,\mrm{fin}}).$$ 
Note that we have  
$$\Fitt_{\Lambda_\chi,0}(X_{\chi,\mrm{fin}}) 
\Fitt_{\Lambda_\chi,0}(X'_\chi) = 
\Fitt_{\Lambda_\chi,0}(X_\chi)$$
since $X'_\chi$ is a $\Lambda_\chi$-module of projective dimension one
(cf.\ \ref{fitt0cor}). 
So, we have 
$$\Fitt_{\Lambda_\chi,0}(X_\chi)\subseteq \mf{C}_{0,\chi}.$$
\end{rem}

\begin{rem}
In this paper, we also study the upper bounds of 
the higher Fitting ideals of the plus-part 
when $\Delta_p \ne 0$ or $\chi(p)=1$.
For the precise statement of our main theorem 
for upper bounds on the Fitting ideals, 
including these cases, 
see Theorem \ref{Main theorem}.
\end{rem}

Theorem \ref{Main theorem, Rough} implies that 
the higher cyclotomic ideals give ``true" upper bounds 
in some special cases.

\begin{cor}\label{no pnull case}
Let $K/\bf{Q}$ and $\chi  \in \widehat \Delta$ be as 
in Theorem \ref{Main theorem, Rough}. 
Assume that $X_{\chi}$ is has no non-trivial 
$\Lambda_\chi$-submodules. Then, we have
$$\Fitt_{\Lambda_\chi, i}(X_{\chi})\subseteq \mf{C}_{i, \chi}$$ 
for any $i\in \bb{Z}_{\ge 0}$.
\end{cor}

\begin{rem}
All known examples of $X_\chi$ is pseudo-null 
(cf.\ Greenberg conjecture, for example, 
see \cite{Gree1} Conjecture 3.4),
so we have no non-trivial example for Corollary \ref{no pnull case}
in present. 
\end{rem}

\begin{rem}
We prove the inequality 
$$\ann_{\Lambda_\chi}(X_{\chi,\mrm{fin}}) 
\Fitt_{\Lambda_\chi,i}(X'_\chi)\subseteq \mf{C}_{i,\chi}$$
in \S \ref{Kurihara's Euler system argument} 
by the Euler system argument using 
analogues of Kurihara's elements. 
(See Theorem \ref{Main theorem} 
and Corollary \ref{half of Main theorem, Rough}.) 
Then, in \S \ref{the cyclotomic ideals and Mazur-Rubin theory}, 
we prove the inequalities 
$$\Fitt_{\Lambda_\chi,i}(X'_\chi)\succ \mf{C}_{i,\chi}\hspace{5mm} 
(i \in \bb{Z}_{\ge 0}) $$
by using the results of Mazur-Rubin theory on 
Kolyvagin systems. 
(See Theorem \ref{local lower bounds}.)
Note that the following is already known. 
\begin{itemize}
\item Without Kurihara's elements, 
we can obtain (non-explicit) estimates  
$$\Fitt_{\Lambda_\chi,i}(X'_\chi)\prec \mf{C}_{i,\chi} \hspace{5mm} 
(i \in \bb{Z}_{\ge 0}), $$
which are weaker than the estimates 
(\ref{explicit estimates}) in 
Theorem \ref{Main theorem, Rough},
by using usual Euler system argument and 
the Iwasawa main conjecture. 
(See Remark \ref{usual ES arguments})
\item By the Mazur-Rubin theory in \cite{MR} \S 5, 
it turns out that the principal Kolyvagin systems of 
$(\Lambda_{\chi^{-1}} \otimes_{\bb{Z}_p}\bb{Z}_p(1),
\mca{F}_{\Lambda})$ completely know 
the pseudo-isomorphism class of $X_\chi$. 
(See Theorem \ref{MR5.2.12}, 
Corollary \ref{AandaBandb} and 
Corollary \ref{akandpartial}.)
But the results in \cite{MR} does not 
give explicit estimates of higher Fitting ideals 
of $X_\chi$ in terms of ideal. 
\end{itemize}
The first assertion of Theorem \ref{Main theorem, Rough}
implies that the higher cyclotomic ideals translates 
the pseudo-isomorphism classes of $X_\chi$ 
into the terms of {\em ideals} of $\Lambda_\chi$ 
(cf.\ Remark \ref{remark for p-isom class}).
What is essentially new in this paper is 
the definition of the higher cyclotomic ideals and 
to give stronger estimates (\ref{explicit estimates})
of higher Fitting ideals
which contains more refined information 
than the pseudo-isomorphism class of $X_\chi$
by using Euler system arguments via Kurihara arguments.
The key slogan is that 
the usual Euler system arguments work well only 
when the relation matrix is diagonal, but  
the usual Euler system arguments 
via Kurihara's elements work well even when 
the relation matrix is square. 
\end{rem}

We remark on the relation between 
higher cyclotomic ideals and the structure of 
$A_{0,\chi}:=A_0 \otimes_{\bb{Z}_p} \mca{O}_\chi$. 
By Mazur-Rubin theory in \cite{MR},  
the isomorphism class of the $\mca{O}_\chi$-module $A_{0,\chi}$
is determined by the Kolyvagin systems of $G_{\bb{Q}}$-module 
$\mca{O}_{\chi^{-1}} \otimes_{\bb{Z}_p}\bb{Z}_p(1)$. 
By comparing Mazur-Rubin theory and higher cyclotomic ideals, 
we obtain the following proposition. 
(See Theorem \ref{0-layer} and 
Corollary \ref{cyclotomic ideals and reduction}.)

\begin{prop}
Let $K/\bf{Q}$ and $\chi  \in \widehat \Delta$ be as 
in Theorem \ref{Main theorem, Rough}. 
Then, the following holds. 
\begin{itemize}
\item[(1)] The image of $\mf{C}_{i,\chi}$ in the ring
$$R_{0,\chi}:=\bb{Z}_p[\Gal(F_0/\bb{Q})]_\chi
=\Lambda_{\chi}/(\gamma-1)
=\varprojlim_{N}R_{0,N,\chi}
\simeq \mca{O}_\chi$$ 
coincides with the ideal 
$\mf{C}_{i,F_0,\chi}:= \varprojlim_{N} \mf{C}_{i,0,N,\chi}$  
for any $i \in \bb{Z}_{\ge 0}$. 
\item[(2)] We have 
$\Fitt_{\mca{O}_\chi,i}(A_{0,\chi})
=\mf{C}_{i,F_0,\chi}$
for any $i \in \bb{Z}_{\ge 0}$. 
\end{itemize}
\end{prop}

So, by Nakayama's lemma, we obtain the following corollary. 
(See Corollary \ref{generator and cyclotomic ideals}.)
\begin{cor}
Let $K/\bf{Q}$ and $\chi  \in \widehat \Delta$ be as 
in Theorem \ref{Main theorem, Rough}.
Let $r$ be a non-negative integer. 
Then, the following two properties are equivalent. 
\begin{itemize}
\item[(1)] The least cardinality of generators of 
the $\Lambda_\chi$-module $X_\chi$ is $r$. 
\item[(2)] $\mf{C}_{r-1,\chi} \ne \Lambda_\chi$ 
and $\mf{C}_{r,\chi} = \Lambda_\chi$. 
\end{itemize}
\end{cor}

In \S \ref{the section of Higher Fitting ideals}
we recall the definition and some basic properties of
of higher Fitting ideals. 
In \S \ref{the section of Preliminaries}, 
we recall some preliminary results on Iwasawa theory. 
In \ref{the section of cyclotomic ideals}, 
we define the higher Fitting ideals, 
and prove our main theorem for $i=0$
(Theorem \ref{Main theorem for i=0}). 
In \S \ref{the section of Kurihara's elements for circular units}, 
we recall some basic facts 
on the Kolyvagin derivatives of the Euler system of circular units, 
and induce some elements $x_{m,N}(n)_\chi$ 
of $(F_m^\times/p^N)_\chi$, are 
analogue of Kurihara's elements in \cite{Ku}. 
The elements $x_{m,N}(n)_\chi$  play important roles 
in the Kurihara's Euler system arguments in the proof of 
Theorem \ref{Main theorem, Rough}.  
Especially, Proposition \ref{[] and phi and x} is 
one of the keys of the Euler system arguments. 
In \S \ref{the section of Chebotarev Density Theorem}, 
we prove Proposition \ref{chebo appli}, 
which is a key proposition in 
the induction arguments in Euler system arguments. 
In \S \ref{Kurihara's Euler system argument}, 
by the Euler system argument using 
analogues of Kurihara's elements,
we prove the estimate 
\[
\ann_{\Lambda_\chi}(X_{\chi,\mrm{fin}}) 
\Fitt_{\Lambda_\chi,i}(X'_\chi)\subseteq \mf{C}_{i,\chi}
\]
for any $i \ge 0$ 
(see Theorem \ref{Main theorem}). 
We also treat the case $\Delta_p \ne 0$ or $\chi(p) = 1$.
In \S \ref{the cyclotomic ideals and Mazur-Rubin theory}, 
we compare the higher Fitting ideals with Mazur-Rubin theory. 
We apply theory on Kolyvagin systems established 
by Mazur and Rubin, and 
we prove Theorem \ref{0-layer}, 
which is a result on the ground level,
and the remaining part of Theorem \ref{Main theorem, Rough}.

\subsection*{Notation}
In this paper, we use the following notation.

Let $F$ be a perfect field, and $\overline{F}$ an algebraic closure. 
We denote the absolute Galois group of $F$ 
by $G_F:=\Gal(\overline{\bb{Q}}/F)$.
For a topological abelian group $T$ with continuous $G_F$ action, 
let $H^*(F,T)=H^*(G_F,T)$ be the continuous Galois cohomology group.

In this paper, an algebraic number field is 
a subfield $F$ of a fixed algebraic closure 
$\overline{\bb{Q}}$ of $\bb{Q}$ such that 
the extension degree of $F/\bb{Q}$ is a finite.
For a finite set $\Sigma$ of places of $\bb{Q}$, 
we denote by $\bb{Q}_{\Sigma}$  
the maximal extension field of $\bb{Q}$ 
unramified outside $\Sigma$.
For any algebraic number field $F$,  
we denote the ring of integers of $F$ by $\mca{O}_F$, 
and the $p$-Sylow subgroup of the ideal class group of $F$ 
by $A_F$. 

We define $\bb{Q}_{\infty}/\bb{Q}$ to be 
the cyclotomic $\bb{Z}_p$-extension. 
For any $m \in \bb{Z}_{\ge 0}$, 
we denote by $\bb{Q}_m$ the unique subfield of $\bb{Q}_\infty$ 
whose extension degree over $\bb{Q}$ is $p^m$. 
Note that 
the field $F_m$ is the composite field of $\bb{Q}_m$ and $F_0$
for any $m \in \bb{Z}_{\ge 0} \cup \{\infty \}$. 
We identify $\Gal(\bb{Q}_\infty/\bb{Q})$ as 
$\Gamma=\Gal(F_\infty/F_0)$ by the natural isomorphism
$\xymatrix{\Gamma \ar[r]^(0.3){\simeq} & \Gal(\bb{Q}_\infty/\bb{Q}) }$. 

Let $L/K$ be a finite Galois extension of algebraic number fields. 
Let $\lambda$ be a prime ideal of $K$, 
and $\lambda'$ a prime ideal of $L$ above $\lambda$. 
We denote the completion of $K$ at $\lambda$ by $K_{\lambda}$. 
If $\lambda$ is unramified in $L/K$,  the arithmetic Frobenius 
at $\lambda'$ is denoted by $(\lambda', L/K) \in \Gal(L/K)$. 
We fix a family of embeddings 
$\{ \xymatrix{\ell_{\overline{\bb{Q}}}\colon\overline{\bb{Q}}\ar@{^{(}->}[r] 
& \overline{\bb{Q}}_\ell} \}_{\ell: \rm{prime}}$ 
satisfying the condition (Chb) as follows:
\begin{itemize}
\item[(Chb)] {\em For any subfield $F \subset \overline{\bb{Q}}$ 
which is a finite Galois extension of $\bb{Q}$ 
and any element $\sigma \in \Gal(F/\bb{Q})$, 
there exist infinitely many prime numbers $\ell$ 
such that $\ell$ is unramified in $F/\bb{Q}$ 
and $(\ell_F, F/\bb{Q})= \sigma$, 
where $\ell_F$ is the prime ideal of $\mca{O}_F$ 
corresponding to the  embedding $\ell_{\overline{\bb{Q}}}|_F$.}
\end{itemize}
The existence of a family satisfying the condition (Chb) follows easily from the Chebotarev density theorem.

Let $\ell$ be a prime number. For an algebraic number field $F$, 
let $\ell_F$ be the prime ideal of $F$ 
corresponding to the  embedding $\ell_{\overline{\bb{Q}}}|_F$. 
Then, if $L \supseteq F$ is an extension of algebraic number fields, 
we have $\ell_{L}| \ell_{F}$.

For an abelian group $M$ and a positive integer $n$, we write $M/n$ 
in place of $M/nM$ for simplicity. 
In particular, for the multiplicative group 
$K^\times$ of a field $K$, 
we write $K^\times/p^N$ in place of $K^\times/(K^\times)^{p^N}$. 
We denote by $M_\mrm{tor}$ the kernel of 
the natural homomorphism $\xymatrix{M \ar[r] & M\otimes\bb{Q}}$.

For a $\Lambda$-module $M$, 
we denote the $\Gamma_m$-invariants 
(resp.\ $\Gamma_m$-coinvariants) of $M$ 
by $M^{\Gamma_m}$ (resp.\  $M_{\Gamma_m}$). 

\section*{Acknowledgment}
The author would like to thank Professors Kazuya Kato, Masato Kurihara, Masataka Chida and Tetsushi Ito for their helpful advices. 
The author also thank to Kenji Sakugawa for fruitful 
conversations and discussion with him. 
This work is supported by Grant-in-Aid for JSPS Fellows (22-2753) from Japan Society for the Promotion of Science.

\section{Higher Fitting ideals}\label{the section of Higher Fitting ideals}
We use the same notation as in the previous section.
In particular, we fix a finite abelian field $K$, and 
We define $\Lambda:=\bb{Z}_p[[\Gal(F_\infty/\bb{Q})]]$, 
where $F_\infty$ is the maximal totally real subfield of 
$K(\mu_{p^{\infty}})$.

Here, we recall the definition and 
some basic properties of higher Fitting ideals briefly. 
\begin{dfn}[higher Fitting ideals, see \cite{No} \S 3.1]\label{Fitt}
Let $R$ be an commutative ring, and $M$ a finitely presented $R$-module. Let
\[
\xymatrix{
R^m \ar[r]^{f} & R^n \ar[r] & M \ar[r] & 0 \\
}
\] 
be an exact sequence of $R$-modules. For each $i \ge 0$, we define {\em the $i$-th Fitting ideal} $\Fitt_{R,i}(M)$ as  follows. 
\begin{itemize}
\item When $0 \le i < n$ and $m \ge n-i$, we define $\Fitt_{R,i}(M)$ to be the ideal of $R$ generated by all $(n-i) \times (n-i)$ minors of the matrix corresponding to $f$. 
\item When $0 \le i < n$ and $m < n-i$, we define $\Fitt_{R,i}(M):=0$. 
\item When $i \ge n$, we define $\Fitt_{R,i}(M):=R$. 
\end{itemize}
The definition of these ideals depends only on $M$, and does not depend on the choice of the above exact sequence.
\end{dfn}

\begin{rem}\label{remark Fitt and generators}
Let $R$ be an commutative ring, $S$ an $R$-algebra,  
and $M$ a finitely presented $R$-module, 
Then, by definition of the higher Fitting ideals and
the right exactness of tensor products,
we have 
\[
\Fitt_{S,i}(M\otimes_{R}S)=\Fitt_{R,i}(M)S
\]
for any $i \ge 0$.
\end{rem}

\begin{rem}\label{remark Fitt and generators}
Let $R$ be an commutative ring, 
and $M$ be a finitely presented $R$-module. 
If we have $\Fitt_{R,i}(M) \neq R$, 
then the least cardinality of generators of $M$ 
is greater than $i+1$. 
Note that when $R$ is a local ring or PID, 
the least cardinality of generators of $M$ is  
$i+1$ if and only if $\Fitt_{R,i}(M)\neq R$ 
and $\Fitt_{R,i+1}(M)=R$.
\end{rem}

\begin{rem}\label{remark for p-isom class}
Fix an arbitrary character $\chi \in \widehat{\Delta}$, and
let $M$ and $N$ be $\Lambda_\chi$-modules. 
We say that $M$ is pseudo-null if the order of $M$ is finite.
We write $M \sim_{\mrm{p.i.}} N$ if there exist homomorphism 
$M \longrightarrow N$
whose kernel and cokernel are both pseudo-null,
and we call $M$ is pseudo-isomorphic to $N$.
Note the relation $\sim_{\mrm{p.i.}}$ 
is an equivalence relation on finitely generated torsion 
$\Lambda_\chi$-modules. 
Assume 
\[
M\sim_{\mrm{p.i.}} \bigoplus_{i=1}^n\Lambda_\chi/f_i\Lambda_\chi 
\]
and $f_i$ divides $f_{i+1}$ for $1 \le i \le n-1$. 
Then, we have  
\[
\Fitt_{\Lambda_\chi,i}(M) \sim 
\begin{cases}
(\prod_{k=1}^{n-i}f_k\big) & (\text{if} \ i<n) \\
\Lambda_\chi & (\text{if} \ i \ge n)
\end{cases}
\] 
for any non-negative integer $i$ (cf.\ \cite{Ku} Lemma 8.2). 
In particular, the pseudo-isomorphism class of $M$ is determined 
by the higher Fitting ideals 
$\{\Fitt_{\Lambda_\chi,i}(M) \}_{i\ge 0}$.
\end{rem}

\begin{rem}
Let $M$ be a finitely generated torsion $\Lambda_\chi$-module. 
Then, the characteristic ideal $\cha_{\Lambda_\chi}(M)$ 
is the minimal principal ideal of $\lambda_\chi$ 
containing $\Fitt_{\Lambda_\chi,0}(M)$.
\end{rem}

Let us recall basic properties of higher Fitting ideals. 

\begin{lem}\label{basicFitting}
Let $R$ be a commutative ring, 
and 
\[ 
\xymatrix{0 \ar[r] & L \ar[r] &
M \ar[r] & N \ar[r] & 0 }
\]
be a short exact sequence of 
finitely presented $R$-modules.
Then, we have the following: 
\begin{enumerate}
\item $\Fitt_{R,i}(M) \subseteq \Fitt_{R,i}(L)$ for any $i \ge 0$.
\item $\Fitt_{R,i}(M) \subseteq \Fitt_{R,i}(N)$ for any $i \ge 0$.
\item $\Fitt_{R,i}(L)\Fitt_{R,0}(N) 
\subseteq \Fitt_{R,i}(M)$ for any $i \ge 0$.
\item $\Fitt_{R,0}(L)\Fitt_{R,i}(N) 
\subseteq \Fitt_{R,i}(M)$ for any $i \ge 0$.
\end{enumerate}
\end{lem}

\begin{proof}
Consider free resolutions
\begin{align*} 
& \xymatrix{R^{s} \ar[r]^f & R^r \ar[r] & L \ar[r] & 0,} \\ 
& \xymatrix{R^{s'} \ar[r]^g & R^{r'} \ar[r] & N \ar[r] & 0}
\end{align*}
of $R$-modules $L$ and $N$. 
Let $A \in M_{r,s}(R)$ (resp.\ $B \in M_{r',s'}(R)$) be 
the matrix associated to 
the $R$-linear map $f$ (resp.\ $g$) for standard basis.
Then, we have an exact sequence 
$$\xymatrix{R^{s+s'} \ar[r]^h & R^{r+r'} \ar[r] & M \ar[r] & 0}$$
such that the $(r+r')\times (s+s')$ matrix $C$ associated to $h$ is given by 
\begin{equation*}
C=
\begin{pmatrix}
A & * \\
0 & B 
\end{pmatrix}.
\end{equation*}
All assertions of the lemma follows 
immediately from the computation of minors of the matrix $C$. 
\end{proof}

Later, we use the following lemma on 
principal Fitting ideals of Iwasawa modules. 

\begin{lem}[for example, see \cite{Ku} Theorem 8.1]\label{no pn}
Let $R= \Lambda_\chi \simeq \mca{O}_\chi[[T]]$ and $M$ a finitely generated torsion $R$-module.
Suppose $M$ contains no non-trivial pseudo-null $R$-submodule. 
Then, there exists an exact sequence 
\[
\xymatrix{0 \ar[r] & R^n \ar[r] & R^n \ar[r] & M \ar[r] & 0}
\] 
for some integer $n>0$, and we have 
\[
\Fitt_{R,0}(M)=\cha_R(M).
\]
\end{lem}

By lemma \ref{no pn}, we obtain the following corollary.  

\begin{cor}\label{fitt0cor}
Let $R= \Lambda_\chi \simeq \mca{O}_\chi[[T]]$ 
and $M$ a finitely generated torsion $R$-module.
We denote the maximal pseudo-null $R$-submodule 
of $M$ by $M_{\mrm{fin}}$. 
Then, w have 
\[
\Fitt_{R,0}(M)=\Fitt_{R,0}(M_{\mrm{fin}})
\Fitt_{R,0}(M/M_{\mrm{fin}}).
\]
\end{cor}

\section{Preliminaries}\label{the section of Preliminaries}
In this section, we recall some preliminary results 
on certain Iwasawa modules. 

\subsection{}
In this subsection, we give some remarks on ``$\chi$-quotients"
of $\Lambda$-modules. 
Recall we denote the $p$-component of $\Delta:=\Gal(F_0/\bb{Q})$
by $\Delta_p$, and the maximal subgroup of 
$\Delta$ of order coprime to $p$ by $\Delta_0$. 
Note that $\Lambda_{\chi_0}:=\mca{O}_{\chi_0}[[\Gamma]][\Delta_p]$ 
is flat over $\Lambda$ 
for any $\chi_0 \in \widehat{\Delta}_0$. 
In particular, if the extension degree of 
$K/\bb{Q}$ is prime to $p$, 
then $\Lambda_\chi$ is flat over $\Lambda$
for any $\chi \in \widehat{\Delta}$. 
When the degree of $K/\bb{Q}$ is divisible by $p$,
we have to treat such $\Lambda$-algebras more carefully.

Let $S_{\widehat{\Delta}}$  be a set of all representatives of 
$\Gal (\overline{\bb{Q}}_p/\bb{Q}_p)$-conjugacy classes of $\widehat{\Delta}$.
We put 
\[\xymatrix{\iota_{S_{\widehat{\Delta}}}\colon 
\Lambda \ar[r] & \prod_\chi \Lambda_\chi}
\]
to be the natural homomorphism, 
where $\chi$ runs all elements of $S_{\widehat{\Delta}}$. 
Note that the cokernel of the homomorphism $\iota_{S_{\widehat{\Delta}}}$ is
annihilated by $\left| \Delta_p \right|$.
We use the following elementary lemma.
\begin{lem}\label{elementary lemma}
Let $M$ be a $\Lambda$-module. Then, we consider a natural homomorphism
\[
\xymatrix{\iota_{M,S_{\widehat{\Delta}}}\colon M \ar[r] & 
\prod_{\chi\in S_{\widehat{\Delta}}} M_\chi.}
\]   
Then, the kernel and the cokernel of 
$\iota_{M,S_{\widehat{\Delta}}}$ 
are annihilated by $\left| \Delta_p \right|$.
\end{lem}
\begin{proof}
We consider the exact sequence
\[
\xymatrix{0 \ar[r] & 
\Lambda \ar[r]^(0.4){\iota_{S_{\widehat{\Delta}}}} & 
\prod_\chi \Lambda_\chi  \ar[r] & 
\prod_\chi \Coker \iota_{S_{\widehat{\Delta}}} \ar[r]  
& 0.}
\]
Then, we obtain the exact sequence 
\[
\xymatrix{\mrm{Tor}^{\Lambda}_1(\Coker \iota_{S_{\widehat{\Delta}}},M) 
\ar[r] & M \ar[r] & \prod_\chi M_\chi  \ar[r] & 
\Coker \iota_{S_{\widehat{\Delta}}}\otimes_{\Lambda} M \ar[r]  
& 0.}
\]
Since the $\Lambda$-module 
$\Coker \iota_{S_{\widehat{\Delta}}}$ is annihilated 
by $\left| \Delta_p \right|$, 
the $\Lambda$-modules 
$\Coker \iota_{S_{\widehat{\Delta}}}\otimes_{\Lambda}M$ 
and $\mrm{Tor}^{\Lambda}_1(\Coker \iota_{S_{\widehat{\Delta}}},M)$ 
are annihilated by $\left| \Delta_p \right|$.
\end{proof}

We denote the image of $I_{\Delta}$ 
in $\Lambda_\chi$ by $\bar{I}_{\Delta,\chi}$ 
for each character $\chi \in \widehat{\Delta}$. 

\begin{cor}\label{torsion free}
Let $M$ be $\Lambda$-modules with no non-zero $\bb{Z}_p$-torsion elements. 
We denote the $\Lambda_\chi$-submodule of $M_\chi$ consisting of all $\bb{Z}_p$-torsion elements by $M_{\chi,\mrm{tor}}$.  
Then, the $\Lambda_\chi$-module $M_{\chi,\mrm{tor}}$ is annihilated by $\left| \Delta_p \right|$.
\end{cor}

\begin{proof}
We consider the commutative diagram of natural homomorphisms
\[
\xymatrix{
M \ar@{^{(}->}[rr]^{f} \ar[d]^{\iota_{M,S_{\widehat{\Delta}}}} 
& & 
M\otimes \bb{Q} \ar[d]^{\iota_{M,S_{\widehat{\Delta}}}}_{\simeq} \\
\prod_{\chi\in S_{\widehat{\Delta}}} M_\chi \ar[rr]^{\prod_\chi f_\chi} 
& &  
\prod_{\chi\in S_{\widehat{\Delta}}} (M\otimes \bb{Q})_\chi.
}
\]   
Then, the corollary follows from this commutative diagram and 
Lemma \ref{elementary lemma}.
\end{proof}

\begin{cor}\label{kernel}
Let $M$ and $N$ be $\Lambda$-modules, 
and $\xymatrix{f\colon M \ar[r] & N}$ 
a homomorphism of $\Lambda$-modules. 
We consider the commutative diagram 
\begin{equation}
\label{diagram ker}
\xymatrix{
M \ar[rr]^{f} \ar[d]^{\iota_{M,S_{\widehat{\Delta}}}} 
& & N \ar[d]^{\iota_{M,S_{\widehat{\Delta}}}} \\
\prod_{\chi\in S_{\widehat{\Delta}}} M_\chi 
\ar[rr]^{\prod_\chi f_\chi} & &  
\prod_{\chi\in S_{\widehat{\Delta}}} N_\chi
}
\end{equation}
induced by the homomorphism $f$. 
\begin{itemize}
\item We have 
\[
\iota_{M,S_{\widehat{\Delta}}}(\Ker f) \supseteq \left| \Delta_p \right|^2\cdot \Ker(\prod_\chi f_\chi).
\]
In particular, for each character $\chi \in \widehat{\Delta}$, 
then 
$\left| \Delta_p \right|^2 \Ker f_\chi$ is contained in the image of 
the kernel of $f$ in $M_\chi$.
\item The natural homomorphism 
$\xymatrix{\mrm{Coker}(f_\chi) \ar[r] & 
\big( \mrm{Coker} (f) \big)_\chi}$ 
is an isomorphism of $\Lambda_\chi$-modules 
for any character $\chi \in \widehat{\Delta}$. 
\end{itemize}
\end{cor}

\begin{proof}
The first assertion follows from the diagram (\ref{diagram ker}) 
and Lemma \ref{elementary lemma}.
The second assertion is clear. 
\end{proof}

\subsection{}\label{subsection of global units}
This and the next subsection, we recall some preliminary results
on Iwasawa theory.
In this section, we refer results on unit groups.

Let $m \in \bb{Z}_{\ge 0}$.
We put $U_m:=\big(\mca{O}_{F_m}\otimes\bb{Z}_p\big)^\times$ to be the group of semi-local units at $p$ of $F_m$, and $U_m^1$ to be the maximal pro-$p$-part of $U_m$.
We denote the group of units of $\mca{O}_{F_m}$ by $E_m$, and the Sinnott's circular units in $F_m$ by $C_m$ (cf.\ \cite{Si} \S 4). 
We define $E_m^{\mrm{cl}}$ (resp.\ $C_m^{\mrm{cl}}$) to be the closure of $E_m$ (resp.\ $C_m$) in $U_m$, and $E_m^{1}$ (resp.\ $C_m^{1}$) by $E_m^{\mrm{cl}} \cap U_m^1$ (resp.\ $C_m^{\mrm{cl}} \cap U_m^1$).
We define $U_\infty:=\plim U_m^1$ and $E_\infty:=\plim E_m^1$, where these projective limit is taken with respect to the norm maps. 
Similarly, we define the limit $C_\infty:=\plim C_m^1$ of 
the projective system with respect to  norm maps.

\begin{rem}
By Leopoldt's conjecture for abelian fields (cf.\ \cite{Wa} Corollary 5.32), we have the natural isomorphism 
$\xymatrix{E_m\otimes\bb{Z}_p \ar[r]^(0.59){\simeq}  & E_m^1}$.
We also have the natural isomorphism 
$\xymatrix{ C_m\otimes\bb{Z}_p \ar[r]^(0.59){\simeq} & C_m^{1}}.$
\end{rem}

\begin{prop}\label{rank one}
Let $\chi \in \widehat{\Delta}$ be a non-trivial character. 
There exists a homomorphism 
${\varphi\colon E_{\infty,\chi} \longrightarrow  \Lambda_\chi}$
of $\Lambda_\chi$-modules whose cokernel has finite order, and 
whose kernel is annihilated by $p$-power. 
(Note that if the extension degree of $K/\bb{Q}$ is prime to $p$, 
then the $E_{\infty}$ has no non-trivial 
$p$-torsion element.)
\end{prop}

Let $\chi \in \widehat{\Delta}$. 
We denote the restriction of $\chi$ to $\Delta_0$ by $\chi_0$. 
We define an integer $a_\chi$ by 
\[
a_\chi=
\begin{cases}
0 &\text{if $\chi_0(p) \ne 1$};\\
2 &\text{if $\chi_0(p)= 1$}.
\end{cases}
\]
For each $m \in \bb{Z}_{\ge 0}$, 
we consider the natural homomorphism 
${P_{m}^E \colon (E_{\infty})_{\Gamma_m} 
\longrightarrow E_{m}^{1}}$.  
We define the homomorphism 
${P_{m}^F \colon (E_{\infty})_{\Gamma_m} 
\longrightarrow F_{m}^{\times}\otimes_\bb{Z}\bb{Z}_p}$
to be the composition of $P_{m}^E$ and the inclusion map
${E_{m}^{1} \simeq \mca{O}_{F_m}^\times\otimes_\bb{Z}\bb{Z}_p
\longrightarrow  F_{m}^{\times}\otimes_\bb{Z}\bb{Z}_p}$. 
For each character $\chi \in \widehat{\Delta}$, 
the $\Lambda_\chi$-homomorphisms 
$P_{m}^E$ and $P_m^{F}$ induce the homomorphisms
\begin{align*}
P_{m,\chi}^E \colon & 
(E_{\infty,\chi})_{\Gamma_m} 
 \longrightarrow E_{m}^{1}, \\
P_{m,\chi}^F \colon & 
(E_{\infty,\chi})_{\Gamma_m} 
 \longrightarrow (F_{m}^{\times}\otimes_\bb{Z}\bb{Z}_p)_\chi.
\end{align*}

\begin{prop}\label{proj of E}
Let $\chi \in \widehat{\Delta}$ be a non-trivial character. 
Then, there exist ideals $I_{P_\chi^E}$ and $J_{P_\chi^E}$ 
of $\Lambda_\chi$ of finite indeces such that 
\begin{align*}
(\gamma-1)^{a_\chi/2}\left| \Delta_p\right|^{2} 
I_{P_\chi^E}\Ker P_{m,\chi}^F&=
(\gamma-1)^{a_\chi/2}\left| \Delta_p\right|^{2}
I_{P_\chi^E}\Ker P_{m,\chi}^F =\{0\}, \\
(\gamma-1)^{a_\chi/2}J_{P_\chi^E}\Coker P_{m,\chi}^E & =\{0\}
\end{align*}
for any $m \in \bb{Z}_{\ge 0}$.
\end{prop}

\begin{proof}
For any cyclic group $G$ and any $G$-module $M$,
we denote the Tate cohomology groups by $\hat{H}^i(G,M)$.
Fix a non-negative integer $m$. 
We have the exact sequence 
\[
\xymatrix{0 \ar[r] & \plim \hat{H}^{-1}(\Gamma_{m',m},E_{m'}) \ar[r] & 
(E_{\infty})_{\Gamma_m} \ar[r]^{P_{m}^E} & 
E_{m}^{1} \ar[r] & \plim \hat{H}^{0}(\Gamma_{m',m},E_m') \ar[r] & 0}
\]
of $\Lambda$-modules. 
This exact sequence and Corollary \ref{kernel} imply that 
\begin{align*}
\Coker P_{m,\chi}^E&= 
\varprojlim \hat{H}^{0}(\Gamma_{m',m},E_{m'})_\chi, \\
\ann_{\Lambda_\chi}(\Ker P_{m,\chi}^E) &\supseteq  
\left| \Delta_p\right|^{2}\ann_{\Lambda_\chi}
(\varprojlim \hat{H}^{-1}(\Gamma_{m',m},E_{m'})_\chi). 
\end{align*}
By Lemma 1.2 of \cite{Ru1}, there exists an integer $k$ satisfying
\[
\left| (\gamma-1)\hat{H}^{i}(\Gamma_{m',m},E_{m'}) \right|
\le p^k
\]
for all $i\in \bb{Z}$ and all $m',m \in \bb{Z}_{\ge 0}$ with $m' \ge m$.
(Note that the setting of Lemma 1.2 of \cite{Ru1} 
seems to be different from ours,
but the argument in the proof of this lemma 
works in more general situation containing our case.)
Therefore, the assertion for characters 
$\chi \in \widehat{\Delta}$ satisfying $\chi_0(p)=1$ follows.

We assume that $\chi \in \widehat{\Delta}$ is a character satisfying 
$\chi_0(p)\ne 1$. 
By Corollary \ref{kernel}, 
it is sufficient to show that for any $i \in \bb{Z}$, 
order of 
the $\Lambda_{\chi_0}$-module 
\[
\hat{H}^{i}(\Gamma_{m',m},E_{m'})_{\chi_0}= 
\hat{H}^{i}(\Gamma_{m',m}, (E_{m'}\otimes \bb{Z}_p)_{\chi_0})
\] 
is finite and bounded by a constant independent of $m'$.

Let $m$ be an integer with $m' \ge m$.
Since the $\bb{Z}_p[\Gal(F_{m'}/\bb{Q})]$-module 
$E_{m'}^1\oplus\bb{Z}_p$ contains 
a submodule of finite index 
which is free of rank one, we have an exact sequence
\[
\xymatrix{0 \ar[r] & \bb{Z}_p[\Gal(F_{m'}/\bb{Q})] 
\ar[r]^(0.59)f & E_{m'}^{1}\oplus\bb{Z}_p  \ar[r] & N \ar[r] & 0,}
\]
where $N$ is a $\bb{Z}_p[\Gal(F_{m'}/\bb{Q})]$-module of finite order.
This exact sequence induces the exact sequence 
\[
\xymatrix{0 \ar[r] & \bb{Z}_p[\Gal(F_{m'}/\bb{Q})]_{\chi_0} 
\ar[r]^(0.54){f_{\chi_0}} & E_{{m'},\chi_0}^{1}\oplus 
(\bb{Z}_p)_{\chi_0} 
\ar[r] & N_{\chi_0} \ar[r] & 0.}
\]
Note that we have $(\bb{Z}_p)_{\chi_0}=0$ 
since $\chi_0$ is non-trivial.
We consider the Herbrand quotients, and obtain 
\begin{align*}
\frac{\#\hat{H}^{0}(\Gamma_{m',m},E_{m',\chi_0}^{1})}{
\#\hat{H}^{-1}(\Gamma_{m',m},E_{m',\chi_0}^{1})}
&= \frac{\#\hat{H}^{0}(\Gamma_{m',m},
\bb{Z}_p[\Gal(F_{m'}/\bb{Q})]_{\chi_0})}{
\#\hat{H}^{-1}(\Gamma_{m',m},\bb{Z}_p[\Gal(F_{m'}/\bb{Q})]_{\chi_0)}} 
\cdot \frac{\#\hat{H}^{0}(\Gamma_{m',m},N_{\chi_0})}{
\#\hat{H}^{-1}(\Gamma_{m',m},N_{\chi_0})} \\
&= 1
\end{align*}

Let $E^{(p)}_{m'}$ be the group of $p$-units of $F_{m'}$. 
Then, we have an exact sequence 
\begin{equation}\label{p-units}
\xymatrix{0 \ar[r] & E_{m'}\otimes \bb{Z}_p \ar[r]^i & E^{(p)}_{m'}\otimes\bb{Z}_p  \ar[r] & S_{m'} \ar[r] & 0,}
\end{equation}
where $S_{m'}$ is a $\bb{Z}_p[\Gal(F_{m'}/\bb{Q})]$-submodule of 
$\mca{I}_{F_{m'}}^p\otimes \bb{Z}_p$, 
which is a free $\bb{Z}_p$-module generated 
by all places of $F_{m'}$ above $p$. 
Note that the group $D_p$ acts trivially on $S_{m'}$.
So, the natural homomorphism 
\[
\xymatrix{i_{\chi_0} \colon (E_{m'}\otimes\bb{Z}_p)_{\chi_0} \ar[r] & 
(E^{(p)}_{m'}\otimes\bb{Z}_p)_{\chi_0}}
\]
is an isomorphism. Then, we have 
\begin{equation}\label{Herbrand quatient}
\frac{\#\hat{H}^{0}(\Gamma_{m',m},
(E^{(p)}_{m'}\otimes\bb{Z}_p)_{\chi_0})}{
\#\hat{H}^{-1}(\Gamma_{m',m},(E^{(p)}_{m'}\otimes\bb{Z}_p)_{\chi_0})}
=\frac{\#\hat{H}^{0}(\Gamma_{m',m},E_{m',\chi_0}^{1})}{
\#\hat{H}^{-1}(\Gamma_{m',m},E_{m',\chi_0}^{1})}=1.
\end{equation}
By Corollary in \S 5.4 of \cite{Iw}, 
there exists an integer $r$ such that
\[\left| \hat{H}^{-1}(\Gamma_{m',m},E^{1}_{m'}) \right|
=\left| \hat{H}^{-1}(\Gamma_{m',m},E^{(p)}_{m'}) \right| \le p^r 
\]
for all $m',m \in \bb{Z}$ satisfying $m' \ge m \ge 0$.
By the equality (\ref{Herbrand quatient}), we also have 
\[\left| \hat{H}^{0}(\Gamma_{m',m},E^{1}_{m',\chi_0}) \right|
=\left| \hat{H}^{0}(\Gamma_{m',m},E^{(p)}_{m',\chi_0}) \right| \le p^r
\]
for all $m',m \in \bb{Z}$ satisfying $m' \ge m \ge 0$.
Therefore, 
order of the $\Lambda_{\chi_0}$-modules 
$\hat{H}^{-1}(\Gamma_{m',m},E_{m'})_{\chi_0}$ and  
$\hat{H}^{0}(\Gamma_{m',m}, (E_{m'}\otimes \bb{Z}_p)_{\chi_0})$
are bounded by a constant independent of $m'$.
This completes the proof of proposition.
\end{proof}

We can prove a more refined proposition 
than Proposition \ref{class group} 
in the case of $\Delta_p=0$ and $\chi(p) \ne 1$, 
by the similar argument to 
the proof of \cite{Ru4} Theorem 7.6. 
(We have to replace $X_\infty$ in \cite{Ru4} 
to $\Gal(M_\infty/F_\infty)$ and $U_\infty$ in \cite{Ru4}
to our $U_\infty$, 
where $M_\infty$ is the maximal pro-$p$ extension field 
of $F_\infty$ unramified outside the places above $p$.)

\begin{prop}\label{proj of E, refined}
Assume that the extension degree of $K/\bb{Q}$ is prime to $p$, 
and the character $\chi\in \widehat{\Delta}$ satisfies 
$\chi (p) \ne 1$. 
Then, we can take 
$I_{P_\chi^E}=\Lambda_\chi$ and 
$J_{P_\chi^E}=ann_{\Lambda_\chi}(X_{\chi,\mrm{fin}})$. 
\end{prop}

\begin{rem}\label{imagquad-real}
Note that 
the assertions of \cite{Ru4} Theorem 7.6 are 
for imaginary quadratic base fields,
but we can also prove Proposition \ref{proj of E, refined}
by direct application of \cite{Ru4} Theorem 7.6 
by the manner as follows. 
Let $L$ be an imaginary quadratic field and 
denote the discriminant of $L/\bb{Q}$ by $D_{L/\bb{Q}}$. 
Assume that $(pD_{K/\bb{Q}},D_{L/\bb{Q}})=1$. 
Then $F_0$ and $L$ are linearly disjoint over $\bb{Q}$. 
So, we have the natural isomorphism $\Delta \simeq \Gal(F_0L/L)$,
and regard $\chi$ as a character of $\Gal(F_0L/L)$.
We define $\tilde{X}_\infty:=\varprojlim_{m}A_{F_mL}$ and 
$\tilde{E}_\infty:=
\varprojlim_{m}\mca{O}^\times_{F_mL}\otimes_{\bb{Z}}\bb{Z}_p$. 
Let 
$$\xymatrix{P^{\tilde{E}}_{m,\chi} \colon 
(\tilde{E}_{\infty})_{\Gamma_m} \ar[r] & 
\mca{O}^\times_{F_mL}\otimes_{\bb{Z}}\bb{Z}_p.}$$
be the natural map. 
We can set  
$(K_\infty/K,F,\chi)$ in \cite{Ru4} to 
$(LF_\infty/L, F_mL,\chi)$ in our notation, 
and apply \cite{Ru4} Theorem 7.6 (ii).
Then, it follows that $\mrm{pr}_{\tilde{E},m}$ is injective, 
and 
\[
\mrm{Coker}(P^{\tilde{E}}_{m,\chi})
\simeq \tilde{X}_{\infty,\chi}^{\Gamma_m}\subseteq  
\tilde{X}_{\mrm{fin},\chi},
\]
where $\tilde{X}_{\mrm{fin},\chi}$ is the maximal pseudo-null 
$\Lambda_\chi$-submodule of $\tilde{X}_{\infty,\chi}$. 
Since $p$ is odd, we obtain 
\begin{align*}
\Ker(P^{E}_{m,\chi})&=\Ker(P^{\tilde{E}}_{m,\chi})^{j=1}=0, \\
\mrm{Coker}(P^{E}_{m,\chi}) &=
\mrm{Coker}(P^{\tilde{E}}_{m,\chi})^{j=1}
\simeq (\tilde{X}_{\infty,\chi}^{\Gamma_m})^{j=1}
={X}_{\infty,\chi}^{\Gamma_m}
\subseteq  
{X}_{\mrm{fin},\chi}.
\end{align*}
This implies Proposition \ref{proj of E, refined}. 
\end{rem}

\subsection{} 
In this subsection, we recall some results on the ideal class groups 
and the statement of the Iwasawa main conjecture.

We denote the $p$-Sylow subgroup of the ideal class group 
of $F_m$ by $A_{F_m}$. 
We define the $\Lambda$-module $X$ by 
$X:=\varprojlim A_{F_m}$, where the projective 
limits are taken with respect to the norm maps.
Note that $X$ is a finitely generated $\Lambda$-torsion module.

For the Iwasawa modules $X$, we need the following well-known results.
\begin{prop}[cf.\ \cite{Wa} Lemma 13.15]\label{class group}
The following holds. 
\begin{itemize}
\item[(i)] For each $m \in \bb{Z}_{\ge 0}$, the natural homomorphism 
${X_{\Gamma_m} \longrightarrow A_{F_m}}$
is surjective. 
\item[(ii)] There exist an $\Lambda$-submodule $Y$ of $X$ such that 
$(\gamma-1)X \subseteq Y \subseteq X$, 
and the kernel of the canonical homomorphism 
${X_{\Gamma_m,\chi} \longrightarrow A_{F_m,\chi}}$
is annihilated by $I_A:=ann_{\Lambda}(Y/(\gamma-1)X)$ 
for any $m \in \bb{Z}_{\ge 0}$.
\end{itemize}
\end{prop}

In the case of $\Delta_p=0$ and $\chi (p) \ne 1$, 
we have a more refined proposition 
than Proposition \ref{class group}. 
The following Proposition \ref{class group, refined} 
is proved by the similar way 
to \cite{Ru4} Theorem 5.4 (i). 
(Note that as in Remark \ref{imagquad-real}, 
we can also prove it by the direct application of 
\cite{Ru4} Theorem 5.4 (i).)

\begin{prop}\label{class group, refined}
Assume that the extension degree of $K/\bb{Q}$ is prime to $p$, 
and the character $\chi\in \widehat{\Delta}$ satisfies 
$\chi (p) \ne 1$. 
Then, the natural homomorphism 
\[
\xymatrix{X_{\Gamma_m,\chi} \ar[r] & A_{F_m,\chi}}
\]
is an isomorphism for any $m \in \bb{Z}_{\ge 0}$.
\end{prop}

Here, we recall 
the statement of 
the plus-part of the Iwasawa main conjecture briefly:
\begin{itemize}
\item[] {\em The $\Lambda$-modules $E_\infty$, 
$C_\infty$ and $X$ are as above. 
Let $\chi \in \hat{\Delta}$ be 
an arbitrary character. Then, we have 
$\cha_{\Lambda_\chi}(X_\chi)=
\cha_{\Lambda_\chi}\big( (E_\infty/C_\infty)_\chi\big)$.}
\end{itemize}
(See \cite{CS}, \cite{MW}, \cite{Ru2}, \cite{Grei} Theorem 3.1 
and loc.\ cit.\ Remark c), et al.)
We use the Iwasawa main conjecture in the proof of our main results.

\section{Higher cyclotomic ideals}\label{the section of cyclotomic ideals}
In this section, we define ideals $\mf{C}_{i,\chi}$ of $\Lambda_\chi$ for each $i \in \bb{Z}_{\ge 0}$ by using circular units,
and prove Theorem \ref{Main theorem, Rough} for $i=0$.  

\subsection{}
Here, we define some special circular units 
in order to define ideals $\mf{C}_{i,\chi}$.

We fix an embedding 
$\xymatrix{\overline{\bb{Q}} \ar@{^{(}->}[r] & \bb{C},}$
and regard $\overline{\bb{Q}}$ as a subfield of $\bb{C}$. 
For each positive integer $n$, we define 
\[
\zeta_n:=\exp({2\pi i/n}) 
\in \overline{\bb{Q}} \subset \bb{C},
\]
which is a primitive $n$-th root of unity.
Note that we have $\zeta_{mn}^m=\zeta_n$ 
for any positive integers $m$ and $n$. 

For each integer $N>0$, we define 
\begin{eqnarray*}
\mca{S}_N &:=& \big\{\ell \ | \ \ell \ \text{is a prime number splitting completely in } K(\mu_{p^N})/\bb{Q} \big\}, \\
\mca{N}_N &:=& \big\{\prod_{i=1}^r \ell_i \ | \ r \in \bb{Z}_{>0}, \ \ell_i \in \mca{S}_N \ (i=1,\dots,r), 
\ \text{and} \ \ell_i\ne \ell_j \ \text{if} \ i\ne j\big\} \cup\{1\}. \\
\end{eqnarray*}
In particular, if $\ell \in \mca{S}_N$, then we have $\ell \equiv 1 \mod{p^N}$.

Let $m$ be a non-negative integer, and put $F:=F_m$.
We denote the conductor of $F/\bb{Q}$ by 
$\mf{f}_F= \mf{f}_{F/\bb{Q}}$. 
For a positive integer $n$  prime to $\mf{f}_F$,
we define $H_{F,n}:=\Gal\big( F(\mu_n)/F \big)$. 
For simplicity, we write $H_{n}:=H_{\bb{Q},n}$. 
If $n$ is decomposed in $n=\prod_{i=1}^r \ell_i^{e_i}$, 
where $\ell_1,\dots,\ell_r$ are 
distinct prime numbers and $e_i>0$ for each $i$, 
then we have natural isomorphisms
\begin{align*}
\Gal(F(\mu_n)/\bb{Q}) & \simeq \Gal(F/F_0)\times H_{F,n}, \\
H_{F,n}  \simeq H_n 
& \simeq  H_{\ell_1^{e_1}}\times \dots \times H_{\ell_r^{e_r}}. 
\end{align*}
We identify these groups by the canonical isomorphisms.

\begin{dfn}
Let $m$ be a non-negative integer, 
and $n$ a positive integer prime to $p \mf{f}_K$.
\begin{itemize}
\item[(i)] For each $d \in \bb{Z}_{>1}$ dividing $\mf{f}_{K}$, we define 
$$\eta_m^d(n):=\mrm{N}_{\bb{Q}(\mu_{p^{m+1}nd})/
\bb{Q}(\mu_{p^{m+1}nd})\cap F_m(\mu_n)}
(1-\zeta_d^{p^{-m}}\zeta_{np^{m+1}}) \in F_m(\mu_n)^\times.$$
\item[(ii)] For each $a \in \bb{Z}$ with $(a,p)=1$, 
we define 
\[
\eta_m^{1,a}(n):=\mrm{N}_{K(\mu_{p^{m+1}n})/F_m(\mu_n)}
\bigg(\frac{1-\zeta_n^{p^{-m}} \zeta_{p^{m+1}}^a}
{1-\zeta_n^{p^{-m}} \zeta_{p^{m+1}}}\bigg) \in F_m(\mu_n)^\times.
\]
\end{itemize}
In this paper, we call the elements $\eta_m^d(n)$ and $\eta_m^{1,a}(n)$ 
the basic circular units of $F_m(\mu_n)$. 
\end{dfn}

The following lemma is well-known, and easily verified.
\begin{lem}\label{norm comp}
Let $m$ be a non-negative integer, 
$n$ a positive integer prime to $p \mf{f}_K$, 
and $\ell$ a prime divisor of $\ell$. 
Let $\eta_m(n)^{\bullet}$ be 
a basic circular units of $F_m(\mu_n)$. 
Then, the following holds. 
\begin{itemize}
\item[(i)] We have
\[
\mrm{N}_{F_m(\mu_n)/F_m
(\mu_{n/\ell})}\big( \eta^{\bullet}_m(n) \big)
=\eta_m^{\bullet}(n/\ell)^{1-\Frob_\ell^{-1}},
\]
where $\Frob_{\ell}$ is the arithmetic Frobenius element at $\ell$ in 
$\Gal\big(F_m(\mu_{n/\ell})/\bb{Q} \big)$.
\item[(ii)] We have
\begin{align*}
\mrm{N}_{F_{m+1}(\mu_n)/F_m(\mu_n)}\big( 
\eta_{m+1}^{\bullet}(n) \big)
&=\eta_m^{\bullet}(n). 
\end{align*}
\end{itemize}
\end{lem}

\begin{rem}
Let $\mca{K}$ be a composite field of 
$\bb{Q}_{\infty}$ and $\bb{Q}(\mu_n)$ 
for all positive integers $n$ 
satisfying $(n,p\mf{f}_{K/\bb{Q}})=1$.
We fix a circular unit 
\[
\eta:=\prod_{d|\mf{f}_K}\eta_0^{d}(1)^{u_d}\times
\prod_{i=1}^r \eta_0^{1,a_i}(1)^{v_i} \in F_0^\times,
\]
where $r \in \bb{Z}_{>0}$, $u_d$ and $v_i$ 
are elements of $\bb{Z}[\Gal(F_m/\bb{Q})]$ 
for each positive integers $d$ and $i$ 
with $d|\mf{f}_K$ and $1 \le i \le r$, 
and $a_1, \cdots, a_r$ are integers prime to $p$. 
For any non-negative integer $m$ and 
any positive integer $n$ satisfying $(n,\mf{f}_{K/\bb{Q}})=1$, 
we put 
\[
\eta_m(n):=\prod_{d|\mf{f}_K}\eta_1^{d}(n)^{u_d}\times
\prod_{i=1}^r \eta_m^{1,a_i}(n)^{v_i} \in F_m^\times.
\]
We also denote 
by $\eta_m(n)_\chi$ the image of 
$\eta_m(n)_\chi$ in $H^1\big( \bb{Q}_m,
\mca{O}_{\chi^{-1}}\otimes_{\bb{Z}_p}\bb{Z}_p(1) \big)$
by the natural homomorphism 
\[
\xymatrix{
(F_m(\mu_n)^{\times}\otimes_{\bb{Z}} \bb{Z}_p)_\chi
=H^1\big( \bb{Q}_m,
\bb{Z}_p[\Delta]\otimes_{\bb{Z}_p}\bb{Z}_p(1) \big)_\chi
\ar[r] &
H^1\big( \bb{Q}_m,
\mca{O}_{\chi^{-1}}\otimes_{\bb{Z}_p}\bb{Z}_p(1) \big).}
\]
Then, the collection 
\[
\{\eta_m(n)_\chi \in H^1\big( \bb{Q}_m,
\mca{O}_{\chi^{-1}}\otimes_{\bb{Z}_p}\bb{Z}_p(1) \big) \}_{m,n}
\]
of Galois cohomology classes 
defines an Euler system for 
$(\mca{O}_{\chi^{-1}}\otimes_{\bb{Z}_p}\bb{Z}_p(1),
\mca{K}/\bb{Q},p\mf{f}_{K/\bb{Q}})$ in the sense of \cite{Ru5}. 
\end{rem}

In particular, Lemma \ref{norm comp} (ii) implies 
that $(\eta_m^d(1))_{m \ge 0}$ is a norm compatible system, 
so it is an element of $C_\infty$.  
Later, we use the following result.

\begin{prop}[See \cite{Grei}Lemma 2.3]\label{generate}
The $\Lambda$-module $C_{\infty}$ is generated by 
\[
\{(\eta_m^d(1))_{m \ge 0} \mathrel{|} d \in \bb{Z}_{>1}, d|\mf{f}_{K}\} 
\cup \{(\eta_m^{1,a}(1))_{m \ge 0} \mathrel{|} a \in \bb{Z}, (a,p)=1\}.
\]
\end{prop}

Moreover, 
in the case of $\Delta_p =0$, 
the following result is known. 

\begin{prop}[\cite{Tsu} Lemma 6.2]\label{circular units, free}
Assume that the extension degree of $K/\bb{Q}$ is prime to $p$, 
and the character $\chi\in \widehat{\Delta}$ is non-trivial. 
Then, The $\Lambda_\chi$-module $C_{\infty,\chi}$ is 
generated by one element. 
So, combining with Proposition \ref{rank one}, 
the $\Lambda_\chi$-module $C_{\infty,\chi}$ is free of rank one. 
\end{prop}

\begin{rem}\label{extending to Euler system}
If the extension degree of $K/\bb{Q}$ is prime to $p$, 
and a character $\chi \in \widehat{\Delta}$ 
satisfies $\chi(p) \ne 1$, 
then we can easily show that any circular unit 
$\eta_\chi \in C^1_{0,\chi}$ extends to 
an element 
\[
\{\eta_{m,\chi}\}_m \in C_{\infty,\chi}
=\varprojlim_{m}C^1_{m,\chi}
\]
satisfying $\eta_{0,\chi}=\eta_\chi$. 
This fact implies that any circular unit 
$\eta \in C^1_{0,\chi}$ extends to an Euler system 
$\{\eta_m(n)_\chi \}_{m,n}$ for 
$(\mca{O}_{\chi^{-1}}\otimes_{\bb{Z}_p}\bb{Z}_p(1),
\mca{K}/\bb{Q},p\mf{f}_{K/\bb{Q}})$
which consists of 
$\Lambda_\chi$-linear combination of basic circular units, and 
satisfies $\eta_0(1)_\chi=\eta_\chi$. 
\end{rem}

\subsection{}\label{def of small cyclot unit}
In this subsection, 
we define the higher cyclotomic ideals $\mf{C}_{i,\chi}$ by using 
Kolyvagin derivatives $\kappa^{\bullet}_{m,N}(n)$ of
Euler systems of circular units.
First, let us recall the notion of Kolyvagin derivatives. 
Let $\ell$ be prime number contained in $\mca{S}_N$. 
We shall take a generator $\sigma_{\ell}$ of a cyclic group 
$H_\ell=\Gal \big( \bb{Q}(\mu_n)/\bb{Q} \big)$ as follows. 
We put $N_{\{\ell\} }:=\mrm{ord}_p(\ell-1)$, 
where $\mrm{ord}_p$ is the normalized additive valuation at $p$, 
namely, $\mrm{ord}_p(p)=1$. 
Then, we have $N_{\{\ell\} } \ge N \ge 1$.
By the fixed embedding 
$\ell_{\overline{\bb{Q}}}\colon 
\overline{\bb{Q}}\hookrightarrow \overline{\bb{Q}}_\ell$, 
we regard $\mu_{p^{N_{\{\ell\}}}}$ as a subset of $\bb{Q}_\ell$. 
We identify $\Gal \big(\bb{Q}_\ell(\mu_\ell)/
\bb{Q}_\ell \big)$ with 
$H_\ell=\Gal \big(\bb{Q}(\mu_\ell)/\bb{Q} \big)$ 
by the isomorphism defined by $\ell_{\overline{\bb{Q}}}$. 
Let $F$ be the maximal $p$-extension field of $\bb{Q}$ contained 
in $\bb{Q}(\mu_\ell)$, and $\pi$ a uniformizer of $F_{\ell_F}$. 
We fix a generator $\sigma_\ell$ of $H_\ell$ 
such that 
\[
\pi^{\sigma_\ell-1}\equiv \zeta_{p^{N_{\{\ell\}}} } 
\pmod{\mf{m}_{\ell}},
\]
where $\mf{m}_{\ell}$ is the maximal ideal of $F_{\ell_F}$, 
and $\zeta_{p^{N_{\{\ell\}} }}$ is 
a primitive $p^{N_{\{\ell\}} }$-th root of unity defined as above. 
Note that the definition of $\sigma_\ell$ 
does not depend on the choice of $\pi$.

\begin{dfn}\label{Dn}
\begin{itemize}
\item[(i)] For $\ell \in \mca{S}_N$, we define 
\[
D_\ell:=\sum_{k=1}^{\ell-2}k\sigma_\ell^k \in \bb{Z}[H_\ell].
\]
\item[(ii)] Let $n = \prod_{i=1}^r \ell_i \in \mca{N}_N$, 
where $\ell_i \in \mca{S}_N$ for each $i$. Then, we define 
\[
D_n:=\prod_{i=1}^r D_{\ell_i} \in \bb{Z}[H_n].
\]
\end{itemize}
\end{dfn}

In order to define Kolyvagin derivatives of circular units, 
we use the following well-known lemma. 

\begin{lem}\label{fix}
Let $n \in \mca{N}_N$. 
Then, for each $d \in \bb{Z}_{>1}$ dividing 
$\mf{f}_{K}$ and for each $a \in \bb{Z}$ prime to $p$, 
the images of $\eta_m^d (n)^{D_{n}}$ and $\eta_m^{1,a}(n)^{D_{n}}$ 
in $F_m(\mu_n)^{\times}/p^N$ are fixed by $H_{n}$.
\end{lem}

Note that $H^0(F_n(n),\mu_{p^N})=0$ in our situation, 
so by Kummer theory and Hochschild-Serre spectral sequence, 
the natural homomorphism 
\[
\xymatrix{F_m^\times/p^N \ar[r] & 
\big( F_m(\mu_n)^\times/p^N \big)^{H_n} }
\]
is an isomorphism. 
By Lemma \ref{fix}, we can define Kolyvagin derivatives 
$\kappa^{\bullet}_{m,N}(n)$ of (basic) circular units 
as follows.

\begin{dfn}\label{kappa} 
Let $n \in \mca{N}_N$. 
For each $d \in \bb{Z}_{>1}$ dividing $\mf{f}_{K}$ 
(resp.\ $a \in \bb{Z}$ prime to $p$), 
we define 
\[
\kappa^{d}_{m,N}(n)\in F_m^\times/p^N \hspace{5mm}
(\text{resp.\ } \kappa^{1,a}_{m,N}(n)\in F_m^\times/p^N)
\] 
to be the unique element 
whose image in $F_m(\mu_n)^\times/p^N$ is 
$\eta_{m}^{d}(n)^{D_n}$
(resp.\ $\eta_{m}^{1,a}(n)^{D_n}$). 
\end{dfn}

Now, let us define the higher cyclotomic ideals 
$\{\mf{C}_{i,\chi}\}_{i \ge 0}$.
First, we fix integers $m$ and $N$ satisfying $N\ge m+1>0$. 
Let $\chi \in \widehat{\Delta}$, 
and put 
\begin{align*}
R_{m,N}&:=\bb{Z}/p^N[\Gal(F_m/\bb{Q})]; \\
R_{m,N,\chi}&:=R_{m,N}\otimes_{\bb{Z}_p[\Delta]} 
\mca{O}_{\chi}\simeq \mca{O}_{\chi}/p^N[\Gamma_{m,0}].
\end{align*}
Then, we have 
\[
R_{m,N} \simeq  \Lambda_{\Gamma_m}/p^N, \ \ \ 
R_{m,N,\chi} \simeq  \Lambda_{\chi, \Gamma_m}/p^N,
\]
where $\Lambda_{\Gamma_m}$ (resp.\ $\Lambda_{\chi,\Gamma_m}$) 
denotes the $\Gamma_m$-coinvariant of $\Lambda$ (resp.\ $\Lambda_\chi$).
As in \cite{Ku}, we need the notion called {\em well-ordered}.

\begin{dfn}
\label{well-ordered}
Let $n\in\mca{N}_N$. 
We call $n$ {\em well-ordered} if and only if 
$n$ has a factorization $n=\prod_{i=1}^r \ell_i$ 
with $\ell_i \in \mca{S}_N$ for each $i$ 
such that $\ell_{i+1}$ splits in 
$F_m(\prod_{j=1}^i \ell_j)/\bb{Q}$ for $i=1,\dots,r-1$. 
In other words, $n$ is well-ordered if and only if 
$n$ has a factorization $n=\prod_{i=1}^r \ell_i$ 
such that $$\ell_{i+1} \equiv 1 \pmod{p^N \prod_{j=1}^i \ell_j}$$ 
for $i=1,\dots,r-1$.
We denote by $\mca{N}_N^{\mrm{w.o.}}$ the set of 
all elements in $\mca{N}_N$ which are well-ordered. 
\end{dfn}

Let $n\in\mca{N}_N^{\mrm{w.o.}}$ with 
the decomposition $n=\prod_{i=1}^r \ell_i$, 
where $\ell_i \in \mca{S}_N$ for each $i$. 
We put $\epsilon (n):=r$. 
We define $\mca{W}_{m,N,\chi}(n)$ to be the $R_{m, N, \chi}$-submodule 
of $(F_m^{\times}/p^N)_\chi$ generated by the image of 
\[
\{ \kappa_{m,a}^d (n) \ | \ d\in \bb{Z}_{>0}\ 
\text{dividing}\  \mf{f}_K  \}\cup 
\{ \kappa_m^{1,a} (n) \ | \ a\in \bb{Z}\ \text{prime to}\ p \}.
\]
We put 
$\mca{H}_{m,N,\chi}:=
\Hom_{R_{m,N,\chi}}\big( (F_m^{\times}/p^N)_\chi,R_{m,N,\chi}\big).$
\begin{dfn}
We define $\mf{C}_{i,m,N,\chi}$ to be the ideal of 
$R_{m,N,\chi}$ generated by 
\[
\bigcup_{f\in \mca{H}_{m,N,\chi}} \bigcup_{n} 
f \big( \mca{W}_{m,N,\chi}(n) \big), 
\]
where $n$ runs through all elements of $\mca{N}_N^{\mrm{w.o.}}$ 
satisfying $\epsilon (n)\leq i$. 
\end{dfn}

\begin{rem}
Note that the $R_{m,N,\chi}$-module
$\Hom(\mca{O}_\chi/p^N[\Gamma_{m,0}],\bb{Q}_p/\bb{Z}_p)$ 
is injective and free of rank one. 
So, $R_{m,N,\chi}$ is 
an injective $R_{m,N,\chi}$-module.
In particular, the restriction map
\[
\xymatrix{\mca{H}_{m,N,\chi} \ar[r] &   
\mca{H}_{m,N,\chi}(n):=
\Hom_{R_{m,N,\chi}}\big( 
\mca{W}_{m,N,\chi}(n),R_{m,N,\chi}\big) 
}
\]
is surjective. This implies that
the ideal $\mf{C}_{i,m,N,\chi}$ coincides with 
the ideal of $R_{m,N,\chi}$ generated by 
\[
\bigcup_{n} \bigcup_{f\in \mca{H}_{m,N,\chi}(n)} 
\Im (f), 
\] 
where $\Im (f)$ is the image of $f$, and 
$n$ runs through all elements of $\mca{N}_N^{\mrm{w.o.}}$ 
satisfying $\epsilon (n)\leq i$. 
\end{rem}

In order to define the higher cyclotomic ideals, 
we need the following lemma and its corollary.
(Note that in this paper, 
we use the first assertion of 
the following lemma only when $n=1$.)

\begin{lem}\label{liftinghom}
Let $m_1$, $m_2$, $N_1$ and $N_2$ be integers satisfying 
$m_2 \ge m_1$ and $N_2 \ge N_1$. 
Take a positive integer $n$ prime to $p\mf{f}_{K/\bb{Q}}$.
Then, the following holds. 
\begin{itemize}
\item[(i)] For any $R_{m_2,N_2,\chi}[H_n]$-homomorphism
\[
\xymatrix{
f_2 \colon 
\big(F_{m_2}(\mu_n)^\times/p^{N_2}\big)_\chi \ar[r] & 
R_{m_2,N_2,\chi}[H_n],
}
\]
there exists an $R_{m_2,N_2,\chi}[H_n]$-homomorphism 
\[
\xymatrix{f_1 \colon 
\big(F_{m_1}(\mu_n)^\times/p^{N_1} \big)_\chi \ar[r] & 
R_{m_1,N_1,\chi}[H_n]
}
\]
which makes the diagram
\[
\xymatrix{
\big(F_{m_2}(\mu_n)^\times/p^N\big)_\chi \ar[r]^(0.54){f_2} 
\ar[d]_{N_{F_{m_2}/F_{m_1}}} & 
R_{m_2,N_2,\chi}[H_n] \ar[d]^{\mod{(\gamma^{p^{m_1}}-1,p^{N_1})}}  \\
\big(F_{m_1}(\mu_n)^\times/p^N\big)_\chi \ar@{-->}[r]^(0.54){f_1} & 
R_{m_1,N_1,\chi}[H_n]
}
\]
commute. 
\item[(ii)] Assume $\Delta_p=0$ and $N_1=N_2=:N$. Then, 
for any $R_{m_1,N_1,\chi}[H_n]$-homomorphism
\[
\xymatrix{f_1 \colon 
\big(F_{m_1}(\mu_n)^\times/p^N\big)_\chi \ar[r] & 
R_{m_1,N,\chi}[H_n],
}
\]
there exists an $R_{m_2,N_2,\chi}[H_n]$-homomorphism 
\[
\xymatrix{f_2 \colon 
\big(F_{m_2}(\mu_n)^\times/p^N\big)_\chi \ar[r] & 
R_{m_2,N_2,\chi}[H_n]
}
\]
which makes the diagram
\[
\xymatrix{
\big(F_{m_2}(\mu_n)^\times/p^N\big)_\chi \ar@{-->}[r]^(0.54){f_2} 
\ar[d]_{N_{F_{m_2}/F_{m_1}}} & 
R_{m_2,N_2,\chi}[H_n] \ar[d]^{\mod{(\gamma^{p^{m_1}}-1,p^{N_1})}}  \\
\big(F_{m_1}(\mu_n)^\times/p^N\big)_\chi \ar[r]^(0.54){f_1} & 
R_{m_1,N_1,\chi}[H_n]
}
\]
commute. 
\end{itemize}
\end{lem}

\begin{proof}
Let us prove the first assertion of the lemma. 
Note that we can easily 
reduce the proof of this claim to the following two cases: 
\begin{enumerate}
\item $(m_2,N_2)=(m_1, N_1+1)$;
\item $(m_2,N_2)=(m_1+1, N_1)$.
\end{enumerate}

In the case (1), our lemma is clear. 
We shall show the lemma in the case (2).
We put $m=m_1$, $N=N_1=N_2$, $R_1= R_{{m}, N}[H_n]$, 
$R_2= R_{{m+1}, N}[H_n]$, 
and the natural surjection 
${\mrm{pr}\colon R_2 \longrightarrow R_1}$.
We denote by $\iota_\chi$ the natural homomorphism 
\[
\xymatrix{\iota\colon 
(F_{m}(\mu_n)^\times/p^{N})_\chi \ar[r] & 
(F_{m+1}(\mu_n)^\times/p^{N})_\chi. 
}
\]
We define an element
\[
N_{m+1/m}:=\sum_{\sigma \in \Gal(F_{m+1}/F_{m})}\sigma 
\in R_{2}. 
\]
Then, there is a unique  isomorphism 
\[\xymatrix{ 
\nu_{m+1/m}\colon 
R_{1}
 \ar[r]^(0.36){\simeq} &
N_{m+1,m}R_{2}=
(R_{2})^{\Gamma_{m+1,m}}
}
\]
of $R_{1}$-modules satisfying  
$1 \mapsto N_{m+1/m}$.    
Let $\mca{NF}$ be the image of 
$(F_{m+1}(\mu_n)^\times/p^N)_\chi$
in $(F_{m}(\mu_n)^\times/p^N)_\chi$ by the norm map. 
Note the composition map 
\[
\xymatrix{
\nu_{m+1/m} \circ \mrm{pr}\colon 
R_2 \ar[r] & R_2 
}
\]
coincides with the scalar multiplication by $N_{m+1/m}$, 
so there exist a unique $R_2$-linear homomorphism 
\[
\xymatrix{
f_1^{(0)}\colon 
\iota_\chi(\mca{NF})=N_{m+1/m}(F_{m+1}(\mu_n)^\times
/p^N)_\chi 
\ar[r] & R_2 
}
\]
which makes the diagram
\[
\xymatrix{
 & &
(F_{m+1}(\mu_n)^\times/p^N)_\chi \ar[r]^(0.7){f_2} 
\ar[d]^{\times N_{m+1,m}} \ar[lld]_{N_{F_m/F_{m+1}}} & 
R_{2} \ar[d]^{\mrm{pr}} 
\ar[rr]^(0.4){\times N_{m/m+1}} & & 
R_{2} \\
\mca{NF} \ar[rr]^{\iota_\chi} & & \iota_\chi(\mca{NF}) 
\ar@{-->}[r]^{f_1^{(0)}} & 
R_{1} \ar@{^{(}->}[rru]_{\nu_{m+1/m}} & & 
}
\]
commute.
By injectivity of $R_{1}$, 
we can extend $f^{(0)}_1 \circ \iota_\chi$ to be 
a homomorphism 
\[
\xymatrix{f_1 \colon 
\big(F_{m_1}(\mu_n)^\times/p^{N_1} \big)_\chi \ar[r] & 
R_1
}
\]
satisfying 
$f_1 \mid_{\mca{NF}}=f^{(0)}_1 \circ \iota$. 
This completes the proof of the first assertion.

Next, let us prove the second assertion. 
It is sufficient to show in the case of 
$(m_2,N_2)=(m_1+1, N_1)$. 
Here, we use the same notation as 
in the proof of the first assertion. 
Let 
${f_1 \colon 
\big(F_{m_1}(\mu_n)^\times/p^N\big)_\chi \longrightarrow 
R_1}$ be a given $R_1$-homomorphism. 
Note that the natural homomorphism 
\[
\xymatrix{\iota\colon 
F_{m}(\mu_n)^\times/p^{N} \ar[r] & 
F_{m+1}(\mu_n)^\times/p^{N} }
\]
is injective since we have 
$H^0(F_{m+1}(\mu_n),\mu_{p^{N}})=0$. 
We regard $\big(F_{m_1}(\mu_n)^\times/p^N\big)_\chi$ 
as an $R_2$-submodule of 
$\big(F_{m+1}(\mu_n)^\times/p^{N} \big)_\chi$
by the natural injection $\iota_\chi$. 
Note $R_2$ is 
an injective $R_2$-module, 
so we can extend the homomorphism 
\[
\xymatrix{\nu_{m+1/m}\circ f_1 \colon 
\big(F_{m}(\mu_n)^\times/p^{N} \big)_\chi \ar[r] & 
R_2 }
\]
to an  $R_2$-homomorphism 
$f_2 \colon 
\big(F_{m+1}(\mu_n)^\times/p^{N} \big)_\chi \longrightarrow 
R_2$.
By definition of $f_2$, we obtain 
the commutative diagram
\[
\xymatrix{
\big(F_{m+1}(\mu_n)^\times/p^{N} \big)_\chi 
\ar@{-->}[r]^(0.7){f_2} 
\ar[d]_{\times N_{m+1/m}} & 
R_2 \ar[d]_{\mrm{pr}} 
\ar[rr]^{\times N_{m+1/m}} & & R_2 \\
\big(F_{m}(\mu_n)^\times/p^{N} \big)_\chi \ar[r]^(0.7){f_1} & 
R_1 \ar@{^{(}->}[urr]_{\nu_{m+1/m}} & & }
\]
in the category of $R_2$-modules.
\end{proof}

Let $m_1$, $m_2$, $N_1$ and $N_2$ be integers satisfying 
$N_1\ge m_1+1$, $N_2\ge m_2+1$, $m_2 \ge m_1$ 
and $N_2 \ge N_1$. 
Let $n \in \mca{N}_{N_2}^{\mrm{w.o.}}$ be 
an element satisfying $\epsilon (n)\leq i$.
Note that it follows from Lemma \ref{norm comp} that 
the image of $\mca{W}_{m_2,N_2,\chi}(n)$ 
by the norm map 
\[
\xymatrix{N_{F_{m_2}/F_{m_1}}\colon 
(F_{m_2}^\times/p^{N_2})_\chi \ar[r] & (F_{m_1}^\times/p^{N_1})_\chi}
\]
is contained in $\mca{W}_{m_1,N_1,\chi}(n)$.
Hence the following corollary follows from Lemma \ref{liftinghom}.

\begin{cor}\label{compatibility of cyclotomic ideals}
Let $m_1$, $m_2$, $N_1$ and $N_2$ be integers satisfying 
$N_1\ge m_1+1$, $N_2\ge m_2+1$, $m_2 \ge m_1$ 
and $N_2 \ge N_1$. 
Then, the image of $\mf{C}_{i,m_2,N_2,\chi}$ 
by the projection
$R_{m_2, N_2,\chi} \longrightarrow R_{m_1, N_1,\chi}$
is contained in $\mf{C}_{i,m_1,N_1,\chi}$. 
Moreover, if we assume $\Delta_p=0$ and $N_1=N_2=:N$, then 
the image of $\mf{C}_{i,m_2,N,\chi}$ 
in $R_{m_1,N,\chi}$
coincides with $\mf{C}_{i,m_1,N,\chi}$. 
\end{cor}

Now, we can define the higher cyclotomic ideals.
\begin{dfn}\label{The $i$-th cyclotomic ideal}
We define {\em the $i$-th cyclotomic ideal} 
$\mf{C}_{i,\chi}$ to be the ideal of 
$\Lambda_\chi$ by 
\[
\mf{C}_{i,\chi}:=\plim\mf{C}_{i,m,N,\chi},
\]
where the projective limit is taken 
with respect to the system of the natural homomorphisms 
$\mf{C}_{i,m_2,N_2,\chi} \longrightarrow \mf{C}_{i,m_1,N_1,\chi}$
for integers $m_1, m_2, N_1$ and $N_2$ satisfying 
$N_1\ge m_1+1$, $N_2\ge m_2+1$, 
$m_2 \ge m_1$ and $N_2 \ge N_1$.
\end{dfn}

\subsection{}\label{subsection of 0-th cyclotomic ideals}
We take a generator $\theta \in \Lambda_\chi$ 
of $\cha_{\Lambda_{\chi}}\big( E_{\infty,\chi}/C_{\infty,\chi} \big)$. 
Then, we denote the ideal of $\Lambda_\chi$ generated by  
\[
\bigcup_{\varphi\in \Hom_{\Lambda_\chi}
(E_{\infty,\chi},\Lambda_\chi) } 
\theta^{-1}\varphi(C_{\infty,\chi})
\]
by $I_C$. 
Note that $I_C$ is an ideal of $\Lambda_\chi$ of finite index. 
Moreover, 
by Proposition \ref{circular units, free}, 
we have $I_C=\Lambda_\chi$ if the extension degree of 
$K/\bb{Q}$ is prime to $p$. 

In the rest of this section, 
we will prove the following theorem, 
which is a part of
Theorem \ref{Main theorem, Rough} for $i=0$.

\begin{thm}\label{Main theorem for i=0}
Let $\chi  \in \widehat \Delta$ be any character.
Then, we have
\begin{itemize}
\item[(i)] 
\ $\mf{C}_{0,\chi} \subseteq \Fitt_{\Lambda_\chi,0}(X'_\chi)$.
\item[(ii)]
\ $(\gamma-1)^{a_\chi}\left| \Delta_p \right|^4 
I_C I_{P_{\chi}^E}J_{P_{\chi}^E}  
\Fitt_{\Lambda_\chi,i}(X'_\chi) \subseteq \mf{C}_{i,\chi}$.
\end{itemize}
\end{thm}

In order to prove Theorem \ref{Main theorem for i=0}, 
by Iwasawa main conjecture, 
it is enough to prove the following Proposition \ref{size}. 

\begin{prop}\label{size}
Let $\chi$ be a non-trivial character in $\widehat{\Delta}$. Then, 
\begin{enumerate}
\item $(\gamma-1)^{a_\chi}\left| \Delta_p \right|^4 
I_{P_{\chi}^E} J_{P_{\chi}^E} I_C \cha_{\Lambda_{\chi}}\big( 
E_{\infty,\chi}/C_{\infty,\chi} \big) 
\subseteq \mf{C}_{0,\chi}; $
\item $ \mf{C}_{0,\chi} \subseteq \cha_{\Lambda_{\chi}}\big( 
E_{\infty,\chi}/C_{\infty,\chi} \big).$
\end{enumerate}
\end{prop}

\begin{proof}
Let us prove Proposition \ref{size}. 
First, we prove the first assertion. 
By Proposition \ref{rank one}, we can take a homomorphism 
$\varphi\colon E_{\infty,\chi} \longrightarrow \Lambda_\chi$ 
of finite index.
This induces a homomorphism 
\[
\xymatrix{\bar{\varphi}_{m,N,\chi}\colon 
(E_{\infty,\chi})_{\Gamma_m}/p^N \ar[r] & R_{m, N,\chi}. }
\]  
We take arbitrary elements 
$\delta_1 \in I_{P_{\chi}^E}$ 
and $\delta_2 \in J_{P_{\chi}^E}$. 

\begin{lem}\label{delta kowaza}
Let $\mca{NO}_{m,N,\chi}$ be the image of the homomorphism 
\[
\xymatrix{(E_{\infty,\chi})_{\Gamma_m}/p^N  \ar[r] & 
\big(F_m^{\times}/p^N\big)_\chi}
\]
induced by the homomorphism $P_{m,\chi}^{F}$ defined 
in \S \ref{subsection of global units}. 
Then, the kernel of the natural homomorphism is annihilated by 
$(\gamma-1)^{a_\chi}\left| \Delta_p \right|^4 
I_{P_{\chi}^E} J_{P_{\chi}^E}$,
and 
there exists a homomorphism  
$\psi\colon \mca{NO}_{m,N,\chi} \longrightarrow R_{m,N,\chi}$ 
which makes the diagram
\[
\xymatrix{
(C_{\infty,\chi})_{\Gamma_m}/p^N \ar[r] \ar@{->>}[d] & 
(E_{\infty,\chi})_{\Gamma_m}/p^N 
\ar[rrrr]^(0.55){(\gamma -1)^{a_\chi} 
\left| \Delta_p \right|^4 \delta_1 \delta_2 
\cdot \bar{\varphi}_{m,N,\chi}}  
\ar@{->>}[d] & & & & R_{m,N,\chi} \\
\mca{W}_{m,N,\chi}(1)  \ar@{^{(}->}[r] & 
\mca{NO}_{m,N,\chi} \ar@{-->}[rrrru]_{\psi}  & & & }
\] 
commute.
\end{lem}
This lemma follows from Corollary \ref{kernel} 
and Proposition \ref{proj of E} straightforward. 

Lemma \ref{delta kowaza} implies 
the first assertion of Proposition \ref{size}.
Indeed, since the image of $(C_{\infty,\chi})_{\Gamma_m}$ 
in $F_m^{\times}/p^N$ coincides with $\mca{W}_{m,N,\chi}(1)$ 
by Proposition \ref{generate}, 
we have 
\[
(\gamma -1)^{a_\chi} \delta_1 \delta_2 
\cdot \bar{\varphi}_{m,N,\chi}\big( \text{the image of } 
(C_{\infty,\chi})_{\Gamma_m} \big) 
\subseteq  \psi(\mca{W}_{m,N,\chi}(1)) 
\subseteq  \mf{C}_{0,m,N,\chi}.
\]

Next, we prove the second assertion. 
We take arbitrary elements 
$\delta'_1 \in \ann_{\Lambda_\chi}(\Ker \varphi)$ 
and $\delta'_2 \in \ann_{\Lambda_\chi}(\Coker \varphi)$. 
In particular, we may take some $p$-powers 
as $\delta'_1$ and $\delta'_2$.
We shall construct a homomorphism 
\[
\psi_{\delta'_1,\delta'_2} \in 
\Hom_{R_{m, N,\chi} }( R_{m, N,\chi}, 
(E_{\infty,\chi})_{\Gamma_m}/p^N ),
\]
which can be regarded as ``inverse" of $\bar{\varphi}_{m,N,\chi}$
in some sense as follows. 
For each $x \in R_{m, N,\chi}$, 
we take $y \in (E_{\infty,\chi})_{\Gamma_m}/p^N $ such that 
\[
\bar{\varphi}_{m,N,\chi}(y)=(\gamma-1)^{a_\chi/2} \delta'_2 x.
\]
Then, we define
\[
{\psi}_{\delta'_1,\delta'_2}(x)
:= (\gamma-1)^{a_\chi/2} \delta'_1 y 
\in (E_{\infty,\chi})_{\Gamma_m}/p^N.
\]
The definition of 
$\psi_{\delta'_1,\delta'_2}(x)$ 
is independent of the choice of $y$, and 
${\psi}_{\delta'_1,\delta'_2}$ 
is contained in 
$\Hom_{R_{m, N,\chi} }( R_{m, N,\chi}, 
(E_{\infty,\chi})_{\Gamma_m}/p^N )$.

Let $f\in \mca{H}_{m,N,\chi}(n)$ 
be an arbitrary homomorphism.
Since $R_{m,N,\chi}$ is an injective 
$R_{m,N,\chi}$-module, 
there exists a homomorphism 
$\tilde{f} \colon 
\big(F_m^{\times}/p^N\big)_\chi \longrightarrow  R_{m,N,\chi}$
whose restriction to $\mca{W}_{m,N,\chi}(1)$ coincides with $f$. 
Then, we have an element $a \in R_{m,N,\chi}$ 
which makes the following diagram
\[
\xymatrix{
(E_{\infty,\chi})_{\Gamma_m}/p^N \ar[ddd]& & & & & R_{m,N,\chi} 
\ar[lllll]_{{\psi}_{\delta'_1,\delta'_2}} 
\ar@{.>}[ddd]^{\times a}  \\
 & & (C_{\infty,\chi})_{\Gamma_m}/p^N \ar[d]^{(\gamma-1)^{a_\chi}
\delta'_1 \delta'_2\cdot j} 
\ar[llu]^(0.47){(\gamma-1)^{a_\chi}\delta'_1 \delta'_2\cdot i \ \ \ \ } 
\ar[rrru]_(0.48){\ \ \bar{\varphi}_{m,N,\chi} \circ i}&  & &   \\
& &\mca{W}_{m,N,\chi}(1) \ar[rrrd]^(0.54){f}  \ar@{_{(}->}[lld] &&&  \\
\big(F_m^{\times}/p^N\big)_\chi \ar@{-->}[rrrrr]^(0.47){\tilde{f}} 
& & & & &  R_{m,N,\chi} }
\] 
commute, where the right vertical arrow $\times a$ 
is the scalar multiplication by $a$, 
and maps 
$i\colon (C_{\infty,\chi})_{\Gamma_m}/p^N \longrightarrow 
(E_{\infty,\chi})_{\Gamma_m}/p^N$ and 
$j\colon (C_{\infty,\chi})_{\Gamma_m}/p^N \longrightarrow  
\mca{W}_{m,N,\chi}(1)$
are the natural homomorphism.
From this diagram, we obtain 
\[
{(\gamma-1)^{a_\chi}\delta'_1 \delta'_2}
f(\mca{W}_{m,N,\chi}(1))_{\Gamma_m}/p^N) 
=  a\bar{\varphi}_{m,N,\chi} \circ 
i((C_{\infty,\chi})_{\Gamma_m}/p^N).
\]
Since the characteristic ideal 
$\cha_{\Lambda_{\chi}}\big( E_{\infty,\chi}/C_{\infty,\chi} \big)$ 
is a principal ideal of 
$\Lambda_\chi$ not containing neither $p$ nor $\gamma-1$,
the second assertion of this proposition follows.
\end{proof}

\section{Kurihara's elements for circular units}\label{the section of Kurihara's elements for circular units}
In this section, we recall some basic facts 
on the Euler system of circular units, 
and define some elements $x_{m,N}(n)_\chi$ 
of $(F_m^\times/p^N)_\chi$, which are 
analogue of Kurihara's elements in \cite{Ku}. 

\subsection{}\label{2-maps}
Let $\chi \in \widehat\Delta$ be a character. 
Recall we put $R_{m,N}:=\bb{Z}/p^N[\Gal(F_m/\bb{Q})]$ 
and $R_{m,N,\chi}:=R_{m,N}\otimes_{\bb{Z}_p[\Delta]}\mca{O}_{\chi}$.  
Here, for each $\ell\in\mca{S}_N$, 
we shall recall the definition of two  homomorphisms 
\begin{align*}
[\cdot ]_{m,N,\chi}^\ell\colon 
& (F_m^{\times}/p^N)_{\chi} 
\longrightarrow R_{m, N, \chi} 
&(\text{cf.\ Definition \ref{[]}}), \\
\bar{\phi}_{m,N,\chi}^\ell\colon & 
(F_m^{\times}/p^N)_{\chi} 
\longrightarrow R_{m, N, \chi} 
 &(\text{cf.\ Definition \ref{phi}}),
\end{align*}  
which are important 
in the induction part in the Euler system arguments.
The homomorphism $[\cdot ]_{m,N,\chi}^\ell$ is defined 
by the valuations of the places above $\ell$, 
and $\bar{\phi}_{m,N,\chi}^\ell$ is defined 
by the local reciprocity maps. 

First, we define $[\cdot ]_{m,N,\chi}^\ell$. 
Let $F$ be an algebraic number field. 
We define 
\[
\mca{I}_F:=\mrm{Div}\big(\mrm{Spec}(\mca{O}_F) \big)
\] 
to be the divisor group, and we write its group law additively.
We define the homomorphism 
$(\cdot )_F
\colon F^{\times} \longrightarrow \mca{I}_F$ 
by 
\[
(x)_F=\sum_{\lambda}\mrm{ord}_{\lambda}(x)\lambda,
\] 
where $\lambda$ runs through all prime ideals of $\mca{O}_F$, 
and $\mrm{ord}_{\lambda}$ is the normalized valuation of $\lambda$. 
For any prime number $\ell$, 
we define  $\mca{I}_F^\ell$ to be 
the subgroup of $\mca{I}_F$ generated by 
all prime ideals above $\ell$. 
Then, we define 
$(\cdot )_F^\ell\colon F^{\times} \longrightarrow \mca{I}_F^\ell$ 
by
\[
(x)_F^\ell
=\sum_{\lambda|\ell}\mrm{ord}_{\lambda}(x)\lambda.
\]

Recall that we fix a family of embeddings 
$\{ \xymatrix{\ell_{\overline{\bb{Q}}}\colon\overline{\bb{Q}} 
\ar@{^{(}->}[r] &
\overline{\bb{Q}}_\ell} \}_{\ell: \rm{prime}}$ (cf. \S 1 Notation). 
We denote the ideal of $F$ corresponding to the  embedding 
$\ell_{\overline{\bb{Q}}}|_K$ 
by $\ell_F$ for each prime number $\ell$ and algebraic number field $F$. 
Assume $F/\bb{Q}$ is Galois extension, 
and $\ell$ splits completely in $F/\bb{Q}$. 
Then, $\mca{I}_F^\ell$ is a free $\bb{Z}[\Gal(F/\bb{Q})]$-module 
generated by $\ell_F$, 
and we identify $\mca{I}_F^\ell$ with $\bb{Z}[\Gal(F/\bb{Q})]$ 
by the isomorphism 
$\xymatrix{\iota\colon \bb{Z}[\Gal(F/\bb{Q})] 
\ar[r]^(0.68){\simeq} & \mca{I}_F^\ell}$ 
defined by 
${x \longmapsto x\cdot \ell_F}$ 
for $x \in \bb{Z}[\Gal(F/\bb{Q})]$.
We also denote the composition map
$\xymatrix{
F^{\times} \ar[r] & \mca{I}_F^\ell \ar[r]^(0.3){\iota^{-1}}& 
\bb{Z}[\Gal(F/\bb{Q})]  
}$
by $(\cdot )_F^\ell$.

\begin{dfn}\label{[]}
We define the $R_{m, N, \chi}$-homomorphism 
\[
\xymatrix{
[\cdot ]_{m,N,\chi}\colon (F_m^{\times}/p^N)_{\chi} 
\ar[r] & (\mca{I}_{F_m}/p^N)_\chi
}
\] 
to be the homomorphism induced by 
$(\cdot )_{F_m}^\ell
\colon F_m^{\times} \longrightarrow \mca{I}_{F_m}$.
Let $\ell\in\mca{S}_N$ be any element. 
We define the $R_{m, N, \chi}$-homomorphism
\[
\xymatrix{
[\cdot ]_{m,N,\chi}^\ell\colon (F_m^{\times}/p^N)_{\chi} 
\ar[r] & R_{m, N, \chi}
=\bb{Z}/p^N[\Gal(F_m/\bb{Q})]
\otimes_{\bb{Z}_p[\Delta]}\mca{O}_{\chi}
}
\] 
to be the homomorphism induced by 
$(\cdot )_{F_m}^\ell\colon F_m^{\times} 
\longrightarrow \bb{Z}[\Gal(F_m/\bb{Q})]$.
\end{dfn}

Second, we will define $\bar{\phi}_{m,N,\chi}^\ell$. 
Let $\ell\in\mca{S}_N$ be any element. 
Since we assume $N \ge m+1$, 
the prime number $\ell$ splits completely in $F_m/\bb{Q}$, 
and we have $F_{m,\lambda}=\bb{Q}_\ell$ 
for any prime ideal $\lambda$ of $F_m$ above $\ell$. 
We regard the groups 
$\bigoplus_{\lambda|\ell}F_{m,\lambda}^{\times}$ 
and $\bigoplus_{\lambda|\ell}H_\ell$ are 
regarded as $\bb{Z}[\Gal(F_m/\bb{Q})]$-modules by the identification 
\begin{align*}
\bigoplus_{\lambda|\ell}F_{m,\lambda}^{\times} 
& = \mca{I}_{F_m}^\ell\otimes_{\bb{Z}} \bb{Q}_\ell^\times, \\ 
\bigoplus_{\lambda|\ell}H_\ell 
& = \mca{I}_{F_m}^\ell\otimes H_\ell,
\end{align*}
respectively. 
Here, we regard $\bb{Q}_\ell^\times$ 
as a $\bb{Z}[\Gal(F_m/\bb{Q})]$-modules 
on which $\Gal(F_m/\bb{Q})$ acts trivially.
We denote by 
\[
\xymatrix{
\phi_{\bb{Q}_\ell}\colon \bb{Q}_\ell^{\times} \ar[r] 
& \Gal\big(\bb{Q}_\ell(\mu_\ell)/\bb{Q}_\ell\big)
=\Gal\big(\bb{Q}(\mu_\ell)/\bb{Q}\big) 
= H_\ell 
}
\] 
the reciprocity map of local class field theory 
defined by 
\[
\phi_{\bb{Q}_\ell}(\ell)
=(\ell_{\bb{Q}(\mu_\ell)},\bb{Q}(\mu_\ell)/\bb{Q}). 
\]
In this paper, we define the homomorphism  
\[
\xymatrix{
\phi_{m}^{\ell}\colon F_m^{\times} \ar[r] & 
\bb{Z}[\Gal(F_m/\bb{Q})]\otimes H_\ell  
}
\] 
to be the composition of the three homomorphisms 
of $\bb{Z}[\Gal( F_m(\mu_n)/\bb{Q} )]$-modules:
\begin{align*}
\mrm{diag}\colon 
& \xymatrix{
F_m^{\times} \ar[r] &  
\bigoplus_{\lambda|\ell}F_{m,\lambda}^{\times},} \\
\oplus\phi_{\bb{Q}_\ell} \colon 
& \xymatrix{
\bigoplus_{\lambda|\ell}F_{m,\lambda}^{\times} 
\ar[r] & \bigoplus_{\lambda|\ell}H_\ell,} \\
\iota_H^{-1} \colon 
& \xymatrix{
\bigoplus_{\lambda|\ell}H_\ell 
\ar[r]^(0.34){\simeq } & \bb{Z}[\Gal( F_m/\bb{Q})]\otimes H_\ell, }
\end{align*}
which are defined  as follows:
\begin{enumerate}
\item the first homomorphism $\mrm{diag}$ is the diagonal inclusion;
\item the second homomorphism $\oplus\phi_{\bb{Q}_\ell}$ 
is the direct sum of the reciprocity maps; 
\item the third isomorphism $\iota_H^{-1}$ 
is the inverse of the isomorphism 
\[
\xymatrix{
\iota_H \colon \bb{Z}[\Gal( F_m/\bb{Q}) ]\otimes H_\ell 
\ar[r]^(0.52){\simeq} 
& \bigoplus_{\lambda|\ell}H_\ell
=\mca{I}_{F_m}^\ell\otimes H_\ell,
}
\]
which is induced by the above isomorphism 
$\xymatrix{\iota\colon \bb{Z}[\Gal(F_m/\bb{Q})] 
\ar[r]^(0.67){\simeq} & \mca{I}_{F_m}^\ell}$
given by 
$x \longmapsto x \cdot \ell_{F_m}$.
\end{enumerate}

\begin{dfn}\label{phi} 
Let $\ell\in\mca{S}_N$ be any element.
We define 
\[
\xymatrix{
\phi_{m,N,\chi}^\ell\colon (F_m^{\times}/p^N)_{\chi} 
\ar[r] & \bb{Z}/p^N[\Gal(F_m/\bb{Q})]_{\chi}\otimes H_\ell
}
\] 
to be the homomorphism of $R_{m,N,\chi}$-modules 
induced by $\phi_{m}^\ell$. 
The choice of a generator $\sigma_\ell$ of $H_\ell$ induces 
the $R_{m,N,\chi}$-homomorphism
\[
\xymatrix{
\bar{\phi}_{m,N,\chi}^\ell\colon 
(F_m^{\times}/p^N)_{\chi} \ar[r] & \bb{Z}[\Gal(F_m/\bb{Q})]_{\chi}
= R_{m,N,\chi}.
}
\]
\end{dfn}

The following formulas 
on Kolyvagin derivatives are well-known.
\begin{prop}
\label{[] and phi}
Let $n$ be an integer contained in $\mca{N}_N$. 
Let $d \in \bb{Z}_{>1}$ be an integer 
dividing $\mf{f}_{K}$ and $a \in \bb{Z}$ 
an integer coprime to $p$.
For simplicity, we denote $\kappa^d_{m,N}(n)$ or 
$\kappa^{1,a}_{m,N}(n)$ by $\kappa^{\bullet }(n)$.
\begin{enumerate}
\item  If $\lambda$ is a prime ideal of $K$ 
not dividing $n$, 
the $\lambda$-component of 
$[\kappa^{\bullet }(n)_\chi]_{m,N,\chi}$ is $0$. 
In particular, if $q \in \mca{S}_N$ is a prime number not dividing $n$, 
we have 
\[
[\kappa^{\bullet }(n)_\chi]_{m,N,\chi}^q = 0.
\]
{\em 
(See \cite{Grei} Lemma 3.6 
and \cite{Ru2} Proposition 2.4.)
}
\item Let $\ell$ be a prime number dividing $n$. 
Then,  
\[
[\kappa^{\bullet }(n)_\chi]_{m,N,\chi}^{\ell}
=\bar{\phi}_{m,N,\chi}^{\ell}(\kappa^{\bullet }(n/\ell)_\chi).
\] 
{\em 
(See \cite{Grei} Lemma 3.6 
and \cite{Ru2} Proposition 2.4.)
}
\item If $n$ is well-ordered, 
then 
\[
\bar{\phi}_{m,N,\chi}^{\ell}
(\kappa^{\bullet }_{m,N}(n)_\chi)=0
\] 
for each prime number $\ell$ dividing $n$. 
{\em (See \cite{MR} Theorem A.4.)}
\end{enumerate}
\end{prop} 

\subsection{}\label{the subsection of x}
In this subsection, 
we will define the Kurihara's elements 
$x_{\nu ,q} \in (F_m^{\times}/p^N)_{\chi}$ 
which become a key of the proof of Theorem \ref{Main theorem}.

We fix circular units 
\[
\eta_m(n):=\prod_{d|\mf{f}_K}\eta_m^{d}(n)^{u_d}\times
\prod_{i=1}^r \eta_m^{1,a_i}(n)^{v_i} \in F_m^\times
\]
for any $m \in \bb{Z}_{\ge 0}$ and 
any $n \in \bb{Z}_{\ge 1}$ 
with $(n,p\mf{f}_K)=1$,
where $r \in \bb{Z}_{>0}$, $u_d$ and $v_i$ 
are elements of $\bb{Z}[\Gal(F_m/\bb{Q})]$ 
for each positive integers $d$ and $i$ 
with $d|\mf{f}_K$ and $1 \le i \le r$, 
and $a_1, \cdots, a_r$ are integers prime to $p$. 
Here, we assume that $r$, $a_1,...,a_r$, 
$u_d$'s and $v_i$'s are 
constant independent of $m$ and $n$. 
Then, we put 
\[
\kappa_{m,N}(\eta;n)
:= \prod_{d|\mf{f}_K}\kappa_{m,N}^d(n)^{u_d}
\times \prod_{i=1}^r \kappa_{m,N}^{a_i}(n)^{v_i} 
\in F_m^\times/p^N.
\]
Note that $\kappa_{m,N}(\eta,\chi)$ is 
the Kolyvagin derivative of the Euler system 
\[
\eta=
\{\eta_m(n)_\chi \in 
(F_m(\mu_n)^\times \otimes_{\bb{Z}} \bb{Z}_p)_\chi \}_{m,n}
\]
of circular units. 
If no confusion arises, we write 
$\kappa(n)=\kappa_{m,N}(\eta;n)$
for simplicity.

\begin{dfn}
Let $q\nu =q\prod_{i=1}^r \ell_i \in\mca{N}_N$, 
where $q, \ell_1,\dots,\ell_r$ are distinct prime numbers. 
For any $e \in \bb{Z}_{>0}$ dividing $\nu$, 
we define $\tilde{\kappa }_{\{ e,q \}} 
\in (F_m^\times/p^N)_{\chi} \otimes(\bigotimes_{\ell |d} H_\ell)$ by 
\[
\tilde{\kappa}_{\{ e,q \}} := 
\kappa (qe)_\chi\otimes\big( \bigotimes_{\ell |e}\sigma_\ell \big).
\] 
\end{dfn}

Let $q\nu \in\mca{N}_N$ and assume $q\nu$ is {\em well-ordered}. 
Assume that for each prime number $\ell$ dividing $\nu$, 
an element $w_\ell \in R_{m,N, \chi} \otimes H_\ell$ is given. 
Then, we have an element $\bar{w}_\ell \in R_{m, N, \chi}$ 
such that $w_\ell = \bar{w}_\ell\otimes \sigma_\ell$. 
Note that we will take $\{ w_\ell \}_{\ell | \nu}$ explicitly later, 
but here, we take arbitrary one.
For any $e \in \bb{Z}_{>0}$ dividing $\nu$, we define 
$$w_e:=\bigotimes_{\ell |e} w_\ell \in R_{m,N, \chi} 
\otimes \big (\bigotimes_{\ell |e} H_\ell \big).$$
We also define the element $\bar{w}_e \in R_{m,N, \chi}$ by 
$w_e = \bar{w}_e \otimes\big( \bigotimes_{\ell |e}\sigma_\ell \big)$.

Note that we write the group law of 
$(F_m^{\times}/p^N)_{\chi} \otimes 
\big( \bigotimes_{\ell |e} H_\ell \big)$ multiplicatively. 
\begin{dfn}
\label{x}
We define the element  $\tilde{x}_{\nu, q} $ by 
\[
\tilde{x}_{\nu ,q} := 
\prod_{e|\nu}w_e\otimes \tilde{\kappa}_{\{\nu/e,q\}} 
\in (F_m^{\times}/p^N)_{\chi} \otimes 
\big( \bigotimes_{\ell |e} H_\ell \big).
\]
Note that we naturally identify 
the $R_{m, N, \chi}$-module 
$(F_m^{\times}/p^N)_{\chi} \otimes 
\big( \bigotimes_{\ell |e} H_\ell \big)$ 
with 
\[
R_{m, N, \chi} \otimes \big( \bigotimes_{\ell |e} H_\ell \big)
\otimes_{R_{m, N, \chi}}(F_m^{\times}/p^N)_{\chi}.
\]
The element $x_{\nu,q} \in (F_m^{\times}/p^N)_{\chi}$ 
is defined by 
$\tilde{x}_{\nu,q}=x_{\nu,q}\otimes
\big( \bigotimes_{\ell | \nu}\sigma_\ell \big)$.
\end{dfn}

The following formulas follows from 
Proposition \ref{[] and  phi} straightforward.

\begin{prop}[cf. \cite{Ku} Proposition 6.1]\label{[] and phi and x}
Let $q\nu \in  \mca{N}_N$ and we assume that $q\nu$ is well-ordered.
\begin{enumerate}
\item If $\lambda$ is a prime ideal of $K$ not dividing $n$, 
the $\lambda$-component of $[x_{\nu,q}]_{m,N,\chi}$ is $0$. 
In particular, if $s$ is a prime number not dividing $q\nu$, 
we have 
\[
[x_{\nu,q}]_{m,N,\chi}^s = 0.
\]
\item Let $\ell$ be a prime number dividing $\nu$. Then, we have 
\[
[x_{\nu,q}]_{m,N,\chi}^{\ell}
=\bar{\phi}_{m,N,\chi}^{\ell}(x_{\nu/\ell,q}).
\]
\item Let $\ell$ be a prime number dividing $\nu$. 
Then, we have 
\[
\bar{\phi}_{m,N,\chi}^{\ell}(x_{\nu,q})
= \bar{w}_\ell\bar{\phi}_{m,N,\chi}^{\ell}(x_{\nu/\ell,q}).
\]
\end{enumerate}
\end{prop}

\section{Chebotarev Density Theorem}\label{the section of Chebotarev Density Theorem}
Recall that we fix a family of embeddings 
$\{ \xymatrix{\ell_{\overline{\bb{Q}}}\colon
\overline{\bb{Q}}\ar@{^{(}->}[r] & 
\overline{\bb{Q}}_\ell} \}_{\ell: \rm{prime}}$ 
satisfying the technical condition (Chb) 
for families of embeddings as follows:
\begin{itemize}
\item[(Chb)]  {\em For any subfield $F \subset \overline{\bb{Q}}$ 
which is a finite Galois extension of $\bb{Q}$ and any element 
$\sigma \in \Gal(F/\bb{Q})$, 
there exist infinitely many prime numbers 
$\ell$ such that $\ell$ is unramified in $F/\bb{Q}$ 
and $(\ell_F, F/\bb{Q})= \sigma$, 
where $\ell_F$ is the prime ideal of $F$ 
corresponding to the  embedding $\ell_{\overline{\bb{Q}}}|_F$.}
\end{itemize}
Note that the existence of such a family of embeddings 
follows from the Chebotarev density theorem. 
We need the condition (Chb) 
in the proof of Proposition \ref{chebo appli}.

Let $\chi \in \widehat{\Delta}$ be any element. 
Recall that we denote the restriction of $\chi$ 
to $\Delta_0$ by $\chi_0$, and 
we put 
\begin{align*}
a_\chi &:=
\begin{cases}
0 & \text{if $\chi(p)\ne 1$; } \\
2 & \text{if $\chi(p)= 1$.} 
\end{cases}
\end{align*}
Here, we shall prove Proposition \ref{chebo appli}, 
which is the key of induction argument in the proof of 
Theorem \ref{Main theorem}.
This proposition corresponds to Lemma 9.1 in \cite{Ku}.

\begin{prop}\label{chebo appli}
Let $\chi \in \widehat\Delta $ be a non-trivial character. 
Assume $qn=q\prod_{i=1}^r\ell_i \in \mca{N}_N$, 
where $q,\ell_1,\dots,\ell_r$ are prime numbers. 
Suppose the following are given:
\begin{itemize}
\item a finite $R_{m, N, \chi}$-submodule $W$ 
of $(F_m^{\times}/p^N)_{\chi}$;
\item a $R_{m, N, \chi}$-homomorphism 
$\xymatrix{\psi  \colon W \ar[r] & R_{m, N, \chi}.}$
\end{itemize}
Then, there exist infinitely many $q' \in \mca{S}_N$
which splits completely in $F_m(\mu_n)/\bb{Q}$, and  
satisfies all of the following properties.
\begin{enumerate}
\item the class of $q'_{F_m}$ in $A_{m,\chi}$ 
coincides with the class of 
$\left| \Delta_p \right| \cdot q_{F_m}$.
{\em (}Recall that we write the group law of 
$\mca{I}_F$ additively.{\em)}
\item there exists an element 
$z \in (F_m^\times\otimes \bb{Z}_p)_{\chi}$ 
such that 
\begin{itemize}
\item $(z)_{_m,\chi} =
(q'_{F_m}- \left| \Delta_p \right|\cdot q_{F_m})_\chi 
\in \big(\mca{I}_{F_m}\otimes\bb{Z}_p \big)_\chi, $
\item $\phi_{m,N,\chi}^{\ell_i}(z)=0$
for each $i=1,\dots,r$.
\end{itemize}
\item the group $W$ is contained 
in the kernel of $[\cdot ]_{m,N,\chi}^{q'}$, 
and 
\[
\psi(x)= \left| \Delta_p \right|^2 
\bar{\phi}_{m,N,\chi}^{q'}(x)
\]
for any $x \in W$.
\end{enumerate}
\end{prop}

\begin{proof}
Proof of this proposition is essentially the same as 
those of Lemma 9.1 in \cite{Ku} 
though we have to treat carefully 
when $\left| \Delta_p \right|\ne 1$.
We shall prove this proposition in four steps.

{\em The first step}.
Let $v$ be a prime ideal of $F_m$. 
We denote the ring of integers of the completion $F_{m,v}$ of $F_m$ 
at $v$ by $\mca{O}_{F_{m,v}}$, 
and define the subgroup $\mca{O}_{F_{m,v}}^1$ 
of $\mca{O}_{F_{m,v}}^\times$ by 
\[
\mca{O}_{F_{m,v}}^1:=\{x \ |\ x \equiv 1 \pmod{\mf{m}_v} \},
\]
where $\mf{v}$ is the maximal ideal of $\mca{O}_{F_{m,v}}$.
We denote the residue field of $F_m$ at $v$ by $k(v)$.
Let $F_m\{n \}$ be the maximal abelian $p$-extension of $F_m$ 
unramified outside $n$. 
By global class field theory, we have the isomorphism
\[
\xymatrix{
\displaystyle \frac{(\prod_{v|n}F_{m,v}^\times
/\mca{O}_{F_{m,v}}^1 )\times 
(\bigoplus_{u \nmid n}F_{m,u}^\times/\mca{O}_{F_{m,u}}^\times)}
{ F_m^\times}\otimes\bb{Z}_p \ar[r]^(0.68){\simeq} & 
\Gal\big({F_m}\{n\}/{F_m}\big),
}
\]
where $u$ runs all finite places outside $n$.
This isomorphism induces the homomorphism
\[
\xymatrix{
\iota\colon\bigoplus_{v|n}k(v)^\times \otimes \bb{Z}_p  
\ar[r]& \Gal\big(F_m\{n\}/F_m\big). 
}
\]
Taking the $\chi$-quotients, we obtain the homomorphism 
\[
\xymatrix{
\iota_\chi\colon\big(\bigoplus_{v|n} k(v)^\times 
\otimes \bb{Z}_p \big)_{\chi}  \ar[r] & 
\Gal\big(F_m\{n\}/F_m\big)_{\chi}
}
\]
of $\bb{Z}_p[\Gal(F_m/\bb{Q})]_\chi$-modules.
We denote by $F_m\{n\}_\chi$ the intermediate field of $F_m\{n\}/F_m$ 
with $\Gal\big(F_m\{n\}_\chi/F_m\big)=\Gal(F_m\{n\}/F_m)_\chi$.
Recall $n=\prod_{i=1}^r\ell_i$, 
and all prime divisors $\ell_i$ of $n$ split completely in $F_m/\bb{Q}$. 
By local Artin maps, we obtain the isomorphism
\[
\xymatrix{
\big(\bigoplus_{v|n}k(v)^\times 
\otimes \bb{Z}_p \big)_{\chi}  
\ar[r]^(0.37){\simeq} & 
\bigoplus_{\i=1}^r\big(\bb{Z}_p[\Gal(F_m/\bb{Q})]_{\chi} 
\cdot \ell_{i,F_m} \big)\otimes H_{\ell_i}, 
}
\]
and we identify them by this isomorphism. 

Let $L:=F_m\{n\}_\chi K(\mu_{np^{N}})$ 
be the composition field. 
Note that the cokernel of the natural homomorphism 
\[
\xymatrix{
\Gal(L/F_m) \ar[r] & \Gal(L/F_m)_\chi \times \Gal(L/F_m)_1
}
\] 
is annihilated by $\left| \Delta_p \right|$ since $\chi \ne 1$.
Since the subgroup $\Delta$ of 
$\Gal(F_m/\bb{Q})$ acts on $\Gal\big(F_m\{n\}_{\chi}/F_m\big)$ 
via $\chi$,
$\Gal\big(F_m\{n\}_{\chi}/F_m \big)$ 
is a quotient of $\Gal(L/F_m)_\chi$.
On the other hand, 
since $\Delta$ of $\Gal(F_m/\bb{Q})$ acts on 
$\Gal\big(K(\mu_{np^N})/F_m \big)$
via the trivial character,  
$\Gal\big(K(\mu_{np^N})/F_m\big)$ is a quotient of $\Gal(L/F_m)_1$.
Then, the cokernel of the natural homomorphism 
\[
\xymatrix{
\Gal(L/F_m) \ar[r] & \Gal\big(F_m\{n\}_{\chi}/F_m\big) 
\times \Gal\big(K(\mu_{np^N})/F_m \big)
}
\] 
is annihilated by $\left| \Delta_p \right|$.
Then, we take an element 
$\sigma \in \Gal\big(L/K(\mu_{np^N})\big)$ 
such that 
\[
\sigma |_{F_m\{n\}_\chi}=
(q_{F_m\{n\}_\chi}, F_m\{n\}_\chi/F_m )^{\left| \Delta_p \right|}.
\]

{\em The second step.}
We fix a finite $R_{m,\chi}$-submodule $\mca{W}$ of 
$F_m^\times/p^N$ whose image in $(F_m^\times/p^N)_\chi$ is $W$.
By Corollary \ref{kernel}, 
we take a homomorphism $\tilde{\psi} \in \Hom_{R_m,N}(\mca{W},R_{m,N})$
which makes the diagram 
\[
\xymatrix{\mca{W} \ar[rr]^(0.51){\tilde{\psi}} 
\ar[d] & &  R_{m,N} \ar[d] \\
W \ar[rr]^{\left| \Delta_p \right|^2\psi} & & R_{m,N,\chi}}
\]
commute.
We define a projection 
$\mrm{pr}\colon R_{m,N} \longrightarrow \bb{Z}/p^N\bb{Z}$ by 
\[
\sum_{g \in \Gal(F_m/\bb{Q})}a_g g \longmapsto  a_1,
\]
where $a_g \in \bb{Z}/p^N\bb{Z}$ for all 
$g \in \Gal(F_m/\bb{Q})$, 
and $1 \in \Gal(F_m/\bb{Q})$ is the unit. 
This projection induces an {\em isomorphism}
\[
\xymatrix{
P\colon \Hom_{R_{m,N}}(\mca{W},R_{m,N}) \ar[r] &  
\Hom(\mca{W},\bb{Z}/p^N\bb{Z})
}
\] given by
$f \longmapsto \mrm{pr}\circ f $. 
We define a homomorphism
\[
\xymatrix{
(P\tilde{\psi})_1 \colon 
\mca{W} \ar[r] & \mu_{p^N}
}
\] 
of abelian groups
by $x \longmapsto {\zeta _{p^{N}}}^{P(\tilde{\psi})(x)}$.

We denote the image of $\mca{W}$ in $K(\mu_{p^N})^\times/p^N$ by $\mca{W}'$. 
Let $M$ be the extension field of 
$K(\mu_{p^N})$ generated by all $p^N$-th roots of 
elements of $K(\mu_{p^N})^\times$ 
whose image in $K(\mu_{p^N})^\times/p^N$ is contained in $\mca{W}'$. 
So, the Kummer pairing 
induces the isomorphism 
\[
\xymatrix{
\mrm{Kum}\colon \Gal\big(M/K(\mu_{p^N})\big) 
\ar[r]^(0.57){\simeq} & \Hom(\mca{W'},\mu_{p^N}).
}
\]
Note that the complex conjugation $c$ acts on
$H^1\big(K(\mu_{p^N})/F_m, \mu_{p^N}\big)$
by $-1$, and $c$ acts on $F_m^\times/p^N$ trivially. 
So, the group $H^1\big(K(\mu_{p^N})/F_m, \mu_{p^N}\big)$
vanishes, 
and the natural homomorphism 
$F_m^\times/p^N \longrightarrow K(\mu_{p^N})^\times/p^N$
is an injection. 
This implies that 
the natural homomorphism 
\[
\xymatrix{i \colon \mca{W} \ar@{->>}[r] & \mca{W}'}
\]
is an isomorphism. 
We take $\lambda \in \Gal(M/K(\mu_{p^N}))$ such that 
\[
i^* \circ \mrm{Kum}(\lambda)={(P\tilde{\psi})_1}.
\]

{\em The third step}. 
Recall that $LM/K(\mu_{p^N})$ is an abelian $p$-extension, 
so we regard $\Gal\big( LM/K(\mu_{p^N}) \big)$ 
as a $\bb{Z}_p[\Gal \big(K(\mu_{p^N})/\bb{Q} \big)]$-module. 
We have the natural isomorphism:
\[
\xymatrix{
\Gal \big(LM/K(\mu_{p^N}) \big) \ar@{^{(}->}[r] & 
\Gal \big( LM/K(\mu_{p^N}) \big)_{+}
\times\Gal\big(LM/K(\mu_{p^N}) \big)_{-},
}
\]
where $\Gal \big( LM/K(\mu_{p^N})\big)_{+}$  
(resp.\ $\Gal \big( LM/K(\mu_{p^N})\big)_{-}$)  denotes 
the maximal quotient of 
$\Gal \big(LM/K(\mu_{p^N})\big)$ 
on which the complex conjugation 
$c$ acts trivially (resp.\ by $-1$).  
We put 
$\tilde{\Delta}:=\Gal\big(
\bb{Q}(\mu_p)/\bb{Q} \big)$, 
and regard $\tilde{\Delta}$ as 
a subgroup of $\Gal\big(
\bb{Q}(\mu_{p^N})/\bb{Q} \big)$. 
Note that on the one hand, 
$\Gal \big(L/K(\mu_{p^N})\big)$ 
is a quotient of $\Gal \big(LM/K(\mu_{p^N-})\big)_{+}$ since 
$\widetilde{\Delta}$ 
acts trivially on 
$\Gal \big(L/K(\mu_{p^N})\big)$.
On the other hand, The complex conjugation $c$ acts 
on $\Gal(M/K(\mu_{p^N})$ by $-1$ 
since $\widetilde{\Delta}$ 
acts on $\Gal \big(M/K(\mu_{p^N})\big)$ 
via the character $\chi^{-1}\omega$. 
This implies $\Gal(M/K(\mu_{p^N})\big)$ is 
a quotient of $\Gal\big(LM/K(\mu_{p^N})\big)_{-}$.
Therefore, the natural homomorphism 
\[
\xymatrix{
\Gal \big(LM/K(\mu_{p^N})\big) 
\ar@{^{(}->}[r] & \Gal \big(L/K(\mu_{p^N})\big)
\times\Gal \big(M/K(\mu_{p^N})\big)
}
\] 
is a surjection.
By the condition (Chb), 
there exists infinitely many prime numbers $q'$ such that 
\[
\begin{cases}
(q'_L,L/K(\mu_{p^N}))=\sigma \\
(q'_M,M/K(\mu_{p^N}))={\lambda}^{-1}. 
\end{cases}
\]

{\em The fourth step}.
Here, we prove that each of such $q'$ unramified in $L/\bb{Q}$ satisfies conditions (1)-(3) of Proposition \ref{chebo appli}.
First, we show $q'$ satisfies conditions (1) and (2). 
Let $\alpha=(\alpha_v)_v \in \bb{A}_{F_m}^\times$ 
be an idele whose $q'_{F_m}$-component 
is a uniformizer of $F_{m,q'_{F_m}}$, 
and other components are $1$. 
Let $\beta=(\beta_v)_v \in \bb{A}_{F_m}^\times$ be 
an element as follows.
The components above $q$ are given by
\[
\left| \Delta_p \right|\cdot (x_v)_{v|q} 
\in \prod_{v|q}F_{m,v}^\times,
\]
where $x_{q_{F_m}}$ is 
a uniformizer of $F_{m,q_{F_m}}$, 
and $x_v=1$ otherwise. 
For all places $v$ of ${F_m}$ not above $q$, 
we put $\beta_v=1$. 

By definition, ideles $\alpha$ and 
$\beta$ have the same image in the group 
\[
\bigg( \frac{(\prod_{v|n}F_{m,v}^\times
/\mca{O}_{F_{m,v}}^1 )
\times (\bigoplus_{u \nmid n}F_{m,u}^\times
/\mca{O}_{F_{m,u}}^\times)}
{ F_m^\times}\otimes\bb{Z}_p \bigg)_{\chi} 
\simeq \Gal\big(F_m\{n\}_\chi/F_m\big).
\]
This implies there exist 
$z \in (F_m^\times\otimes\bb{Z}_p)_\chi$ 
such that
\[
\alpha=z\beta \ \ \text{in} \  
\bigg( \big((\prod_{v|n}F_{m,v}^\times
/\mca{O}_{F_{m,v}}^1 )
\times (\bigoplus_{u\nmid n}F_{m,u}^\times
/\mca{O}_{F_{m,u}}^\times)\big)\otimes\bb{Z}_p\bigg)_\chi. 
\]
Hence, we have 
$(z)_{F_{m,\chi}}=
(q'_{F_m}-\left| 
\Delta_p \right|\cdot q_{F_m})_\chi,$
and $\phi_{m,N,\chi}^{\ell_i}(z)=0$ for each $i=1,\dots,r$. 
Obviously, the prime number $q'$ 
satisfies conditions (1) and (2).

Next, we shall prove $q'$ satisfies condition (3).
Recall that we have 
\[
(q'_M,M/K(\mu_{p^N}))=\lambda^{-1}. 
\]
So, by definition of $\lambda$, 
we have
\[
{\zeta _{p^{N}}}^{P(\tilde{\psi})(x)}
=(x^{1/{p^N}})^{1-\Frob_{q'}},
\]
for any $x \in W$, 
where we put
\[
\Frob_{q'}:=(q'_M,M/K(\mu_{p^N})) 
\in \Gal \big(M/K(\mu_{p^N}) \big),
\] 
and $x^{1/{p^N}}\in L$ is a $p^N$-th root of $x$. 
Since $q'$ is unramified in $M/\bb{Q}$, the group $W$ 
is contained in the kernel of $[\cdot ]_{m,N,\chi}^{q'}$.
So, we obtain 
\[
{\zeta _{p^{N}}}^{P(\tilde{\psi})(x)} \equiv 
x^{(q'-1)/p^N} \pmod{q'}.
\]

We can take the unique intermediate field $F$ 
of $F_m(\mu_{q'})/F_m$ whose degree over $F_m$ 
is $p^N$ since $q'\equiv 1 \pmod{p^N}$. 
We denote the image of $\sigma_{q'} \in H_{q'}$ in 
\[ 
 H_{q'}\otimes_{\bb{Z}}\bb{Z}_p
=\Gal(F/F_m)
\]
by $\bar{\sigma}_{q'}$.
Let $\pi$ be a uniformizer of $F_{q'_F}$. 
By definition of $\sigma_{q'}$, we have 
\[
\pi^{\bar{\sigma}_{q'}-1} \equiv \zeta_{p^N} 
\pmod{\mf{m}_{q'}},
\]
where $\mf{m}_{q'}$ is the maximal ideal of $F_{{q'}_F}$.
Recall that $W$ is contained 
in the kernel of $[\cdot ]_{m,N,\chi}^{q'}$. 
We put 
\[
\phi(x):=\bar{\sigma}_{q'}^{P(\bar{\phi}_{m,N}^{q'})(x)} 
\in \Gal(F/F_m).
\]
By \cite{Se} Chapter XIV Proposition 6, 
we have
\[
{\zeta _{p^{N}}}^{P(\bar{\phi}_{m,N}^{q'})(x)} 
= \pi^{\phi(x)-1}
\equiv x^{(1-q')/p^N} 
\pmod{\mf{m}_{q'}}
\]
for all $x \in W$. 
Hence, we obtain $${\zeta _{p^{N}}}^{P(\tilde{\psi})(x)}
= {\zeta_{p^{N}}}^{P(\bar{\phi}_{m,N}^{q'})(x)}$$
for all $x \in W$.
Therefore $q'$ satisfies condition (3) of Proposition \ref{chebo appli}.
\end{proof}

\section{Euler system argument via Kurihara's elements}\label{Kurihara's Euler system argument}

In this section, we prove the assertions of 
Theorem \ref{Main theorem, Rough} 
on the upper bounds of the higher Fitting ideals 
by using Kurihara's Euler system arguments 
established in \cite{Ku}. 
Let us state the assertions of Theorem \ref{Main theorem},
which is the goal of this section. 
Let the ideals $I_{P_{\chi}^E}$ and 
$J_{P_{\chi}^E}$ of $\Lambda_\chi$ be 
as in Proposition \ref{proj of E}, 
the ideal $I_C$ as in \S \ref{subsection of 0-th cyclotomic ideals}, 
and the $\Lambda_\chi$-submodule $Y$ of $X_\chi$ as in 
Proposition \ref{class group}. 
We denote the ideal of $\Lambda_\chi$ generated by $i$-th power of 
elements of $\ann_{\Lambda_\chi}(Y/(\gamma-1)X)_\chi$ by $I_i$
for each $i \in \bb{Z}_{i \ge 0}$.
The goal of this section is the following theorem. 

\begin{thm}\label{Main theorem}
Let $\chi  \in \widehat \Delta$ be a non-trivial character.
Then, we have
\begin{itemize}
\item[(i)] 
\ $\mf{C}_{0,\chi} \subseteq \Fitt_{\Lambda_\chi,0}(X'_\chi)$.
\item[(ii-$0$)]
\ $(\gamma-1)^{a_\chi}\left| \Delta_p \right|^4 
I_C I_{P_{\chi}^E}J_{P_{\chi}^E}  
\Fitt_{\Lambda_\chi,i}(X'_\chi) \subseteq \mf{C}_{i,\chi}$.
\item[(ii-$i$)]\ 
$(\gamma-1)^{a_\chi}\left| \Delta_p \right|^{6+4i} 
I_i I_C I_{P_{\chi}^E}J_{P_{\chi}^E}  
\Fitt_{\Lambda_\chi,i}(X'_\chi) \subseteq \mf{C}_{i,\chi}$
for any $i \in \bb{Z}_{\ge 1}$.
\end{itemize}
In particular, we have 
\begin{itemize}
\item[(i)] 
$\mf{C}_{0,\chi} \prec \Fitt_{\Lambda_\chi,0}(X_\chi)$.
\item[(ii-$0$)]
\ $(\gamma-1)^{a_\chi}\left| \Delta_p \right|^4  
\Fitt_{\Lambda_\chi,i}(X'_\chi) \subseteq \mf{C}_{i,\chi}$.
\item[(ii-$i$)]
$(\gamma -1)^{a_\chi} \left| \Delta_p \right|^{6+4i} 
\Fitt_{\Lambda_\chi,i}(X_\chi)\prec \mf{C}_{i,\chi}$
for for any $i \in \bb{Z}_{\ge 1}$.
\end{itemize}
\end{thm}

The assertions of 
Theorem \ref{Main theorem, Rough} 
on the upper bounds of the higher Fitting ideals are special cases 
of Theorem \ref{Main theorem}.

\begin{cor}\label{half of Main theorem, Rough}
Assume that the extension degree of $K/\bb{Q}$ is prime to $p$, 
and $\chi \in \widehat{\Delta}$ is a character 
satisfying $\chi(p)\ne 1$.
Then, we have the following. 
\begin{itemize}
\item[(i)] 
$\mf{C}_{0,\chi} \subseteq \Fitt_{\Lambda_\chi,0}(X'_\chi)$.
\item[(ii)]
$\ann_{\Lambda_\chi}(X_{\chi,\mrm{fin}})
\Fitt_{\Lambda_\chi,i}(X_\chi) \subseteq \mf{C}_{i,\chi}$
for any $i \in \bb{Z}_{\ge 0}$.
\end{itemize}
\end{cor}
\begin{proof}
Here, let us deduce the corollary 
from Theorem \ref{Main theorem}. 
Under the assumption of Corollary \ref{half of Main theorem, Rough}, 
we have $\left| \Delta_p \right|=1$, and 
we can take $I_i= I_C=I_{P_{\chi}^E}=\Lambda_\chi$ 
and $J_{P_{\chi}^E}=\ann_{\Lambda_\chi}(X_{\chi,\mrm{fin}})$ 
by Proposition \ref{proj of E, refined}, 
Proposition \ref{class group, refined} 
and Proposition \ref{circular units, free}.
The corollary follows from these results. 
\end{proof}

\subsection{}
We spend this subsection on the setting of notations.
We assume that $\chi \in \widehat{\Delta}$ is non-trivial. 
Let $\mf{m}_\chi$ be the maximal ideal of $\Lambda_\chi$. 
Recall that we denote by 
$X_{\chi,\mrm{fin}}$ the maximal pseudo-null submodule of $X_{\chi}$, 
and put $X'_\chi:=X_\chi/X_{\chi,\mrm{fin}}$.
Since $X'_{\chi}$ has no non-trivial pseudo-null submodules, 
we have an exact sequence
\begin{align}\label{su}
\xymatrix{0 \ar[r] & \Lambda_\chi^h \ar[r]^{f} & \Lambda_\chi^h 
\ar[r]^{g} & {X'_{\chi}} \ar[r] & 0,} 
\end{align}
by Lemma \ref{no pn}. 
Let $M$ be the matrix corresponding to $f$ with respect to 
the standard basis $( \mb{e}_i )_{i=1}^h$ of $\Lambda_\chi^h$. 
We may assume that all entries of $M$ are contained 
in $\mf{m}_\chi$. 
In particular, we have 
\begin{equation}\label{disticct condition}
\mb{e}_i -a \mb{e}_j \notin \Ker g
\end{equation}
for all $i, j \in \bb{Z}$ 
with $1 \le i\ne j \le h$ and all $a \in \Lambda_\chi$.

Let $\{ m_1,...,m_h \}$ and $\{ n_1,...,n_h\}$ 
be permutations of $\{1,...,h\}$, and let
$i$ be an integer satisfying $1 \le i \le h-1$. 
Let us consider 
the matrix $M_i$ which is obtained from $M$ 
by eliminating the $n_j$-th rows ($j=1,...,i$) 
and the $m_k$-th columns ($k=1,...,i$). 
If $\det (M_i)=0$, this is trivial, so we assume that $\det (M_i) \ne 0$. 
If necessary, we permute $\{ m_1,...,m_i \}$, 
and assume $\det (M_r) \ne 0$ 
for all integers $r$ satisfying $ 0 \le r \le i$.

We fix a sequence $\{N_m \}_{m \in \bb{Z}_{\ge 0}}$ 
of positive integers satisfying $N_{m+1}> N_m > m+1$ 
and $p^{N_m} > \left| A_m \right|$ 
for any $m \in \bb{Z}_{\ge 0} $. 
First, we fix a sufficiently large integer $m$. 
For simplicity, we put $F:=F_m$, $R:=\bb{Z}_p[\Gal(F_m/\bb{Q})]_\chi$,  
$N:=N_m$ and $R_N:=R_{m,N,\chi}=\bb{Z}/p^N[\Gal(F_m/\bb{Q})]_\chi$. 
From the exact sequence (\ref{su}), we obtain the exact sequence 
\[
\xymatrix{
0 \ar[r] & R^h \ar[r]^{\bar{f}} & R^h \ar[r]^{\bar{g}} & 
{X'_{\chi,\Gamma_m}} \ar[r] & 0,
}
\]
by taking the $\Gamma_m$-coinvariants.
Let $A_{m,\chi,\mrm{fin}}$ be the image of 
$X_{\chi,\mrm{fin}}$ in $A_{m,\chi}$ by the natural homomorphism 
\[
\xymatrix{
X_{\chi} \ar@{->>}[r] & X_{\Gamma_m,\chi} \ar@{->>}[r] &  A_m.
}
\] 
We put 
$A'_{m,\chi}:= A_{m,\chi}/A_{\mrm{fin},\chi}$.
The image of $\mb{e}_r$ in $R^h$ is denoted by $\mb{e}_r^{(m)}$. 
For each $i \in \bb{Z}$ with $1 \le i \le h$, 
we denote by $\mb{c}_r^{(m)} \in {A'_{m,\chi}}$ 
the image of $\mb{e}_r^{(m)}$ by the homomorphism 
\[
\xymatrix{R^h \ar@{->>}[r]^(0.43){\bar{g}} & 
{X'_{\chi,\Gamma_m}} \ar@{->>}[r] & A'_m.}
\]
We fix a lift 
$\tilde {\mb{c}}_r^{(m)} 
\in A_{m,\chi}$ of $\mb{c}_r^{(m)}$. 
The condition (\ref{disticct condition}) and 
Proposition \ref{class group} imply that
if necessary, we replace $m$ with larger one, and
we may assume 
$\tilde {\mb{c}}_r^{(m)}\ne \tilde {\mb{c}}_s^{(m)}$
for any  $r,s \in \bb{Z}$ with $1 \le r, s \le h$
and $r \ne s$.  
If the extension degree of $K/\bb{Q}$ is divisible by $p$, 
then we additionally assume that 
$\tilde {\mb{c}}_r^{(m)}\ne 
\left| \Delta_p \right| \cdot  \tilde {\mb{c}}_s^{(m)}$
for any $r,s \in \bb{Z}$ satisfying $1 \le r,s \le h$.
We define 
\begin{align*}
P_r& :=\{ \ell \in \mca{S}_N \ | \ [\ell_{F}]_\chi 
= \tilde {\mb{c}}_r^{(m)} \}; \\
P_r'& :=\{ \ell' \in \mca{S}_N \ | \ [\ell'_{F}]_\chi 
= \left| \Delta_p \right| \cdot \tilde {\mb{c}}_r^{(m)} \},
\end{align*}
where $[\ell_{F}]_\chi$ is the ideal class of $\ell_{F}$ in $A_{m,\chi}$. 
By the condition (Chb), 
note that $P_r$ and $P'_r$ are not empty for all $r$.
We define a set $P$ of prime numbers by the union
\[
P:=\coprod_{r=1}^i(P_r \cup P_r'),
\]
and we denote by $P_F$ the set of all prime ideals of 
$\mca{O}_F$ above primes contained in  $P$. 

Let $J$ be the subgroup of $\mca{I}_{F}$ generated by $P_{F}$, and 
$\mca{J}$ the image of $(J\otimes\bb{Z}_p)_{\chi}$ 
in $(\mca{I}_{F}\otimes\bb{Z}_p)_\chi$. 
We denote by $\mca{F}$  
the inverse image of $\mca{J}$ by the homomorphism  
\[
\xymatrix{
(\cdot )_{F}\colon ({F}^{\times} \otimes \bb{Z}_p)_\chi 
\ar[r] &  (\mca{I}_{F}\otimes \bb{Z}_p)_{\chi}
}.
\] 
We define a surjective homomorphism 
\[
\xymatrix{
\alpha\colon \mca{J} \ar[r] & R^h
}
\] 
by $\ell_{F} \mapsto \mb{e}_r$ 
(resp.\ $\ell'_{F} \mapsto \left| \Delta_p \right|\cdot \mb{e}_r$ ) 
for each $\ell \in P_r $ (resp.\ $\ell \in P_r' $) 
and $r $ with $1 \le r \le h$. 
We define 
\[
\xymatrix{
\alpha_r := 
\mrm{pr}_r \circ \alpha\colon \mca{J} \ar[r]^(0.76){\alpha} & 
R^h  \ar[r]^{\mrm{pr}_r} & R
}
\]
to be the composition of $\alpha$ and the $r$-th projection.
We consider the following diagram  
\[
\xymatrix{
& \mca{F} \ar[rr]^{(\cdot)_{F, \chi}} &
& \mca{J} \ar@{->>}[rr] \ar[d]^\alpha &  & A'_{m,\chi}  &  \\
0 \ar[r] & R^h \ar[rr]^{\bar{f}} & & R^h \ar[rr]^{\bar{g}} & 
& X'_{\chi,\Gamma_m} \ar[u]_{\iota_A} \ar[r] & 0, \\
}
\]
where $\iota_A$ is induced by the canonical homomorphism. 
We fix a non-zero element 
$\varepsilon \in \ann_{\Lambda_\chi}(Y/(\gamma-1)X)$.
By Lemma \ref{class group}, we can define the homomorphism 
$\beta\colon \mca{F} \longrightarrow R^h$ 
to make the diagram
\begin{align}\label{diagram2}
\xymatrix{
& \mca{F} \ar[rr]^{(\cdot)_{F, \chi}} \ar@{-->}[d]^\beta 
& & \mca{J} \ar[rr]^{\varepsilon \cdot \pi'_A} 
\ar[d]^{\varepsilon \cdot \alpha} 
& &  A'_{m,\chi}  &  \\
0 \ar[r] & R^h \ar[rr]^{\bar{f}} 
& & R^h \ar[rr]^{\bar{g}} & &  
X'_{\chi,\Gamma_m} \ar[u]_{\iota_A} \ar[r] & 0
}
\end{align} 
commute, where $\pi'_A$ is the natural homomorphism.
Note that since the second row of the diagram is exact, 
$\beta$ is well-defined
We define 
\[
\xymatrix{
\beta_r := \mrm{pr}_r \circ \beta\colon \mca{F} 
\ar[r]^(0.71){\beta} 
& R^h \ar[r]^{\mrm{pr}_r} & R
}
\]
to be the composition of $\beta$ and the $r$-th projection.
We consider the diagram (\ref{diagram2}) 
by taking $( -  \otimes \bb{Z}/p\bb{Z})$. 

We regard $({F}^{\times}/p^N)_\chi$ 
as a $\Lambda_\chi$-module. 
For an element $x \in ({F}^{\times}/p^N)_{\chi}$ 
and $\delta \in \ann_{\Lambda_\chi}(X_{\chi,\mrm{fin}})$, 
we denote the scaler multiple of $x$ 
by $\delta \in \Lambda_\chi$ by $x^\delta$.

We need the following two lemmas, 
namely Lemma \ref{tannsha} and \ref{delta kowaza 2}.

\begin{lem}\label{tannsha}
The kernel of the natural homomorphism 
\[
\xymatrix{\iota_{\mca{F},N} \colon 
\mca{F}/p^N \ar[r] & 
({F}^{\times}/p^N)_{\chi}}
\]
is annihilated by $\left| \Delta_p \right|$.
\end{lem}

\begin{proof}
Let $x$ be an element in the kernel of the homomorphism 
\[
\xymatrix{\iota_{\mca{F},N} \colon 
\mca{F}/p^N \ar[r] & 
({F}^{\times}/p^N)_{\chi}}
\]
and $\tilde{x}$ a lift of $x$ in $\mca{F}$. 
Then, there exists  $y \in {F}^{\times} \otimes \bb{Z}_p$ 
such that $\tilde{x}={y_\chi}^{p^N}$. 
Note that all $\bb{Z}_p$-torsion elements of 
$(\mca{I}_{F} \otimes \bb{Z}_p)_\chi/\mca{J}$ 
are annihilated by $\left| \Delta_p \right|$ 
by Corollary \ref{torsion free} 
since the $\bb{Z}_p$-module 
$(\mca{I}_{F}\otimes \bb{Z}_p)/(J\otimes \bb{Z}_p)$ 
is torsion free.
Since $(\tilde{x})_{F, \chi} \in \mca{J}$,
we have 
$(y^{\left| \Delta_p \right|})_{F, \chi} \in \mca{J}$.
Therefore we have 
$y^{\left| \Delta_p \right|} \in \mca{F}$, 
and we obtain $x^{\left| \Delta_p \right|}=1$.
\end{proof}

We denote the image of the natural homomorphism 
$\iota_{\mca{F},N} \colon 
\mca{F}/p^N \longrightarrow ({F}^{\times}/p^N)_{\chi}$ 
by $\bar{\mca{F}}_N$. 
By Lemma \ref{tannsha}, 
there exists an $R_N$-homomorphism
\[
\xymatrix{
\tilde{\beta}_{r,N} \colon 
\bar{\mca{F}}_N \ar[r] & R_N
}
\]
which makes the diagram
\[
\xymatrix{
\mca{F}/p^N \ar@{->>}[rr]^(0.56){\iota_{\mca{F},N}} 
\ar[rrd]_{\left| \Delta_p \right| \cdot 
\bar{\beta}_{r,N}} & & 
\bar{\mca{F}}_N 
\ar@{-->}[d]^{\tilde{\beta}_{r,N}} \\
& & R_N
}
\]
for each integer $r$ with $1 \le r \le h$,
where 
$\bar{\beta}_{r,N}\colon \mca{F}/p^N \longrightarrow R_N$ 
is the homomorphism induced by $\beta_r$.

\begin{lem}\label{delta kowaza 2}
Let $[\cdot ]_{F,N,\chi} \colon 
({F}^{\times}/p^N)_\chi \longrightarrow
(\mca{I}_{F}/p^N)_{\chi}$ 
be the homomorphism induced by 
$(\cdot )_{F}\colon {F}^\times \longrightarrow
\mca{I}_{F}$. 
Let $x$ be an element of 
$({F}^{\times}/p^N)_\chi$ such that 
$[x]_{F,N,\chi} \in \mca{J}/p^N.$
Then, for any 
$\delta \in \ann_{\Lambda_\chi}(X_{\chi,\mrm{fin}})$, 
the element $x^{\delta\left| \Delta_p \right|^2}$ 
is contained 
in $\bar{\mca{F}}_N$.
\end{lem}

\begin{proof}
Recall the natural exact sequence:
\[
\xymatrix{
0 \ar[r] & \mca{P} \ar[r] & 
\mca{I}_{F}\otimes\bb{Z}_p \ar[r] & A_{m} \ar[r] & 0,
}
\] 
where $\mca{P}$ is defined by 
$\mca{P}: = ({F}^\times/\mca{O}_{F}^\times)\otimes\bb{Z}_p$. 
By the snake lemma for the commutative diagram
\[
\xymatrix{
0 \ar[r] & \mca{P} \ar[r] \ar[d]^{\times p^N} & 
\mca{I}_{F}\otimes\bb{Z}_p \ar[r] \ar[d]^{\times p^N} & 
A_{m} \ar[r] \ar[d]^{\times p^N} & 0 \\
0 \ar[r] & \mca{P} \ar[r]  & 
\mca{I}_{F}\otimes\bb{Z}_p \ar[r]  & A_{m} \ar[r]  & 0, 
}
\]
we obtain the following exact sequence 
\[
\xymatrix{
0 \ar[r] & A_m \ar[r] & 
\mca{P}/p^N \ar[r] & 
\mca{I}_{F}/p^N \ar[r] & A_m \ar[r] & 0. 
}
\]
(Recall we take $p^{N_m} > \left|A_m \right|$.)

Let $B_m$ be the image of 
$J\otimes\bb{Z}_p$ in $A_m$, 
and $\mca{P}_0 \subset  \mca{P}$ 
the inverse image of $J\otimes\bb{Z}_p$. 
Then, we have the exact sequence  
\[
\xymatrix{
0 \ar[r] & \mca{P}_0 \ar[r] & 
J\otimes\bb{Z}_p \ar[r] & B_m \ar[r] & 0,
}
\] 
and by a similar argument as above, 
we obtain the exact sequence
\[
\xymatrix{
0 \ar[r] & B_m \ar[r] & 
\mca{P}_0/p^N \ar[r] & J/p^N \ar[r] 
& B_m \ar[r] & 0.
}
\] 
Now, we obtain the commutative diagram
\begin{align}\label{kernelfinite}
\xymatrix{
0 \ar[r] & B_m \ar[r] \ar@{^{(}->}[d] & \mca{P}_0/p^N \ar[r] 
\ar@{^{(}->}[d] & J/p^N \ar[r] \ar@{^{(}->}[d] & 
B_m \ar[r] \ar@{^{(}->}[d] & 0 \\
0 \ar[r] & A_m \ar[r] & \mca{P}/p^N \ar[r] & 
\mca{I}_{F}/p^N \ar[r]  & A_m \ar[r]  & 0 \\
}
\end{align}
whose two rows are exact, and the vertical arrows are injective.
Let $\tilde{\delta}$ be 
an arbitrary element of $\ann_{\Lambda}(A_m/B_m)$.
Let $x$ be an arbitrary element of $\mca{P}/p^N$ satisfying 
$[x]_N\in J/p^N$.
Let us show that $x^\delta$ is contained in $\mca{P}_0/p^N$.
By the diagram (\ref{kernelfinite}), there exists an element 
$y \in \mca{P}_0/p^N$ satisfying $[x]_N=[y]_N$. 
Since $[xy^{-1}]_N=0$, the element $xy^{-1}$ is contained in 
the image of $A_m$. 
Since $\delta A_m$ is contained in $B_m$, 
$(xy^{-1})^{\delta}$ is contained in the image of $B_m$.
In particular, we have $(xy^{-1})^{\delta} \in \mca{P}_0/p^N$,
and we obtain $x^\delta \in \mca{P}_0/p^N$.
Combine this result with Lemma \ref{elementary lemma}, 
we obtain the lemma.
\end{proof}

We define the $R_N$-submodule $\bar{\mca{F}}_{N}'$ of 
$(F^{\times}/p^N)_{\chi}$ by 
\[
\bar{\mca{F}}_N' :=
\{x^{\left| \Delta_p \right|^6} \mathrel{|} 
x\in \bar{\mca{F}}_N \}.
\]
By the first row of the diagram (\ref{kernelfinite}), 
we obtain the following corollary.

\begin{cor}\label{yuugensei}
The order of the kernel of 
$[\cdot]_{m,N,\chi}\colon 
\mca{F}_N' \longrightarrow \mca{J}/p^N$
is finite.
\end{cor}

\begin{proof}
Let $\mca{P}_0 \subseteq 
({F}^\times/\mca{O}_F^\times)\otimes \bb{Z}_p$ 
be as in  the proof of Lemma \ref{delta kowaza 2}. 
We denote the inverse image of $\mca{P}_0$ 
by the natural homomorphism 
\[
\xymatrix{
{F}^\times\otimes \bb{Z}_p \ar[r] & 
\mca{P}=({F}^\times/\mca{O}_{F}^\times)\otimes \bb{Z}_p
}
\]
by $\widetilde{\mca{P}}_0$.
Note that 
the kernel of the natural homomorphism 
\[
\xymatrix{
\mrm{pr_{\mca{P}_0,N}}\colon 
\widetilde{\mca{P}}_0/p^N 
\ar[r] & \mca{P}_0/p^N 
}
\] 
coincides with the image of 
$\mca{O}_{F}^\times/p^N$
in $\widetilde{\mca{P}}_0/p^N$. 
So, $\Ker\mrm{pr_{\mca{P}_0}}$ has finite order.
The top row of the diagram (\ref{kernelfinite}) 
implies that 
the kernel of the homomorphism 
\[
\xymatrix{
[\cdot]'_{m,N}\colon 
\mca{P}_0/p^N 
\ar[r] & (J/p^N)_\chi
}
\] 
induced by the natural homomorphism 
$[\cdot]_{m,N} \colon 
F^{\times}/p^N 
\longrightarrow (J/p^N)_\chi$ 
has finite order. 
Then, the kernel of the composite map 
\[
\xymatrix{
[\cdot]'_{m,N}\circ \mrm{pr}_{\mca{P}_0,N} \colon 
\widetilde{\mca{P}}_0/p^N \ar[r] & 
(J/p^N)_\chi
}
\]
is finite. 

Let us consider the commutative diagram 
\begin{equation}\label{commutative diagram for finiteness}
\xymatrix{
(\widetilde{\mca{P}}_0/p^N)_\chi 
\ar@{->>}[rr]^{ 
([\cdot]'_{m,N}\circ \mrm{pr}_{\mca{P}_0,N})_\chi
} 
\ar[d]_{\iota_1} 
& & (J/p^N)_\chi \ar@{->>}[d]^{\iota_2} \\
\mca{F}/p^N \ar@{->>}[r]^{\iota_{\mca{F},N}} & 
\bar{\mca{F}}_N  \ar[r]^(0.46){[\cdot]_{m,N,\chi}}  
& \mca{J}/p^N
} 
\end{equation}
of natural homomorphisms. 
Lemma \ref{elementary lemma} implies that  
the $\Coker{\iota_1}$ is annihilated by 
$\left| \Delta_p \right|^2$. 
By Corollary \ref{kernel}, 
it follows that $\Ker{\iota_2}$ is annihilated 
by ${\left| \Delta_p \right|}^2$.
The finiteness of the kernel of 
$[\cdot]'_{m,N}\circ \mrm{pr}_{\mca{P}_0,N}$ and 
Corollary \ref{kernel} imply that 
the order of 
${\left| \Delta_p \right|}^2
\Ker{([\cdot]'_{m,N}\circ \mrm{pr}_{\mca{P}_0,N})_\chi}$ 
is finite. 
Hence we obtain the corollary 
by chasing the commutative diagram 
(\ref{commutative diagram for finiteness}).
\end{proof}

Let $n$ be an element of $\mca{N}_N$ 
whose prime divisors are in $P$. 
We define $P^n_F$ to be the set of 
all elements of $P$ dividing $n$. 
We define $J_n$ to be the subgroup of $J$ 
generated by $P^n_F$, 
and the submodule $\mca{J}_{n,N}$ of $\mca{J}/p^N$ 
the image of $J_n\otimes\bb{Z}_p$ in $\mca{J}/p^N$.
We denote by $\bar{\mca{F}}_{n,N} $ 
the inverse image of $\mca{J}_{n,N}$ by 
$[\cdot]_{m, \chi} \colon ({F}^{\times}/p^N)_\chi 
\longrightarrow \mca{J}/p^N$.
Then, we define the $R_N$-submodule 
$\bar{\mca{F}}_{n,N}'$ of $(F^\times/p^N)_\chi$
\[
\bar{\mca{F}}_{n,N}':= 
\bar{\mca{F}}_{n,N} \cap \bar{\mca{F}}_{N}'.
\]
Note that $\bar{\mca{F}}'_{n,N}$ is a {\em finite} 
$R_N$-submodule of $({F}^{\times}/p^N)_\chi$ 
by Corollary \ref{yuugensei}. 

For each integer $r$ with $1 \le r \le h$, 
let
\[
\xymatrix{
\bar{\alpha}_{r,N} \colon 
\mca{J}_{n,N} \ar[r] & R_N
}
\]
be the $R_N$-homomorphism induced by $\alpha_r$. 
We put 
\begin{align*}
\bar{\alpha}_N:=(\bar{\alpha}_{s,N})_{s=1}^h \colon &
\mca{J}_{n,N}  \longrightarrow R_N^h,  \\
\tilde{\beta}_N:=(\tilde{\beta}_{s,N})_{s=1}^h \colon &  
\bar{\mca{F}}_N \longrightarrow R_N^h. 
\end{align*}
Then, we obtain the commutative diagram
\begin{equation}\label{commutative diagram with alpha and beta}
\xymatrix{
\bar{\mca{F}}'_{n,N} \ar[rr]^{[\cdot]_{F, \chi}} 
\ar[d]_{\tilde{\beta}_{N}} & & 
\mca{J}_{n,N} 
\ar[d]^{\left| \Delta_p \right| \varepsilon 
\cdot \bar{\alpha}_{N}} \\
R_N^h \ar[rr]^{\bar{f}} & & R_N^h.   \\
}
\end{equation}
of $R_N$-modules. 

\subsection{}
Let $\delta_A$ be a non-zero element of 
$\ann_{\Lambda_{\chi}}(X_{\chi,\mrm{fin}})$.
In this and the next subsection, 
we write $\bar{\phi}^\ell$ 
in place of $\bar{\phi}_{m,N,\chi}^\ell$ for simplicity. 
Here, as in \cite{Ku}, 
we shall take the element 
$x_{\nu,q}  \in (F^\times/p^N)_\chi$ 
which is defined in Definition \ref{x}, 
to translate $\beta_r$ 
to homomorphisms of the type $\bar{\phi}^\ell$. 
Recall the element $x_{\nu,q}  \in (F^\times/p^N)_\chi$ 
is determined by $\eta$, $q$, $\nu$, 
and $\{ w_\ell \}_{\ell |\nu}$. We shall take them as follows.  

First, let us take a prime number $q$ by the following way. 
For each integer $r$ with $1 \le r \le h$, 
we fix a prime number $q_r \in P_{n_r}$.
We put $Q:=\prod_{r=1}^h q_r \in \mca{N}_N$. 
We fix a homomorphism 
$\varphi\colon E_{\infty,\chi} 
\longrightarrow \Lambda_{\chi}$
of $\Lambda_\chi$-modules 
with pseudo-null cokernel.
By the Iwasawa main conjecture, 
we have 
\[
\varphi(\bar{C}_{\infty,\chi})
=\det (M_0) \cdot I_\varphi(E;C),
\]
where $\bar{C}_{\infty,\chi}$ is 
the image of $C_{\infty,\chi}$ 
in $E_{\infty,\chi}$, 
and $I_\varphi(E;C)$ is an ideal of $\Lambda_\chi$ 
of finite index. 
We fix an element $\delta_\varphi \in I_\varphi(E;C)$.
Then, we fix a family $(\eta_m)_{m \ge 0} \in C_{\infty,\chi}$
of circular units which is defined by 
$\Lambda_\chi$-linear combination of basic circular units, 
and satisfies  
$\varphi\big( (\eta_m)_{m \ge 0} \big)
= \delta_\varphi \det (M_0)$.
We write  
\[
\eta_m=\eta_m(1):=
\prod_{d|\mf{f}_K}\eta_m^{d}(1)^{u_d} \times 
\prod_{j=1}^r 
\eta_m^{1,a_j}(1)^{v_j} \in F_m^\times
\]
where $r$ is a positive integer, 
$a_1,..., a_r$ 
are integers prime to $p$, 
and $u_d$ and $v_j$ are elements of $\Lambda_\chi$ 
for each positive integers $d$ 
and $i$ with $d|\mf{f}_K$ and $1 \le i \le r$. 
Here, we assume that $r$, $a_1,...,a_r$, $u_d$'s and $v_j$'s are 
constant independent of $m$. 
As in the previous section, we fix $m$ in this and the next section, 
and put $\eta:=\eta_m$ for simplicity.

Let 
$\bar{\varphi}_{F,N, \chi}\colon 
(E_{\infty,\chi})_{\Gamma_m}/p^N \longrightarrow R_N$ 
be the induced homomorphism by $\bar{\varphi}$. 
Recall that in the proof of Proposition \ref{size}, 
we define $\mca{NO}:=\mca{NO}_{m,N,\chi}$ 
to be the image of the natural homomorphism 
\[
\xymatrix{
(E_{\infty,\chi})_{\Gamma_m}/p^N \ar[r] & 
\big(\mathcal{O}_{F}^\times/p^N\big)_\chi \ar[r] 
& \big({F_m}^{\times}/p^N\big)_\chi.
}
\]
We fix non-zero elements $\delta_I \in I_{P_{\chi}^{E}}$ 
and $\delta_J \in J_{P_{\chi}^{E}}$.
By the same argument as in Lemma \ref{delta kowaza}, 
there exists a homomorphism  
$\xymatrix{\psi\colon \mca{NO} 
\ar[r] & R_{N}}$ 
which makes the diagram
\[
\xymatrix{
(C_{\infty,\chi})_{\Gamma_m}/p^N \ar[r] \ar@{->>}[d] & 
(E_{\infty,\chi})_{\Gamma_m}/p^N_\chi 
\ar[rrrr]^(0.55){(\gamma -1)^{a_\chi} 
\cdot \delta_I \delta_J \left| \Delta_p \right|^4 
\cdot \bar{\varphi}_{F,N,\chi}}  \ar[d] & & & & R_{N} \\
\mca{W}_{m,N,\chi}(1)  \ar@{^{(}->}[r] & 
\mca{NO} \ar@{-->}[rrrru]_{\psi}  & & & 
}
\] 
commute. 
By Proposition \ref{chebo appli}, 
we can take a prime number $q\in \mca{S}_N$ 
satisfying the following two conditions:
\begin{enumerate}
\item[(q1)] $q \in P_{n_1}'$, 
and $q \ne q_1$ if $\Delta_p=0$;
\item[(q2)] $\mca{NO}_{m,N,\chi}$ is contained in 
the kernel of $[\cdot ]_{m,N,\chi}^{q}$, 
and for all $x \in \mca{NO}_{m,N,\chi}$,
\[
\bar{\phi}^{q}(x) = \left| \Delta_p \right|^2\psi(x).
\] 
\end{enumerate}
In particular, we have 
\begin{align*}
\bar{\phi}^{q}(\eta) &= 
\left| \Delta_p \right|^2 \psi(\eta) \\
&=  (\gamma -1)^{a_\chi}{\left| \Delta_p \right|^6 
\delta_I \delta_J \bar{\varphi}_{m,N,\chi}}
\big( (\eta_m )_{m\ge 0} \big) \\
&= \left| \Delta_p \right|^6 \delta_I \delta_J 
\delta_\varphi(\gamma -1)^{a_\chi/2}\cdot  \det (M_0).
\end{align*}

Next, let us take  $\nu$ and $\{ w_\ell \}_{\ell |\nu}$. 
First, we consider 
$\tilde{\beta}_{m_1,N} \colon 
\bar{\mca{F}}'_{Qq,N} \longrightarrow R_N$. 
By Proposition \ref{chebo appli}, 
we can take $\ell_2 \in \mca{S}_N$ 
which splits completely in $F(\mu_{q})/\bb{Q}$, and 
satisfies $\ell_2 \in P'_{n_2}$, 
$\ell\ne q_2$ and 
\[
\bar{\phi}^{\ell_2}(x) = 
{\left| \Delta_p \right|^2 }\cdot
\tilde{\beta}_{m_1,N}(x)
\] 
for all $x \in \bar{\mca{F}}'_{Qq,N}$. 
We put $\nu_1:=1$.

In the case $i=1$, 
we put $\nu:=\nu_1=1$, 
and $x_{\nu,q}=x_{1,q}= 
\kappa_{m,N}(\eta;q)=\kappa(q)$. 
It follows from 
Proposition \ref{[] and phi and x} (1) 
and Lemma \ref{delta kowaza 2} 
that $x_{1,q}^{\delta_A\left| \Delta_p \right|^8 }$ 
is an element of $\mca{F}'_{Qq,N}$.  

Suppose $i \ge 2$. 
To take $\nu$ and $\{ w_\ell \}_{\ell |\nu}$, 
we choose prime numbers $\ell_r$ for each $r$ 
with $2\le r \le i+1$  by induction on $r$ as follows. 
Let $r$ be an integer satisfying $2<r \le i+1$. 
Suppose that for each $s$ with $2 \le s \le r-1$, 
we have chosen distinct prime numbers 
$\ell_{s} \in \mca{S}_N$ 
which splits completely in 
$F(\mu_{q \nu_{s-1}})/\bb{Q}$. 
We put 
$\nu_{r-1} := 
\prod_{s=2}^{r-1}\ell_s$. 
We consider the $R_N$-linear homomorphism 
\[
\xymatrix{
\tilde{\beta}_{m_{r-1},N}\colon 
\bar{\mca{F}}'_{Q q \nu_{r-1},N} \ar[r] &  R_N.
}
\]  
Applying Proposition \ref{chebo appli}, 
we can take $\ell_r \in \mca{S}_N$
which splits completely in 
$F(\mu_{q \nu_{r-1}})/\bb{Q}$, 
and satisfies the following conditions:
\begin{itemize}
\item[(x1)] $\ell_r \in P_{n_r}'$, 
and $\ell_r \ne q_r$ if $\Delta_p=0$;
\item[(x2)] there exists $b_r 
\in ({F}^\times \otimes \bb{Z}_p)_\chi$ 
such that 
$(b_r)_{F, \chi} = 
(\ell_{r,F}-\left| \Delta_p \right| 
\cdot q_{r,F})_{\chi}$ and 
$\bar{\phi}^{\ell_s}(b_r)=0$ 
for any $s$ with $2 \le s < r$;  
\item[(x3)] $\bar{\phi}^{\ell_r}(x)= 
\left| \Delta_p \right|^2\cdot 
\tilde{\beta}_{m_{r-1},N}(x)$ for any 
$x \in \bar{\mca{F}}'_{Q q \nu_{r-1},N}$.
\end{itemize}
Thus, we have taken $\ell_2, \dots , \ell_{i+1}$, 
and we put 
$\nu:=\nu_{i}=\prod_{r=2}^i\ell_r 
\in \mca{N}_N$. 
For each $r$ with $2 \le r \le i$, 
we put 
\[w_{\ell_r}:=-{\phi}^{\ell_r}(b_r) 
\in R_N \otimes H_{\ell_r},
\]
and we obtain 
$x_{\nu,q} \in ({F}^\times/p^N)_\chi$. 
It follows from Proposition \ref{[] and phi and x} (1) 
and Lemma \ref{delta kowaza 2} 
that $ x_{\nu,q}^{\delta_A \left| \Delta_p \right|^8}$ 
is an element of $\bar{\mca{F}}'_{Q q \nu,N}$. 
Note that $q\nu$ is {\em well-ordered}. 

\subsection{}
In this subsection, 
we observe two homomorphism 
$\alpha$ and $\beta$ by using $x_{\nu,q}$, 
and describe $\det (M_i)$ in $R_N$. 
First, we prepare the following lemma.
\begin{lem}[cf. \cite{Ku} Lemma 9.2]\label{ab}
Suppose $i \ge 2$. Then,
\begin{enumerate}
\item $ \tilde{\beta}_{m_{r-1},N}
(x_{\nu,q}^{\delta_A \left| \Delta_p \right|^{10}} )=0$ 
for all $r$ with $2 \le r \le i$;
\item $\bar{\alpha}_{j,N}([x_{\nu,q}]_{m,N,\chi})=0$ 
for any $j\ne n_1,...,n_i.$ 
\end{enumerate}
\end{lem}

\begin{proof}
The second assertion (2) of 
the above lemma is clear 
by Proposition \ref{[] and phi and x} (1). 

Let us prove the first assertion. 
We define an element $y_r \in (F^\times/p^N)_\chi$ by 
\[
y_r=x_{\nu,q}
\prod_{s=r}^ib_s^{\bar{\phi}^{\ell_s}(x_{\nu/\ell_s,q})} 
\]
for any $r$ satisfying $2 \le r \le i$. 
Note that we have $\bar{\alpha}_N ([b_r]_{m,N,\chi})=0$ 
for any $r$ satisfying $2 \le r \le i$ 
since we have 
$(b_r)_{F, \chi} =
(\ell_{r,F}-\left| \Delta_p \right|\cdot q_{r,F})_{\chi}$
and $\ell_r \in P_{n_r}'$.
By definition of $\beta$, 
we have ${\beta}(b_r)=0$ 
for any $r$ satisfying $2 \le r \le i$. 
So, we have 
\[
\tilde{\beta}( x_{\nu,q}^{\delta_A \left| \Delta_p \right|^8})=
\tilde{\beta}( y_r^{\delta_A \left| \Delta_p \right|^8})
\] 
for any $r$ with $2 \le r \le i$. 
Let us show   
$\tilde{\beta}_{m_{r-1},N}
(y_r^{\delta_A \left| \Delta_p \right|^8})=0$ 
for any integer $r$ satisfying $2 \le r \le i$. 
By Proposition \ref{[] and phi and x} (1), 
we have $ [y_r]_{F,N,\chi} \in J_{Qq \nu_{r-1}}$.  
Then, by Lemma \ref{delta kowaza 2}, 
we have 
$y_r^{\delta_A \left| \Delta_p \right|^8} 
\in \mca{F}_{Qq \nu_{r-1},N}$. 
Therefore, we obtain 
\[
{\delta_A \left| \Delta_p \right|^8} 
\bar{\phi}^{\ell_r}( y_r)= 
\left| \Delta_p \right|^2 \tilde{\beta}_{m_{r-1},N}
(y_r^{\delta_A \left| \Delta_p \right|^8} )
\]  
by the condition (x3). 
Since $\bar{\phi}^{\ell_r}( b_{s})=0$ 
for all integers $s$ satisfying $r+1 \le s \le i$ by the condition (x2), 
we have 
\[
\bar{\phi}^{\ell_r}( y_r) =
 \bar{\phi}^{\ell_r}( x_{\nu,q} 
b_r^{\bar{\phi}^{\ell_r}(x_{\nu/\ell_r,q})}).
\]
By Proposition \ref{[] and phi and x} (3), 
we have 
\begin{align*}
\bar{\phi}^{\ell_r}( x_{\nu,q} 
b_r^{\bar{\phi}^{\ell_r}(x_{\nu/\ell_r,q})}) 
&= \bar{\phi}^{\ell_r}( x_{\nu,q} ) 
+ \bar{\phi}^{\ell_r}( 
b_r^{\bar{\phi}^{\ell_r}(x_{\nu/\ell_r,q})}) \\
&= -\bar{\phi}^{\ell_r}(b_r)
\bar{\phi}^{\ell_r}(x_{\nu/\ell_r,q})
+\bar{\phi}^{\ell_r}(x_{\nu/\ell_r,q}) 
\bar {\phi}^{\ell_r}(b_r) \\
&= 0.
\end{align*}
Therefore, we obtain 
\[
\tilde{\beta}_{m_{r-1},N}
(x_{\nu,q}^{\delta_A \left| \Delta_p \right|^{10}})
=\left| \Delta_p \right|^{2}\tilde{\beta}_{m_{r-1},N}
(y_r^{\delta_A \left| \Delta_p \right|^8})
=\delta_A \left| \Delta_p \right|^8
\bar{\phi}^{\ell_r}( y_r) =0.
\] 
\end{proof} 

The goal of this subsection is the following proposition.

\begin{prop}[cf. \cite{Ku} p.44]\label{ind}
The following equalities 
of elements contained in $R_{N}$ holds. 
\begin{enumerate}
\item We have 
\begin{align*}
\delta_A & \left| \Delta_p \right|^{10}  
\det (M) \bar{\phi}^{\ell_2}(x_{1,q}) \\
 &= \pm  \left| \Delta_p \right|^{20} (\gamma -1)^{a_\chi} 
\delta_A \delta_I \delta_J   
\varepsilon \det (M_1)  
\bar{\varphi}_{F,N,\chi}\big( (\eta_m )_{m\ge 0} \big).
\end{align*}
\item We have 
\[
\delta_A \left| \Delta_p \right|^{10}
\det (M_{r-1}) 
\bar{\phi}^{\ell_{r+1}}(x_{\nu_{r},q}) 
= \pm \left| \Delta_p \right|^{14} \delta_A  
\varepsilon \det (M_{r})  
\bar{\phi}^{\ell_{r}}(x_{\nu_{r-1},q})
\]  
for any $r$ with $2 \le r \le i $.
\end{enumerate}
The signs $\pm$ in {\em (1)} and {\em (2)} do not depend on $m$.
\end{prop}

\begin{proof}
For each $r$ satisfying $1 \le r \le i $ we put 
\begin{align*}
\mb{x}^{(r)}&:=
\tilde{\beta}_N
({x_{\nu_r,q}}^{\left| \Delta_p \right|^{10} 
\delta_A}) \in R_N^h; \\
\mb{y}^{(r)}&:=
\left| \Delta_p \right| \varepsilon 
\bar{\alpha}_N({x_{\nu_r,q}}^{\left| 
\Delta_p \right|^{10} \delta_A}) 
\in R_N^h,
\end{align*}
and regard them as column vectors.
Then, by the commutative diagram 
(\ref{commutative diagram with alpha and beta}), 
we have $\mb{y}^{(r)}=M\mb{x}^{(r)}$
in $R_N^h$.

We first prove the assertion (1) of the above proposition.
Note that $\delta_A \left| \Delta_p \right|^2$ times of 
$x_{1,q}=\kappa(q)$ 
is an element of $\mca{F}_{q,N}$, and we have 
\begin{align*}
\mb{y}^{(1)}&= 
\left| \Delta_p \right|^{12} \delta_A \varepsilon 
[\kappa(q)_\chi]_{F,N,\chi}^q\mb{e}_{n_1}^{(m)}  \\
&= \left| \Delta_p \right|^{12} \delta_A \varepsilon 
\bar{\phi}^{q}(\eta_m )\mb{e}_{n_1}^{(m)}  \\
&= \left| \Delta_p \right|^{14} \delta_A \varepsilon \psi(\eta_m) \\
&=\left| \Delta_p \right|^{18} (\gamma -1)^{a_\chi}\delta_A   
\delta_I \delta_J \varepsilon \cdot 
\bar{\varphi}_{F,N,\chi}\big( (\eta_m )_{m\ge 0} \big).
\end{align*}
Note that the first equation 
follows from the condition (q1) and 
definition of $\alpha$, 
the second one 
follows from Proposition \ref{[] and phi} (2), 
and the third one follows from the condition (q2).
Let $\widetilde{M}$ be the matrix of cofactors of $M$. 
Multiplying the both sides of $\mb{y}^{(1)}=M\mb{x}^{(1)}$ 
by $\widetilde{M}$,
and comparing the $m_1$-st components, we obtain 
\begin{align*}
& (-1)^{n_1+m_1}\left| \Delta_p \right|^{18} 
(\gamma -1)^{a_\chi}\delta_A   
\delta_I \delta_J \varepsilon \det (M_1) 
\bar{\varphi}_{F,N,\chi}\big( (\eta_m )_{m\ge 0} \big)\\
& =\det (M) \tilde{\beta}_{m_1,N}(
{x_{1,q}}^{\left| \Delta_p \right|^{10} \delta_A}).
\end{align*}
By condition (x3) for $\ell_2$, we have 
\[
\left| \Delta_p \right|^{2}
\tilde{\beta}_{m_1,N}({x_{1,q}}^{\left| \Delta_p \right|^{10} \delta_A}) 
= {\left| \Delta_p \right|^{10} \delta_A} 
\bar{\phi}^{\ell_2}(x_{1,q}),
\]
and the assertion (1) follows.

Next, we assume $i \ge 2$, and we shall prove Proposition \ref{ind} (2). 
The proof is essentially the same as the proof of assertion (1). 
It is sufficient to prove the assertion when $r=i$.
We write $\mb{x}=\mb{x}^{(i)}$ and $\mb{y}=\mb{y}^{(i)}$.
Let $\mb{x}' \in R_N^{h-i+1}$ be the vector obtained from 
$\mb{x}$ by eliminating the $m_j$-th rows for $j=1,...,i-1$, 
and $\mb{y}'$ the vector obtained from $\mb{y}$ 
by eliminating the $n_k$-th rows  for $k=1,...,i-1$. 
Since the $m_r$-th rows of $\mb{x}$ are 0 
for all $r$ with $1 \le r \le i-1$ by Lemma \ref{ab} (1), 
we have $\mb{y}'=M_{i-1}\mb{x}'.$ 
We assume the $m_i'$-th component of $\mb{x}'$ 
corresponds to the $m_i$-th component of $\mb{x}$, 
and the $n_i'$-th component of $\mb{y}'$ 
corresponds to the $n_i$-th component of $\mb{y}$. 
By Lemma \ref{ab} (2) and Proposition \ref{[] and phi and x} (2), 
we have 
\[
\mb{y}'=\left| \Delta_p \right|^{12} \delta_A 
\varepsilon \bar{\phi}^{\ell_i}(x_{\nu_{i-1},q})\mb{e'}_{n_i '}^{(m)},
\]
where $( \mb{e'}_i^{(m)} )_{i=1}^{h-i+1}$ 
denotes the standard basis of $R_N^{h-i+1}$.

Let $\widetilde{M}_{i-1}$ be the matrix of cofactors of $M_{i-1}$. 
Multiplying the both sides of $\mb{y}'=M_{i-1}\mb{x}'$ 
by $\widetilde{M}_{i-1}$,
and comparing the $m_i '$-th components, we obtain 
\[
(-1)^{n_i'+m_i'}\left| \Delta_p \right|^{12}
\delta_A \varepsilon \det (M_i)
\bar{\phi}^{\ell_i}(x_{\nu_{i-1},q})
=\det (M_{i-1}) 
\tilde{\beta}_{m_i,N}(
{x_{\nu,q}}^{\left| \Delta_p \right|^{10}\delta_A}).
\]
By condition (x3) for $\ell_{i+1}$, and 
since $x_{n,q}^\delta $ is an element of $\mca{F}_{Qq\nu,N}$, 
we have 
\[
\left| \Delta_p \right|^2
\tilde{\beta}_{m_i,N}(
{x_{\nu,q}}^{\left| \Delta_p \right|^{10} \delta_A}) 
= {\left| \Delta_p \right|^{10} \delta_A}
\bar{\phi}^{\ell_{i+1}}(x_{\nu,q}).
\]
Here, the proof of Proposition \ref{ind} is complete.
\end{proof}  

\subsection{} 
Now let us prove Theorem \ref{Main theorem}.
It is convenient to use the following notion of convergence. 

\begin{dfn}\label{convergence}
A sequence $(a_m)_{m \ge 0} \in \prod_{m \ge 0} R_{F_m,N_m,\chi}$ 
is said to {\em converge} to 
\[
b=(b_m)_{m \ge 0} \in \plim_{m \ge 0}R_{F_m,N_m,\chi}=\Lambda_{\chi}
\]
if for each $m$, 
there exists an integer $L_m$ such that the image of $a_{m'}$ 
in $R_{F_m,N_m,\chi}$ coincides to $b_m \in R_{F_m,N_m,\chi}$ 
for any $m' \ge L_m$.
\end{dfn}

\begin{proof}[Proof of Theorem \ref{Main theorem} ]
Here, we vary $m$. 
In this subsection, we denote the element 
\[
\bar{\phi}^{\ell_{r+1}}(x_{\nu_r,q}) 
\in R_N=(\bb{Z}/p^{N_m})[\Gal(F_m/F_0)]_\chi
\] 
defined in \S 6.2 
by $\bar{\phi}^{\ell_{r+1}}(x_{\nu_r,q})_m$.
By induction on $r$, 
let us prove that the sequence
$\big( \bar{\phi}^{\ell_{r+1}}(x_{\nu_{r},q})_m \big)_{m\ge 0}$ 
converges to 
\[
\pm \left| \Delta_p \right|^{6+4r} (\gamma -1)^{a_\chi}
\delta_I \delta_J \delta_\varphi \varepsilon^r  
\det (M_{r}) \in \Lambda_{\chi}
\]
in the sense of Definition \ref{convergence}
for any integer $r$ satisfying $0 \le r \le i$.

First, we consider the equality
\begin{align*}
\left| \Delta_p \right|^{10} & \delta_A
\det (M)\cdot  \bar{\phi}^{\ell_2}(x_{1,q}) \\
&= \pm \left| \Delta_p \right|^{20}(\gamma -1)^{a_\chi}  
\delta_A \delta_I \delta_J \varepsilon 
\det (M_1) \bar{\varphi}_{F,N,\chi}\big( (\eta_m )_{m\ge 0} \big). 
\end{align*}
Since the right hand side converges to 
\[
\pm \left| \Delta_p \right|^{20}(\gamma -1)^{a_\chi} 
\delta_A \delta_I \delta_J \delta_\varphi  
\varepsilon \det (M_1)\det (M)
\]
and $\left| \Delta_p \right|^{10} \delta_A \det (M)$ 
is non-zero element, it follows that 
$\big( \bar{\phi}^{\ell_2}(x_{1,q})_m \big)_{m\ge 0}$ 
converges to 
\[
\pm \left| \Delta_p \right|^{10}(\gamma -1)^{a_\chi} 
\delta_I \delta_J \delta_\varphi \varepsilon \det (M_1).
\]
(Note the sign $\pm$ does not depend on $m$, 
see Proposition \ref{ind}). 

Next, we assume that the sequence 
$\big( \bar{\phi}^{\ell_r}(x_{\nu_{r-1},q})_m \big)_{m \ge 0}$ 
converges to 
\[
\pm \left| \Delta_p \right|^{2+4r}(\gamma -1)^{a_\chi} 
\delta_I \delta_J \delta_\varphi 
\varepsilon^{r-1} \det (M_{r-1}).
\]
Then, the right hand side of 
\[
\left| \Delta_p \right|^{10} \delta_A\det (M_{r-1}) 
\bar{\phi}^{\ell_{r+1}}(x_{\nu_{r},q}) 
= \pm \left| \Delta_p \right|^{14} \delta_A \varepsilon 
\det (M_{r}) \bar{\phi}^{\ell_{r}}(x_{\nu_{r-1},q})
\]
converges to 
\[
\pm \left| \Delta_p \right|^{16+4r}(\gamma -1)^{a_\chi} 
\delta_A \delta_I \delta_J \delta_\varphi   
\varepsilon^r  \det (M_{r}) \det (M_{r-1})
\]
Since we take $\det (M_{r-1}) \ne 0$, 
the sequence 
$\big( \bar{\phi}^{\ell_{r+1}}(x_{\nu_r,q})_m \big)_{m \ge 0}$ 
converges to 
\[
\pm  \left| \Delta_p \right|^{6+4r} (\gamma -1)^{a_\chi} 
\delta_I \delta_J \delta_\varphi  
\varepsilon^r \det (M_{r}).
\]

By induction, in particular, we conclude 
$\big( \bar{\phi}^{\ell_{i+1}}(x_{\nu,q})_m \big)$ 
converges to 
\[
\pm \left| \Delta_p \right|^{6+4i}(\gamma -1)^{a_\chi} 
\delta_I \delta_J \delta_\varphi \varepsilon^i \det (M_{i}).
\]
Since $(x_{\nu,q})_m$ is contained in an $R_N$-submodule of 
$(F^\times/p^N)_\chi$ generated by the set
$\bigcup_{e|q\nu} \mca{W}_{m,N,\chi}(e)$ with $\epsilon (q\nu)=i$, 
we have 
$\bar{\phi}^{\ell_{i+1}}(x_{\nu,q})_m \in\mf{C}_{i,F_m,N \chi}$ 
for all $m \in \bb{Z}_{\ge 0}$. 
Hence we have 
\[
\pm \left| \Delta_p \right|^{6+4i} (\gamma -1)^{a_\chi}  
\delta_I \delta_J \delta_\varphi \varepsilon^i 
\det (M_{i}) \in \mf{C}_{i,\chi}.
\]
This completes the proof of theorem.
\end{proof}

\section{The higher cyclotomic ideals and Mazur-Rubin theory}\label{the cyclotomic ideals and Mazur-Rubin theory}
In this section, we assume that 
the extension degree of $K/\bb{Q}$ is coprime to $p$, and 
fix a character $\chi \in \widehat{\Delta}$ 
satisfying $\chi (p)\ne 1$. 
We denote by $\mca{O}$ the $\bb{Z}_p$-algebra 
isomorphic to $\mca{O}_\chi$ with trivial $G_{\bb{Q}}$-action. 
We identify the ring $\mca{O}$ (resp.\ $\mca{O}[[\Gamma]]$) 
with $\mca{O}_\chi$ (resp.\ $\Lambda_\chi$) 
when we ignore the action of $G_{\bb{Q}}$. 
In particular, we sometimes regard $\mf{C}_{i,\chi}$ 
as an ideal of $\mca{O}[[\Gamma]]$, and 
$R_{m,N,\chi}$ as a quotient ring of $\mca{O}[[\Gamma]]$. 
In this section, 
we complete the proof of Theorem \ref{Main theorem, Rough} 
in \S \ref{notation}. 

\begin{thm}[Theorem \ref{Main theorem, Rough}]
We assume that the extension degree of $K/\bb{Q}$ is prime to $p$. 
Let $\chi  \in \widehat \Delta$ be a character
satisfying $\chi(p)\ne 1$.
Then, we have 
\[
 \Fitt_{\Lambda_\chi,i}(X_\chi)\sim \mf{C}_{i,\chi}
\]
for any $i \in \bb{Z}_{\ge 0}$. 
Moreover, we have 
\[
\ann_{\Lambda_\chi}(X_{\chi,\mrm{fin}}) 
\Fitt_{\Lambda_\chi,i}(X'_\chi)\subseteq \mf{C}_{i,\chi}
\]
for any $i \in \bb{Z}_{\ge 0}$. 
\end{thm}

In the previous section, we have already proved 
\[
\ann_{\Lambda_\chi}(X_{\chi,\mrm{fin}}) 
\Fitt_{\Lambda_\chi,i}(X'_\chi)\subseteq \mf{C}_{i,\chi}
\]
for any $i \ge 0$. 
In order to complete the proof of the theorem, 
we have to show 
\[
\mf{C}_{i,\chi} \prec 
\Fitt_{\Lambda_\chi,i}(X_\chi)
\]
for any $i \in \bb{Z}_{\ge 0}$. 
It is sufficient to show the following theorem. 

\begin{thm}\label{local lower bounds}
Let $i$ be a positive integer, and 
$\mf{P} \subset \mca{O}[[\Gamma]]=\Lambda_\chi$ 
a prime ideal of hight one containing 
$\Fitt_{\Lambda_\chi,i}(X_\chi)$.   
We define two integers 
$\alpha_{i}({\mf{P}})$ and 
$\beta_{i}({\mf{P}})$ by
\begin{align*}
\Fitt_{\Lambda_{\chi,\mf{P}},i}
(X_\chi \otimes_{\mca{O}[[\Gamma]]} \Lambda_{\chi,\mf{P}} )
& =\mf{P}^{\alpha_i({\mf{P}})}\Lambda_{\chi,\mf{P}}, \\
\mf{C}_{i,\chi}\Lambda_{\chi,\mf{P}}
& =\mf{P}^{\beta_{i}({\mf{P}})}\Lambda_{\chi,\mf{P}}. 
\end{align*}
Then, we have $\beta_{i}({\mf{P}})\ge \alpha_{i}({\mf{P}})$.
\end{thm}

In the rest of this section, 
we prove Theorem \ref{local lower bounds}.
The key of the proof is comparison between 
the higher cyclotomic ideals $\mf{C}_{i,\chi}$ 
defined in this paper and 
the theory of Kolyvagin systems, 
which is established by Mazur and Rubin in \cite{MR}. 

\subsection{}
In the first three subsections, 
we will review some results 
on Kolyvagin systems briefly. 
Here, we recall the notion of Kolyvagin systems 
in general setting. 

Let $(R,\mf{m})$ be a noetherian complete local ring 
whose residue field $R/\mf{m}$
is a finite field of characteristic $p$. 
We call a triple $(T,\mca{F},\Sigma)$ which consists of 
the following data a Selmer structure 
(cf.\ \cite{MR} Definition 2.1.1):
\begin{itemize}
\item a finite set $\Sigma$ of places of $\bb{Q}$ containing 
$p$ and $\infty$;
\item a free $R$-module $T$ of finite rank with 
continuous $G_{\bb{Q}}$-action unramified outside 
$\Sigma \setminus \{ \infty \}$.
\item a local condition 
\[
\mca{F}=\{ H^1_{\mca{F}}(\bb{Q}_v,T) 
\subseteq H^1(\bb{Q}_v,T)\}_{v:\text{place of $\bb{Q}$}}
\]
on $T$ satisfying 
\[
H^1_{\mca{F}}(\bb{Q}_v,T)=H^1_f(\bb{Q}_v,T)
=H^1(\bb{Q}_v^{\mrm{unr}}/\bb{Q}_v,T)
\]
for all finite places $v \notin \Sigma$.
\end{itemize}
We define the Selmer group 
for a Selmer structure $(T,\mca{F},\Sigma)$ by
\[
H^1_{\mca{F}}(\bb{Q}, T):=
\Ker\bigg(\xymatrix{
H^1(\bb{Q},T) \ar[r] & 
\displaystyle\prod_{v}
\frac{H^1(\bb{Q}_v, T)}{H^1_{\mca{F}}(\bb{Q}_v, T)} 
\bigg),
}
\]
where in the product, $v$ runs through all places of $\bb{Q}$.

Let $(T,\mca{F},\Sigma)$ be a Selmer structure. 
For each prime number $\ell \notin \Sigma$, 
we define $$H^1_{\mrm{tr}}(\bb{Q}_\ell,T):=
\Ker(\xymatrix{H^1(\bb{Q}_\ell,T) \ar[r] & 
H^1(\bb{Q}_\ell(\mu_\ell),T)}).$$
Then, we have a direct decomposition 
$H^1(\bb{Q}_\ell,T)=
H^1_{\mca{F}}(\bb{Q}_\ell,T) \oplus
H^1_{\mrm{tr}}(\bb{Q}_\ell,T)$, 
so the natural projection 
\[
\xymatrix{
H^1_{\mrm{tr}}(\bb{Q}_\ell,T) \ar[r] & 
H^1_s(\bb{Q}_\ell,T):=H^1(\bb{Q}_\ell,T)/H^1_f (\bb{Q}_\ell,T)
}
\]
is an isomorphism. 
For each square-free product $n$ of finite numbers of primes
not contained in $\Sigma$, 
we define a new local condition
\begin{equation*}
H^1_{\mca{F}(n)}(\bb{Q}_v, T):=
\begin{cases}
H^1_{\mca{F}}(\bb{Q}_v,T) 
& \text{if $v\not| n$;} \\
H^1_{\mrm{tr}}(\bb{Q}_v,T) & \text{if $v|n$.} 
\end{cases}
\end{equation*}
Let $I$ be an ideal of $R$. 
Then, we denote the image of 
$H^1_{\mca{F}}(\bb{Q}_v, T)$ 
in $H^1(\bb{Q}_v, T/IT)$
by $H^1_{\mca{F}}(\bb{Q}_v, T/IT)$ 
for any place $v$ of $\bb{Q}$.

In order to define the Kolyvagin systems,
we have to construct two homomorphisms of Galois cohomology groups, 
the ``localization" map and the ``finite-singular comparison" map. 
Let $(T,\mca{F},\Sigma)$ be a Selmer structure over $R$. 
For each prime number $\ell \notin \Sigma$, 
we define 
\[
P_\ell(x):=\det_R(1- \mathrm{Fr}_\ell x\mid T),
\]
where $\mathrm{Fr}_\ell \in G_{\bb{Q}}$ is 
an arithmetic Frobenius element. 
Note that the polynomial $P_\ell(x)$ is well-defined
since the action of $G_{\bb{Q}}$ on $T$ is unramified at $\ell$. 
Let $I_\ell$ be the ideal of $R$ generated by $\ell-1$ and 
$P_\ell(1)$. 
Take a square-free product $n:=\ell_1\times \cdots \times \ell_r$,
where $\ell_i$ is a prime number not contained in $\Sigma$ for 
$i=1,...,r$. 
Then, we define an ideal 
\[
I_n = \sum_{i=1}^rI_{\ell_i}.
\]
Recall that we defined 
$H_\ell:=\Gal(\bb{Q}(\mu_\ell)/\bb{Q}) 
\simeq \bb{F}_\ell^\times$.
By \cite{MR} Definition 1.2.2 and Definition 2.2.1, 
we can construct a canonical map
\[
\xymatrix{
\phi_\ell^{\rm{fs}} \colon 
H^1_f(\bb{Q}_{\ell_i}, T/I_nT) \ar[r] & 
H^1_s(\bb{Q}_{\ell_i}, T/I_nT)\otimes_{\bb{Z}} H_\ell 
}
\]
called the finite-singular comparison map 
for each integer $i$
with $1 \le i \le r$. 
(See loc.\ cit.\ for construction and detail
of $\phi_\ell^{\rm{fs}}$.)
Note that if $R=\mca{O}/p^N$ and 
$T=\mu_{p^N} \otimes_{\bb{Z}_p} \mca{O}_{\chi^{-1}}$, 
then the natural localization map 
\[
\xymatrix{
(\cdot)_{\ell,s}\colon
H^1_{\mca{F}(n\ell)}(\bb{Q}, T/I_{n\ell} T)
\otimes_{\bb{Z}} H_{n\ell} 
\ar[r] & H^1_s(\bb{Q}_{\ell}, T/I_{n\ell}T)
\otimes_{\bb{Z}} H_{n\ell}, 
}
\]
and the composite map 
\[
\xymatrix@R=2mm{
H^1_{\mca{F}(n)}(\bb{Q}, T/I_{n} T)\otimes_{\bb{Z}} H_{n} 
\ar[r] & 
H^1_f(\bb{Q}_{\ell}, T/I_{n\ell}T)\otimes_{\bb{Z}} H_{n} 
\\
\hspace{40mm} \ar^(0.48){\phi_\ell^{\mrm{fs}}\otimes 1}[r] & 
H^1_s(\bb{Q}_{\ell}, T/I_{n\ell}T)\otimes_{\bb{Z}} H_{n\ell}, 
} 
\]
coincide with the map
\[
\xymatrix{
[\cdot ]_{0,N,\chi}^\ell \otimes 1 \colon 
 (F_0^{\times}/p^N)_{\chi}\otimes_{\bb{Z}} H_{n\ell} 
\ar[r] & R_{0, N, \chi}\otimes_{\bb{Z}} H_{n\ell}   
}
\]
defined in Definition \ref{[]} and
\[
\xymatrix{
{\phi}_{0,N,\chi}^\ell \otimes 1 \colon  
(F_0^{\times}/p^N)_{\chi} \otimes_{\bb{Z}} H_{n}
\ar[r] & R_{0, N, \chi}
\otimes_{\bb{Z}} H_{n\ell}
}
\]
defined in Definition \ref{phi}  
respectively for any $\ell \in \mca{S}_N$ and 
any $n \in \mca{N}_N$ with $(\ell,n)=1$.

Let us recall the definition of Kolyvagin systems. 
Let $(T,\mca{F},\Sigma)$ be a Selmer structure over $R$ and 
$\mca{P}$ be a set of rational primes disjoint from $\Sigma$. 
We denote the set of all square-free products of $\mca{P}$ 
by $\mca{N}(\mca{P})$. 
(Note $1 \in \mca{N}(\mca{P})$.) 
We call such a triple $(T,\mca{F},\mca{P})$ a Selmer triple.

\begin{dfn}
A Kolyvagin system for a Selmer triple $(T,\mca{F},\mca{P})$ 
is a family of cohomology classes 
\[
\kappa=\{\kappa_n \in  
H^1_{\mca{F}(n)}(\bb{Q}, T/I_{n} T)
\otimes_{\bb{Z}} H_{n}\}_{n \in \mca{N}(\mca{P})}
\]
satisfying 
\[
(\kappa_{n\ell})_{\ell,s}=
\phi_{\ell}^{\mrm{fs}}(\kappa_{n}) \hspace{3mm}
\text{in} \hspace{3mm}
H^1_{s}(\bb{Q}_\ell, T/I_{n\ell} T)
\otimes_{\bb{Z}} H_{n\ell}
\]
for any $n \in \mca{N}(\mca{P})$ and any $\ell\in \mca{P}$
satisfying $n\ell \in \mca{N}(\mca{P})$. 
We denote the set of all Kolyvagin systems for $(T,\mca{F},\mca{P})$
by $\mrm{KS}(T,\mca{F},\mca{P})$. 
\end{dfn}

Let $(T,\mca{F},\mca{P})$ be a Selmer triple over $R$. 
We naturally regard an abelian group  
\[
T^*:=\Hom_{\bb{Z}_p}(T,\mu_{p^\infty})
\]
as an $R[G_{\bb{Q}}]$-module.
For all finite places $v$ of $\bb{Q}$, 
we denote the orthogonal complement of 
$H^1_{\mca{F}}(\bb{Q}_v,T)$ 
in the sense of 
the local duality pairing 
\[
\xymatrix{
H^1(\bb{Q}_v,T) \times H^1(\bb{Q}_{\ell},T^*) \ar[r] & 
H^2(\bb{Q}_v,\mu_{p^\infty})\simeq \bb{Q}_p/\bb{Z}_p}
\]
by $H^1_{\mca{F}_{\mrm{can}}}(\bb{Q}_v,T^*)$. 
If $v=\infty$, we have 
$H^1(\bb{Q}_v,T)=0 \hspace{3mm} \text{and} \hspace{3mm}
H^1(\bb{Q}_v, T^*)=0$ since $p$ is odd.
In certain good situations, Kolyvagin systems for the triple
$(T,\mca{F}, \mca{P})$ describe the structure of 
the Selmer group
\[
H^1_{\mca{F}^*}(\bb{Q}, T^*)=
\Ker\bigg(
\xymatrix{
H^1(\bb{Q},T^*) \ar[r] & 
\displaystyle\prod_{v}
\frac{H^1(\bb{Q}_v, T^*)}{H^1_{\mca{F}^*}(\bb{Q}_v, T^*)} 
\bigg),
}
\]
where in the product, $v$ runs through all places of $\bb{Q}$. 

For $r \in \bb{Z}_{\ge 0}$, we denote by $\mca{P}_r$
the set of all prime numbers $\ell$ 
not contained in $\Sigma$ satisfying both of the following to conditions: 
\begin{itemize}
\item $T/(\mf{m}^rT+(\mathrm{Fr}_\ell-1)T)$ is 
a free $R$-module of rank one;
\item $I_\ell \subseteq \mf{m}^r$.
\end{itemize}
Then, we have $\mca{P}_i \subseteq \mca{P}_{i+1}$ 
for any positive integer $i$. 
\begin{dfn}
Let $(T,\mca{F},\mca{P})$ be a Selmer triple over $R$. 
Then, we define
$$\overline{\mrm{KS}}(T,\mca{F},\mca{P}):=
\varprojlim_k \varinjlim_j 
{\mrm{KS}}(T/\mf{m}^kT,\mca{F},\mca{P} \cap \mca{P}_j).$$
Note that $\overline{\mrm{KS}}(T,\mca{F},\mca{P})$ 
matches Euler system arguments better 
than ${\mrm{KS}}(T,\mca{F},\mca{P})$. 
(See Theorem \ref{ESKS}.)
\end{dfn}

\begin{rem}
We have a natural homomorphism
\[
\xymatrix{
{\mrm{KS}}(T,\mca{F},\mca{P})
 \ar[r] & 
\overline{\mrm{KS}}(T,\mca{F},\mca{P}).
}
\]
of $R$-modules. 
It may not be
either injective or surjective in general. 
(See \cite{MR} p.\ 21.)
But, later, it will turns out that
this homomorphism is an isomorphism 
in the special case which we treat in our paper
(cf.\ Proposition \ref{KS-barKS}.)
\end{rem}

\subsection{}
In this subsection, 
we review the notation on Selmer structures 
over complete discrete valuation rings 
and Iwasawa algebras. 

Now, we set a Selmer structure over 
complete discrete valuation rings.
Let $R$ be 
the integer ring of a finite extension field of $\bb{Q}_p$, 
and $T$ a free $R$-module of finite rank with 
continuous $G_{\bb{Q}}$-action unramified at all  
but finitely many primes. 
Throughout this subsection, we assume that
$T$ satisfies conditions (H.0)-(H.4) 
in \cite{MR} \S 3.5. 
We set the local condition $\mca{F}_{\mathrm{can}}$ for $T$ by
\begin{equation*}
H^1_{\mca{F}_{\mathrm{can}}}(\bb{Q}_v, T):=
\begin{cases}
H^1_f(\bb{Q}_v,T) 
& \text{if $v$ is a finite place prime to p;} \\
H^1(\bb{Q}_p,T) & \text{if $v=p$;} \\
H^1(\bb{R},T)=0 & \text{if $v = \infty$,}
\end{cases}
\end{equation*}
where for finite places $v$ prime to $p$,
the local condition $H^1_f(\bb{Q}_v,T)$ is in the sense of 
Bloch-Kato. Namely, we define 
\[
H^1_f(\bb{Q}_v,T) = 
\Ker\big( \xymatrix{H^1(\bb{Q}_v,T) \ar[r] & 
H^1(\bb{Q}^{\mathrm{unr}}_v,T \otimes_{\bb{Z}_p}\bb{Q}_p)} \big).
\] 
Note that the triple $(T,\mca{F},\mca{P}_1)$ satisfies 
all of the conditions (H.0)-(H.6) in \cite{MR} \S 3.5.

Next, we set a Selmer structure over Iwasawa algebra. 
Let $R$ be the integer ring of a finite extension field of $\bb{Q}_p$, 
and $(T,\mca{F},\Sigma)$ a Selmer structure over $R$. 
We define the $R[[\Gamma]]$-module $\mathbf{T}$ 
with a continuous $R[[\Gamma]]$-linear $G_{\bb{Q}}$-action by
$$\mathbf{T}:=T\otimes_{R}R[[\Gamma]],$$
and the local condition $\mca{F}_{\Lambda}$ on $\mathbf{T}$ by
\begin{equation*}
H^1_{\mca{F}_{\Lambda}}(\bb{Q}_v, \mathbf{T}):=
H^1(\bb{Q}_v,\mathbf{T}). 
\end{equation*}
Let $\bb{Q}_{\Sigma}$ be the maximal extension of $\bb{Q}$ 
unramified outside $\Sigma$. 
Then, by standard arguments, we have 
$$H^1(\bb{Q},\mathbf{T})
=H^1(\bb{Q}_{\Sigma}/\bb{Q},\mathbf{T})
\simeq 
\plim H^1(\bb{Q}_{\Sigma}/\bb{Q}_n,{T}),
$$
where the limit in the most right term is defined 
by the projective system with respect to corestriction maps. 
(See, for example \cite{MR} Lemma 5.3.1.) 
In particular, the local condition $\mca{F}_{\Lambda}$ 
coincides the finite local condition $f$. 
So, the triple $(\mathbf{T},\mca{F}_{\Lambda},\Sigma)$ 
is a Selmer structure.

In the last of this subsection, 
we recall the relation between Euler systems and 
Kolyvagin systems. 
Let $R$ be the integer ring of a finite extension field of $\bb{Q}_p$, 
and $({T},\mca{F},\mca{P})$ a Selmer triple over $R$. 
We define the field $\mca{K}$ to be the composite field of 
$\bb{Q}_{\infty}$ and $\bb{Q}(\mu_n)$ 
for all positive integers $n$ 
satisfying $(n,\ell)=1$ for any prime number $\ell \in \Sigma$.
Let $\mf{N}_{\Sigma}$ be the ideal of $\bb{Z}$ 
defined by the square free product of the prime numbers 
contained in $\Sigma$.  
We denote the set of 
Euler systems for $(T,\mca{K}/\bb{Q},\mf{N}_{\Sigma})$
in the sense of \cite{Ru5} 
by $\mathrm{ES}(T,\Sigma)$. 

\begin{thm}[Theorem 3.2.4 in \cite{MR}]\label{ESKS}
Let $({T},\mca{F},\mca{P})$ be a Selmer triple satisfying 
the conditions {\em (H.0)-(H.6)} in \cite{MR} \S 3.5. 
Assume the following two conditions:
\begin{itemize}
\item the R-module $T/(\mathrm{Fr}_\ell-1)T$ is cyclic
for any $\ell \in \mca{P}$; 
\item the action of $\mrm{Fr}_\ell^{p^k}-1$ on $T$
is injective 
for any $\ell \in \mca{P}$ and 
any $k \in \bb{Z}_{\ge 0}$. 
\end{itemize} 
Then, there exists a $R$-linear map 
$$\mathrm{ES}(T,\Sigma) \longrightarrow 
\overline{\mrm{KS}}(T,\mca{F},\mca{P}); \ \ 
c=\{c_F \}_{F \subseteq \mca{K}} \longmapsto
\kappa(c):=\{ \kappa(c)_n \}_n$$
satisfying $\kappa(c)_1=c_{\bb{Q}}$ in $H^1(\bb{Q},T)$. 
This map is constructed by using Kolyvagin derivative classes of 
Euler systems. For the detail of the construction of this map, 
see \cite{MR} Appendix A.
\end{thm}

\subsection{}
In this subsection, 
we recall some results
(specialized for our purpose)
on Kolyvagin system 
proved in \cite{MR} \S\S 5.2-5.3.  
(In \cite{MR} \S 5.3, they treat only the case
of $\mca{O}_\chi=\bb{Z}_p$, 
but we can prove similar results for general $\mca{O}_\chi$ 
by similar arguments.)

Recall that $\mca{O}$ is a $\bb{Z}_p$-algebra isomorphic to 
$\mca{O}_\chi$ with trivial $G_{\bb{Q}}$-action. 
We define a $\mca{O}$-module $T_\chi$ by 
$$T_\chi:=\bb{Z}_p(1)\otimes_{\bb{Z}_p}\mca{O}_{\chi^{-1}}.$$
We define a set $\Sigma$ of places of $\bb{Q}$ by
$$\Sigma:=\{p,\infty \}  \cup  
\{\ell \mathrel{|} \text{$\ell$ ramifies in $K/\bb{Q}$}  \},$$
and consider the Selmer structure 
$(T_\chi,\mca{F}_{\mrm{can}},\Sigma)$. 
We fix $\mca{P}:=\mca{P}_1$. 
Note that 
by Kummer theory, we have
$$H^1(\bb{Q},T/p^NT)
=(F_0^\times/p^N)_\chi$$
for any positive integer $N$.
Since we assume $\chi(p)\ne 1$,
we have 
$$H^1(\bb{Q}_p,T_\chi^*)=\big( 
\bigoplus_{\lambda|p}\Hom(G_{F_{0,\lambda}},\bb{Q}_p/\bb{Z}_p)
\big)_{\chi^{-1}}=0, $$ 
so by global class field theory, we have 
\begin{align*}
& H^1_{\mca{F}^*_{\mrm{can}}} (\bb{Q},T_\chi^*) \\
&=\Ker\bigg(
\Hom(G_{F_{0}},\bb{Q}_p/\bb{Z}_p))_{\chi^{-1}}
\longrightarrow 
\prod_{\ell \neq p,\infty}
\big( \bigoplus_{\lambda|\ell}
\Hom(G_{F_{0,\lambda}^{\mrm{unr}}},
\bb{Q}_p/\bb{Z}_p)\big)_{\chi^{-1}}\bigg) \\
& \hspace{6mm}\cap \Ker\bigg(
\Hom(G_{F_{0}},\bb{Q}_p/\bb{Z}_p))_{\chi^{-1}}
\longrightarrow 
\big( \bigoplus_{\lambda|p}
\Hom(G_{F_{0,\lambda}},
\bb{Q}_p/\bb{Z}_p)\big)_{\chi^{-1}}\bigg) \\
&=\Ker\bigg(
\Hom(G_{F_{0}},\bb{Q}_p/\bb{Z}_p))_{\chi^{-1}}
\longrightarrow 
\prod_{\ell \neq \infty}
\big( \bigoplus_{\lambda|\ell}
\Hom(G_{F_{0,\lambda}^{\mrm{unr}}},
\bb{Q}_p/\bb{Z}_p)\big)_{\chi^{-1}}\bigg) \\
&=\Hom(A_{F_0,\chi},\bb{Q}_p/\bb{Z}_p ). 
\end{align*}
Recall we define $\Lambda_\chi:=\mca{O}_\chi[[\Gamma_0]]$. 
We define 
$$\mathbf{T}_\chi:=T_\chi 
\otimes_{\mca{O}}\mca{O}[[\Gamma]]
=\Lambda_{\chi^{-1}}\otimes_{\bb{Z}_p}\bb{Z}_p(1) .$$
Note that by local duality, 
we have $H^2(\bb{Q}_p,T_\chi^*)=0$, 
and by (the limit of) Poitou-Tate exact sequence, we obtain
\[
X_\chi=
\Hom(H^1_{\mca{F}_\Lambda^*}(\bb{Q},T^*),\bb{Q}_p/\bb{Z}_p)
=H^2(\bb{Q}_{\Sigma}/\bb{Q},\mathbf{T}).
\]

We denote define a set $\Sigma_{\Lambda}$
of hight one prime ideals of the ring 
$\mca{O}[[\Gamma]]=\Lambda_\chi$ by 
\[
\Sigma_\Lambda:=\big\{\mf{P} 
\in \mathrm{Spec}(\Lambda_\chi)
 \mathrel{|} 
\mathrm{ht}\mf{P}=1\ \text{and} \ 
\cha_{\Lambda_\chi}(X_\chi) \subseteq\mf{P} \big\}
\cup 
\{p \Lambda_\chi  \}.
\]
Let $\mf{Q} \subseteq \mca{O}[[\Gamma]]=\Lambda_\chi$ 
be a prime ideal of hight one not contained in $\Sigma_{\Lambda}$. 
We denote the integral closure of $\mca{O}[[\Gamma]]/\mf{Q}$
by $S_{\mf{Q}}$. 
Note that $S_{\mf{Q}}$ is 
a complete discrete valuation ring of 
mixed characteristic $(0,p)$ 
whose residue field is finite. 
We fix a uniformizer of $S_{\mf{P}}$ by $\pi$.

Here, let us consider the Selmer triple 
\[
\big(
\mathbf{T}_{\chi}\otimes_{\mca{O}[[\Gamma]]}S_{\mf{Q}}=
T_\chi\otimes_{\mca{O}}S_{\mf{Q}},\mca{F}_{\mrm{can}},
\mca{P}:=\mca{P}_1
\big).
\]
Note that this Selmer triple satisfies conditions (H.0)-(H.6) 
in \cite{MR} \S 3.5. 
We have the following lemma. 
\begin{lem}[\cite{MR} Lemma 5.3.16]\label{corerank}
Let $\mf{Q} \subseteq \mca{O}[[\Gamma]]$ 
be a prime ideal of hight one not contained in $\Sigma_{\Lambda}$. 
Then, we have 
$$\chi(T_\chi\otimes_{\mca{O}}S_{\mf{Q}})
=\mathrm{rank}_{\mca{O}}(T_\chi^-)=1$$
where $\chi(T_\chi\otimes_{\mca{O}}S_{\mf{Q}})$ is the core rank of 
$T_\chi\otimes_{\mca{O}}S_{\mf{Q}}$ in the sense of 
\cite{MR} Definition 5.2.4. 
\end{lem}

\begin{rem}
Note that the original statement of Lemma 5.3.16 in \cite{MR} 
treats only the case of $\mca{O}=\bb{Z}_p$, 
but the similar proof works for general $\mca{O}$.
\end{rem}

By Lemma \ref{corerank} and \cite{MR} Proposition 5.2.9, 
we obtain the following corollary.
\begin{cor}\label{KS-barKS}
Let $\mf{Q} \subseteq \mca{O}[[\Gamma]]$ 
be a prime ideal of hight one not contained in $\Sigma_{\Lambda}$. 
Then, the natural homomorphism 
$$\xymatrix{{\mrm{KS}}(T_\chi\otimes_{\mca{O}}S_{\mf{Q}},
\mca{F}_{\mrm{can}},\mca{P})
 \ar[r] & \overline{\mrm{KS}}(T_\chi\otimes_{\mca{O}}S_{\mf{Q}},
\mca{F}_{\mrm{can}},\mca{P})}$$
is an isomorphism. 
\end{cor}

For each positive integer $N$,
we define 
\[
\mca{S}_{N}(\mf{Q}):=\{\ell \in \mca{S}_N \mathrm{|} 
\text{ $\mathrm{Fr}_\ell$ acts trivially on 
$T\otimes_{\mca{O}}
S_{\mf{Q}}/p^NS_{\mf{Q}} $} \},
\] 
and we denote the set of all 
well-ordered square-free products of 
$\mca{S}_{N}(\mf{Q})$
by $\mca{N}_{N}^{\mathrm{w.o.}}(\mf{Q})$. 
Note that we have 
$\mca{S}_{N'} \subseteq \mca{S}_{N}(\mf{Q})$
for any sufficiently large $N'$. 
Consider the $R$-linear map 
\[
\mathrm{ES}(T_\chi\otimes_{\mca{O}}S_{\mf{Q}},
\mca{F}_{\mrm{can}},\mca{P})
 \longrightarrow 
\overline{\mrm{KS}}(T_\chi\otimes_{\mca{O}}S_{\mf{Q}},
\mca{F}_{\mrm{can}},\mca{P})=
{\mrm{KS}}(T_\chi\otimes_{\mca{O}}S_{\mf{Q}},
\mca{F}_{\mrm{can}},\mca{P})
\]
in Theorem \ref{ESKS}. 
Then, by construction of this map (cf.\ \cite{MR} pp.\ 80-81.), 
we obtain the following proposition.

\begin{prop}\label{Comparisonofkappa}
Take $c \in \mrm{ES}(T_\chi\otimes_{\mca{O}}S_{\mf{Q}},
\mca{K}/\bb{Q},\mca{P})$ and 
$n \in \mca{N}_{N}^{\mathrm{w.o.}}(\mf{Q})$. 
Let $\kappa_{0,N}(c;n)$ be the Kolyvagin derivative class of $c$ 
at $n$. 
Then, we have $$\kappa(c)_n=\kappa_{0,N}(c;n)$$ 
in $H^1(\bb{Q}, T_\chi\otimes_{\mca{O}}(S_{\mf{Q}}/p^N))$. 
\end{prop}

Recall that
for each element 
$n:=\ell_1\times \cdots \times \ell_r \in \mca{N}(\mca{P})$,
we denote the number of prime divisors of $n$ by 
$\epsilon(n):=r$. For any non-zero Kolyvagin system 
$\kappa=\{\kappa_n \} \in \mathrm{KS}
(T_\chi S_{\mf{Q}},\mca{F}_{\mrm{can}},\mca{P})$, 
we define
$$\partial_{i}(\kappa;\mf{Q}):=\max\{j \in \bb{Z}_{\ge 0} \mathrel{|}
\pi^{j}H_{\mca{F}(n)}(\bb{Q},T_\chi\otimes_{\mca{O}}
S_{\mf{Q}}/I_n) \text{ for all 
$n \in \mca{N}(\mca{P})$ with $\epsilon(n)=i$} \} $$
for any non-negative integer $i$. 
We also define 
\[
\partial_{i}({\mf{Q}}):=\min\{\partial_{i}(\kappa;\mf{Q}) 
\mathrel{|}
\kappa=\{\kappa_n \} \in \mathrm{KS}
(T_\chi\otimes_{\mca{O}}S_{\mf{Q}},\mca{F}_{\mrm{can}},\mca{P})\}. 
\]
Note that we have 
$e_i({\mf{Q}}):=\partial_{i}(\mf{Q}) - \partial_{(i+1)}(\mf{Q}) 
\ge 0$ for any 
$i \in \bb{Z}_{\ge 0}$, and 
$\partial_{i}(\mf{Q})=0$ for sufficiently large $i$. 
Note that the core rank of $T_\chi\otimes_{\mca{O}}S_{\mf{Q}}$
is one by Lemma \ref{corerank}, and 
we can check the Selmer triple 
$(T_\chi\otimes_{\mca{O}}S_{\mf{Q}},\mca{F}_{\mrm{can}},\mca{P})$
satisfies all conditions required in \cite{MR} Theorem 5.2.12,
we obtain the following result.

\begin{thm}[a special case of \cite{MR} Theorem 5.2.12]\label{MR5.2.12}
Let $\mf{Q} \subseteq \mca{O}[[\Gamma]]$ 
be a prime ideal of hight one not contained in $\Sigma_{\Lambda}$. 
We fix a uniformizer $\pi_{\mf{Q}}$ of $S_{\mf{Q}}$. 
Then, we have an isomorphism
$$H^1_{\mca{F}_\Lambda^*}(\bb{Q},
(T_\chi\otimes_{\mca{O}}S_{\mf{Q}})^*) \simeq 
\bigoplus_{i \ge 0}S_{\mf{Q}}/\pi_{\mf{Q}}^{e_i(\mf{Q})}S_{\mf{Q}} $$
of $S_{\mf{P}}$-modules.
In other words, we have 
\[
\Fitt_{S_{\mf{Q}},i}\big( 
H^1_{\mca{F}_{\mrm{can}}^*}(\bb{Q},
(T_\chi\otimes_{\mca{O}}S_{\mf{Q}})^*) \big)
=\pi_{\mf{Q}}^{\partial_{i}({\mf{Q}})}S_{\mf{Q}}
\]
for any $i \in \bb{Z}_{\ge 0}$. 
\end{thm}

\subsection{} 
In this section, we treat the results 
on higher Fitting ideals of $A_{0,\chi}$. 
By Theorem \ref{MR5.2.12} and 
usual Euler system arguments (without Kurihara's elements), 
we obtain the following theorem. 

\begin{thm}\label{0-layer}
Assume the extension degree of 
$K/\bb{Q}$ is prime to $p$. 
Let $\chi \in \widehat{\Delta}$ be a character 
satisfying $\chi(p)\ne 1$. 
Then, we have
\[
\Fitt_{\mca{O}_\chi/p^N,i}(A_{0,\chi})
=\mf{C}_{i,0,N,\chi}
\]
for any non-negative integer $i$ and 
sufficiently large integer $N$. 
\end{thm}

\begin{proof}
The order of $\chi$ is prime to $p$, 
so $p$ is a prime element of 
the discrete valuation ring $\mca{O}_\chi$. 
Since $A_{0,\chi}$ is a finitely generated 
torsion $\mca{O}_{\chi}$-module, 
we have an exact sequence
\begin{align*}
\xymatrix{0 \ar[r] & \mca{O}^r_\chi \ar[r]^{f} 
& \mca{O}^r_\chi \ar[r]^{g} & {A_{0,\chi}} \ar[r] & 0,} 
\end{align*}
of $\mca{O}_\chi$-modules, 
where the matrix $M_f$ associated to 
$f$ for the standard basis $(\mathbf{e}_j)_{j=1}^r$ 
of $\mca{O}_\chi^r$ is 
a diagonal matrix 
\begin{equation*}
M_f:=\begin{pmatrix}
p^{d_1} & & & & \\
& p^{d_2} & & & \\
& & & \ddots & \\
& & & & p^{d_r}
\end{pmatrix}
\end{equation*}
satisfying $d_1\ge d_2 \ge \cdots \ge d_r$. 

We fix an integer $N$ satisfying $p^N \ge \# A_{0,\chi}$. 
First, let us show the inequality 
$\Fitt_{\mca{O}_\chi/p^N,i}(A_{0,\chi})
\supseteq \mf{C}_{i,0,N,\chi}$. 
Let 
\[
\eta=\{\eta_m(n)_\chi \in 
(F_m(\mu_n)^\times \otimes_{\bb{Z}} \bb{Z}_p)_\chi \}_{m,n}
\] 
be an Euler system of circular units  
defined by a 
$\Lambda_\chi$-linear combination of 
basic circular units, 
$n \in \mca{N}_N^{\mrm{w.o.}}$ satisfying 
$\epsilon (n) \le i$, 
and 
\[
\xymatrix{f \colon 
(F_0^\times /p^N)_\chi \ar[r] & 
R_{0,N,\chi}=\mca{O}_\chi/p^N}
\]
an arbitrary homomorphism of 
$R_{0,N,\chi}$-modules. 
Then, from Theorem \ref{MR5.2.12} 
for the prime ideal $(\gamma-1)\Lambda_\chi$ and
Proposition \ref{Comparisonofkappa},
it follows that 
$\kappa_{0,N,\chi}(n,\eta)_\chi$ 
is a $p^{\sum_{j=i+1}^{r} d_j}$-power 
of some element in $(F_0^\times /p^N)_\chi$. 
This implies 
\[
f(\kappa_{0,N}(\eta;n)_\chi) \in 
p^{\sum_{j=i+1}^{r} d_j}(\mca{O}_\chi/p^N)=
\Fitt_{\mca{O}_\chi,i}(A_{0,\chi})(\mca{O}_\chi/p^N),
\]
and we obtain $\Fitt_{\mca{O}_\chi/p^N,i}(A_{0,\chi})
\supseteq \mf{C}_{i,0,N,\chi}$. 

Note that the inequality 
$\Fitt_{\mca{O}_\chi,i}(A_{0,\chi})
\subseteq \mf{C}_{i,0,N,\chi}$
follows from the usual Euler system argument. 
(See, for example, the arguments in \cite{Ru2} \S 4.)
We sketch the proof of this inequality briefly.
Let $N$ be a sufficiently integer. 
Note that any circular unit in $F_0$ extends to 
an Euler system defined by a 
$\Lambda_\chi$-linear combination of 
basic circular units 
since we assume $\chi(p) \ne 1$
(cf.\ Remark \ref{extending to Euler system}).
Fix an Euler system
\[
\eta=\{\eta_m(n)_\chi \in 
(F_m(\mu_n)^\times \otimes_{\bb{Z}} \bb{Z}_p)_\chi \}_{m,n}
\]
circular units  
defined by a 
$\Lambda_\chi$-linear combination of 
basic circular units, and assume that
the circular unit $\eta_0(1)$ generates 
the free $\mca{O}_\chi$-module 
$C_{0,\chi}= (C_{0}\otimes \bb{Z}_p)_\chi$ of rank one. 
Recall that 
$E_{0,\chi}:=
(\mca{O}_{F_0}^\times \otimes \bb{Z}_p)_\chi$
is a free $\mca{O}_\chi$-module of rank one.  
We fix an isomorphism 
\[
\xymatrix{
\psi_0\colon E_{0,\chi}/p^N \ar[r] & 
R_{0,N,\chi}=\mca{O}_\chi/p^N 
}
\]
of $\mca{O}_\chi$-modules and 
a prime number $\ell_1$ whose 
ideal class $[\ell_1]_{F_0,\chi}$ in $A_{0,\chi}$ 
coincides with $g(\mathbf{e}_{1})$ and 
satisfying 
\[
\phi_{m,N,\chi}^{\ell_1}\mid_{E_{0,\chi}/p^N}
=\psi_0.
\]
(Note that Proposition \ref{chebo appli} ensures the existence 
of such a prime number $\ell_1$.)
By the arguments in \cite{Ru2} \S 4 combined 
with Proposition \ref{chebo appli}, 
we can inductively take prime numbers 
$\ell_1,...,\ell_{r+1} \in \mca{S}_N$
homomorphisms
\[
\xymatrix{
\psi_j\colon (F_{0}^\times/p^N)_\chi \ar[r] & 
R_{0,N,\chi}=\mca{O}_\chi/p^N}  
\hspace{5mm}
(j=1,...,r)
\]
satisfying the following conditions. 
\begin{itemize}
\item $[\ell_{j,F_0}]_{\chi}=g(\mathbf{e}_{j})$ in $A_{0,\chi}$  
for any integer $j$ with $1 \le j \le r$. 
\item The integer $n_j:=\prod_{\nu=1}^j\ell_j$ is well-ordered
any integer $j$ with $1 \le j \le r$. 
\item $p^{d_{j-1}}\psi_j(\kappa_{0,N}(\eta;n_j))
=\psi_{j-1}(\kappa_{0,N}(\eta;n_{j-1})_\chi)$
for any integer $j$ with $1 \le j \le r$. 
Here, we put $n_0:=1$. 
\item The  restriction of $\phi_{0,N,\chi}^{\ell_j}$
on $\mca{W}_{0,N.\chi}(n_j)$ coincides with $\psi_{j-1}$
any integer $j$ with $1 \le j \le r$. 
\end{itemize}
Then, we obtain 
\[
p^{\sum_{j=1}^{i-1} d_j}\psi_i(\kappa_{0,N}(\eta;n_i)_\chi)
=p^{\sum_{j=1}^{i-2} d_j}\psi_{i-1}(\kappa_{0,N}(\eta;n_{i-1})_\chi)
=\cdots=\psi_0(\eta).
\]
By \cite{Ru2} Theorem 4.2 (see \cite{MW} Theorem 1.10.1 or 
\cite{Ru5} Corollary 3.2.4 for general cases), 
there exists a unit $u \in \mca{O}_\chi^{\times}$
such that 
\[
\psi_0(\eta)=u\# A_{0,\chi}
=up^{\sum_{j=1}^r d_j}.
\]
Therefore, we obtain 
\[
\mf{C}_{0,N,\chi}\subseteq 
\psi_i(\kappa_{0,N}(\eta;n_i)_\chi)R_{0,N,\chi}
=p^{\sum_{j=i+1}^{r} d_j}R_{0,N,\chi}.
\]
This completes the proof. 
\end{proof}

\begin{rem}\label{usual ES arguments}
Fix a pseudo-isomorphism 
\[
\xymatrix{
X_\chi \ar[r] & 
\bigoplus_{j=1}^r \Lambda_\chi/f_i\Lambda_\chi,
}
\]
where $f_1,...,f_r$ are non-unit and non-zero elements of $\Lambda_\chi$
satisfying $f_r|\cdots |f_2|f_1$. 
Then, by similar argument to the proof of 
the inequality 
$\Fitt_{\mca{O}_\chi,i}(A_{0,\chi})
\subseteq \mf{C}_{i,0,N,\chi}$, 
we can prove  
rough estimates
\begin{equation}\label{rough estimates}
\Fitt_{\Lambda_\chi,i}(X_\chi) \prec 
\mf{C}_{i,\chi}
\end{equation}
without using Kurihara's elements. 
In this argument,  
we have to use the argument of \cite{Ru2} \S 5 and 
the Iwasawa main conjecture instead of 
the argument of \cite{Ru2} \S 4 and 
\cite{Ru2} Theorem 4.2. 
Note that when we apply 
such arguments without Kurihara's elements, 
we have to ignore error factors 
(\ref{rough estimates}) completely. 
So, Theorem \ref{Main theorem}, 
which is proved by 
Euler arguments via Kurihara's elements, 
is stronger than results obtained 
by usual arguments without Kurihara's elements. 
\end{rem}

Note that Theorem \ref{0-layer}
implies that 
for any non-negative integer $i$ and
any two integers $N$ and $N'$
satisfying $N' \ge N >0$, 
the image of 
$\mf{C}_{0,N',\chi}$ in 
$R_{0,N,\chi}$ coincides with 
$\mf{C}_{0,N,\chi}$. 
Combining this fact and the second assertion of 
Corollary \ref{compatibility of cyclotomic ideals}, 
we obtain the following corollary immediately. 

\begin{cor}\label{cyclotomic ideals and reduction}
Assume the extension degree of 
$K/\bb{Q}$ is prime to $p$. 
Let $i$ be a non-negative integer, and
$\chi \in \widehat{\Delta}$ a character 
satisfying $\chi(p)\ne 1$.  
Then, the following holds. 
\begin{enumerate}
\item The image of $\mf{C}_{i,\chi}$ in 
$R_{0,N,\chi}$ coincides with the ideal 
$\mf{C}_{i,0,N,\chi}$
for any positive integer $N$. 
\item The image of $\mf{C}_{i,\chi}$ in 
$R_{0,\chi}:=\bb{Z}_p[\Gal(F_0/\bb{Q})]_\chi$
coincides with the ideal 
$\mf{C}_{i,F_0,\chi}:= \varprojlim_{N} \mf{C}_{i,0,N,\chi}$. 
\end{enumerate}
\end{cor}

We put $\mf{m}:=p\Lambda_\chi+(\gamma-1)\Lambda_\chi$. 
Note we have the natural isomorphism 
$$X_\chi/\mf{m}X_\chi\simeq A_{0,\chi}/p$$ 
by Proposition \ref{class group, refined}. 
So, the least cardinality of generators of 
the $\Lambda_\chi$-module $X_\chi$ 
coincides to that of 
the $\mca{O}_\chi/p$-module $A_{0,\chi}$
by Nakayama's lemma. 
Hence the following corollary follows from 
Remark \ref{remark Fitt and generators}, 
Theorem \ref{0-layer} and 
Corollary \ref{cyclotomic ideals and reduction}.

\begin{cor}\label{generator and cyclotomic ideals}
Let $K/\bf{Q}$ and $\chi  \in \widehat \Delta$ be as 
in Theorem \ref{Main theorem, Rough}.
Let $r$ be a non-negative integer. 
Then, the following two properties are equivalent. 
\begin{itemize}
\item[(1)] The least cardinality of generators of 
the $\Lambda_\chi$-module $X_\chi$ is $r$. 
\item[(2)] $\mf{C}_{r-1,\chi} \ne \Lambda_\chi$ 
and $\mf{C}_{r,\chi} = \Lambda_\chi$. 
\end{itemize}
\end{cor}

\begin{exa}
In general, the computation of the higher cyclotomic ideals
$\mf{C}_{i,\chi}$ is  hard. 
But in a certain very special case, we determine 
the higher cyclotomic ideals explicitly and prove that 
they coincides with the higher Fitting ideals. 
Let $p=3$ and $K:=\bb{Q}(\sqrt{257})$. 
Then, we have $F_0=K$. 
We take a unique non-trivial character 
$\chi \in \widehat{\Delta}$. 
In this case, Greenberg proved that 
$A_{n,\chi}$ is a cyclic group of order 3
for any $n \ge 0$. (See \cite{Gree2} \S 7.)
So, we have 
$X_\chi=X_{\chi,\mrm{fin}}\simeq \Lambda_\chi/(3,\gamma-1)$
by Proposition \ref{class group, refined}, 
and we obtain $\Fitt_{\Lambda_\chi,0}(X_\chi)=(3,\gamma-1)$ 
and $\Fitt_{\Lambda_\chi,i}(X_\chi)=\Lambda_\chi$ for $i \ge 1$.
Note that $\gamma-1$ annihilates $X_{\chi,\mrm{fin}}$, 
so Corollary \ref{half of Main theorem, Rough}
implies that the element $\gamma-1$ belongs to 
$\mf{C}_{i,\chi}$ for any $i \ge 0$. 
Since both $\Fitt_{\Lambda_\chi,i}(X_\chi)$ and 
$\mf{C}_{i,\chi}$  
contain the ideal $(\gamma-1)\Lambda_\chi$
for any $i \ge 0$, 
we deduce that
$$\mf{C}_{i,\chi}=\Fitt_{\Lambda_\chi,i}(X_\chi)=
\begin{cases}
(3,\gamma-1) & \text{if $i=0$} \\
\Lambda_\chi & \text{if $i>0$}
\end{cases}
$$
from Theorem \ref{0-layer} and 
Corollary \ref{cyclotomic ideals and reduction}.
\end{exa}

\subsection{}
In this subsection,  
we prove Theorem \ref{local lower bounds}.
Here, we fix a positive integer $i$ 
and a height one prime ideal $\mf{P}$ of $\mca{O}[[\Gamma]]$
containing $\Fitt_{\Lambda_\chi}(X_\chi)$. 
In particular, we have $\mf{P} \ne (p)$. 
For simplicity, we put 
$\alpha:=\alpha_{i}(\mf{P})$ and 
$\beta:=\beta_i(\mf{P})$. 
We define a non-negative integer $s$ 
by 
\[
p^s=(S_{\mf{P}}:\mca{O}[[\Gamma]]/\mf{P}).
\]

We regard $\mca{O}[[\Gamma]]$ as the ring $\mca{O}[[T]]$ 
of formal power series by the isomorphism 
$\mca{O}[[\Gamma]] \simeq \mca{O}[[T]]$ defined by 
$\gamma \mapsto 1+T$. 
Let $f(T) \in \mca{O}_\chi[T]$ be 
the Weierstrass polynomial generating 
the fixed prime ideal $\mf{P}$ of 
$\mca{O}[[\Gamma]]= \mca{O}_\chi[[T]]$.
For any positive integer$k$, 
we put $$f_k(T):=f(T)+p^k$$ and let 
$\mf{P}_k$ be the principal ideal of 
$\Lambda_{\chi}= \mca{O}_\chi[[T]]$ generated by $f_k(T)$. 
We need the following lemma (cf.\ \cite{MR} p.\ 66). 

\begin{lem}
There exists a positive integer $N(\mf{P})$ satisfying 
the following properties. 
\begin{enumerate}
\item The ideal $\mf{P}_k$ is a prime for any $k \ge N(\mf{P})$.
\item The ideal $\mf{P}_k$ is not contained in $\Sigma_{\Lambda}$ 
for any $k \ge N(\mf{P})$. 
\item The residue ring $\mca{O}[[\Gamma]]/\mf{P}_k$ 
is (non-canonically)
isomorphic to $\mca{O}[[\Gamma]]/\mf{P}$ as $\mca{O}$-algebra for 
any $k \ge N(\mf{P})$. 
\item The action of $\mrm{Fr}_\ell^{p^m}-1$ on 
$T_\chi\otimes_{\mca{O}}S_{\mf{P}_k}$
is injective 
for any $\ell \in \mca{P}$, 
any $m \ge 0$ and any $k \ge N(\mf{P})$. 
\end{enumerate}
\end{lem}

\begin{proof}
The arguments 
in \cite{MR} p.\ 66 implies 
that there exists an integer $N'(\mf{P})$ such that 
the conditions (1)-(3) 
in the lemma holds for 
any integer $k$ satisfying $k \ge N'(\mf{P})$.
(See loc.\ cit.\ for detail.)
So, it is sufficient to show that the fourth condition holds 
for any sufficiently large integer $k$. 
We denote the cyclotomic character by 
$$\xymatrix{\chi^{\mrm{cyc}} \colon 
\Gamma \ar[r] & 1+p\bb{Z}_p \ar@{^{(}->}[r] &  
\mca{O}^\times.  }$$
Let $k$ be an integer satisfying $k \ge N(\mf{P})$. 
The natural projection induces 
a continuous character
$$\xymatrix{\rho_{k}\colon 
\Gamma \ar[r] &
(\mca{O}[[\Gamma]]/\mf{P}_k)^\times. }$$
Note that the action of $\mrm{Fr}_\ell^{p^m}-1$ on 
$T_\chi\otimes_{\mca{O}}S_{\mf{P}_k}$
is {\em not} injective 
for some $\ell \in \mca{P}$ and 
some $m \ge 0$ if and only if 
the order of the character 
$\chi^{\mrm{cyc}}\rho_k$ is finite. 

We denote the order of $p$-power torsion part 
of $(\mca{O}[[\Gamma]]/\mf{P}_k)^\times$ by $p^{\nu}$. 
Assume that 
the character 
$\chi^{\mrm{cyc}}\rho_k$ has finite order.
The image of 
$\chi^{\mrm{cyc}}\rho_k$ is contained in 
$(\mca{O}[[\Gamma]]/\mf{P}_k)^\times$, 
so the character $\chi^{\mrm{cyc}}\rho_k$ is annihilated by $p^{\nu}$. 
In particular, we have 
$$\rho_k(\gamma^{p^\nu})=(\chi^{\mrm{cyc}})^{-p^{\nu}}(\gamma)$$
in $\mca{O}[[\Gamma]]/\mf{P}_k$. 
This implies the polynomial 
$$(1+T)^{p^{\nu}}-(\chi^{\mrm{cyc}})^{-p^{\nu}}(\gamma)
\in \mca{O}[T]$$ 
is divisible by the monic polynomial $f_k(T)$. 
Obviously, such a situation occurs for only finitely many $k$, so 
the condition (4) holds 
for any sufficiently large integer $k$.
\end{proof}

\begin{dfn}
Let $M$ be an integer, and 
$\{x_k \}_{k \in \bb{Z}_{\ge M}}$ and 
$\{y_k \}_{k \in \subset \bb{Z}_{\ge M}}$  
sequences of real numbers. 
We write $x_k \succ y_k$ 
if and only if 
$\liminf_{k\to \infty}(x_k-y_k) \ne -\infty$. 
We write $x_N \sim y_N$ if and only if 
$x_k \succ y_k$ and $y_k \succ x_k$. 
In other words, 
we write $x_k \sim y_k$ if and only if 
$\left| x_k-y_k\right|$ is bounded independent of $N$.
\end{dfn}

We denote the ramification index of 
$\mathrm{Frac}(S_{\mf{P}})/\bb{Q}_p$ by $e_{\mf{P}}$, 
and the extension degree of the residue field of $\mca{O}_\chi$ 
over $\bb{F}_p$ by $f_{\chi}$. 
Let us recall the observations in \cite{MR} p.\ 66. 
Let $d$ be a non-negative integer. 
Then, we have 
\[
\mf{P}_k+\mf{P}^d
=\mf{P}_k+p^{d e_{\mf{P}} k}\mca{O}[[\Gamma]]
\]
for any sufficiently large integer $k$.
So, we obtain the natural isomorphism
\[
(\mca{O}[[\Gamma]]/\mf{P}^{d})\otimes_{\mca{O}[[\Gamma]]} 
S_{\mf{P}_k} 
\simeq S_{\mf{P}_k}/(\pi_k)^{d e_{\mf{P}} k}
\]
of $S_{\mf{P}_k}$-algebras for any sufficiently large $k$. 
Moreover, we obtain the following lemma from 
the observations in \cite{MR} p.\ 66. 
(See \cite{MR} loc.\ cit.\ for the proof.)

\begin{lem}\label{AandaBandb1}
Let $M$ be a finitely generated torsion $\mca{O}[[\Gamma]]$-module, and 
$$E:=\displaystyle\bigoplus_{j=0}^r \mca{O}[[\Gamma]]/\mf{P}^{d_j}
\oplus \bigoplus_{j'=0}^{r'}\mca{O}[[\Gamma]]/(g_{j'}(T)^{e_{j'}}) $$
an elementary $\mca{O}[[\Gamma]]$-module, where 
where $d_j$ and $e_{j'}$ are positive integers, and 
$g_{j'}(T)$ is a Weierstrass polynomial in $\mca{O}_\chi[T]$ 
prime to $f(T)$ for any integer $i$  and $j'$. 
Suppose that 
the $\mca{O}[[\Gamma]]$-module $M$ is pseudo-isomorphic 
to $E$. Then, there exists 
a sequence 
$$\big\{\xymatrix{\iota_k \colon 
M\otimes_{\mca{O}[[\Gamma]]}S_{\mf{P}_k} \ar[r] & 
\bigoplus_{j=0}^r S_{\mf{P}_k}/(\pi_k)^{d_j e_{\mf{P}} k}; 
\text{$S_{\mf{P}_k}$-linear} }
\big\}_{k >N(\mf{P}) }$$
of homomorphisms 
such that the orders of 
the kernel and cokernel of $\iota_k$ 
are finite for any $k \ge N(\mf{P})$, 
and bounded by a constant independent of $k$.
\end{lem}

Then, we immediately obtain 
the following Corollary \ref{AandaBandb} of  
Lemma \ref{AandaBandb1}
combined with Lemma \ref{basicFitting}.
This corollary plays an important role in this section.

\begin{cor}\label{AandaBandb}
Let $M$ be a finitely generated torsion $\mca{O}[[\Gamma]]$-module. 
We define a non-negative integers $C$ by 
$$\Fitt_{\mca{O}[[\Gamma]],i}
(M\otimes_{\mca{O}[[\Gamma]]}\Lambda_{\chi,\mf{P}})
=\mf{P}^C\Lambda_{\chi,\mf{P}}.$$
For each positive integer $k$ with 
$k \ge N(\mf{P})$, 
fix a uniformizer $\pi_k$ of $S_{\mf{P}_k}$, and 
define a non-negative integer $c_k$ by 
$$\Fitt_{S_{\mf{P}},i}
(M\otimes_{\mca{O}[[\Gamma]]}S_{\mf{P}})
=\pi_k^{c_k}S_{\mf{P}}. $$
Then, we have $c_k \sim Ce_{\mf{P}}k$.
\end{cor}

\begin{dfn}
Let $k$ be a non-negative integer with 
$k \ge N(\mf{P})$. 
We define 
define non-negative integers $a_k$ and $b_k$ by
\begin{align*}
\pi_k^{a_k}S_{\mf{P}_k} & =
\Fitt_{S_{\mf{P}_k},i}
(X_\chi \otimes_{\mca{O}[[\Gamma]]} S_{\mf{P}_k} ), \\
b_k &= 
\mrm{length}_{S_{\mf{P}_k}}\big(
(\mca{O}[[\Gamma]]/\mf{C}_{i,\chi})\otimes_{\mca{O}[[\Gamma]]} S_{\mf{P}_k}
\big).
\end{align*}
By Corollary \ref{AandaBandb}, we have 
$a_k \sim \alpha e_{\mf{P}} k$ and $b_k \sim \beta e_{\mf{P}} k$. 
\end{dfn}

\begin{prop}[\cite{MR} Proposition 5.3.14]\label{MR5.3.14}
Let $k$ be an integer satisfying $k \ge N(\mf{P})$, and 
$$\xymatrix{ 
\pi_k \colon X_\chi\otimes_{\mca{O}[[\Gamma]]}S_{\mf{P}_k} \ar[r] & 
\Hom\bigg( 
H^1_{\mca{F}_{\mrm{can}}^*}(\bb{Q},
\big(
T_\chi\otimes_{\mca{O}}S_{\mf{P}_k})^*
\big),
\bb{Q}_p/\bb{Z}_p
\bigg)}$$
a natural homomorphism. 
Then, the kernel and cokernel of $\pi_k$ are both finite, 
and the orders of kernel and cokernel of $\pi_k$ 
are bounded by a constant independent of $k$. 
\end{prop}

Combining Proposition \ref{MR5.3.14} with 
Theorem \ref{MR5.2.12}, 
we obtain the following corollary.  

\begin{cor}\label{akandpartial}
We have $a_k \sim \partial_i({\mf{P}_k})$.
\end{cor}

For an integer $k$ with $k \ge N(\mf{P})$, 
we take an integer $N'_k$ satisfying 
\begin{itemize}
\item $N'_k \ge \partial_i({\mf{P}_k})$, and 
\item $p^{N'_k} \in \mf{C}_{i,\chi}+\mf{P}_k$. 
\end{itemize}
Note that there exist such an $N'_k$ since 
the ideal $\mf{C}_{i,\chi}+\mf{P}_k$ has 
finite index in $\mca{O}[[\Gamma]]$. 
Then, we take $N''_k \ge N'_k$ satisfying 
\begin{itemize}
\item $\gamma^{p^{N''_k-1}}-1 \in 
\mf{P}_k+p^{N''_k}\mca{O}[[\Gamma]]$, and 
\item $\mca{S}_{N''_k} \subseteq 
\mca{S}_{N'_k}(\mf{P}_{k})$. 
\end{itemize}
We put $m_k:=N''_k-1$.

\begin{proof}[Proof of \ref{local lower bounds}]
Now, we shall prove Theorem \ref{local lower bounds}. 
It is sufficient to show 
$Be_{\mf{P}}k \succ Ae_{\mf{P}}k$. 
Let $k$ be any positive integer satisfying $k \ge N(\mf{P})$. 
Then, we have 
\begin{align*}
\beta e_{\mf{P}}k & \sim b_k= \mrm{length}_{S_{\mf{P}_k}}
\big(
(\mca{O}[[\Gamma]]/(\mf{C}_{i,\chi}+\mf{P}_k))
\otimes_{\mca{O}[[\Gamma]]} S_{\mf{P}_k}
\big) \\
&= \mrm{length}_{S_{\mf{P}_k}}\big(
(\mca{O}[[\Gamma]]/(\mf{C}_{i,\chi}+\mf{P}_k+p^{N'_k}\mca{O}[[\Gamma]]))
\otimes_{\mca{O}[[\Gamma]]} S_{\mf{P}_k}
\big) \\
& = \mrm{length}_{S_{\mf{P}_k}}\big(
(\mca{O}[[\Gamma]]/(\mf{C}_{i,\chi}+\mf{P}_k+
(p^{N'_k},\gamma^{p^{m_k}}-1)))\otimes_{\mca{O}[[\Gamma]]} S_{\mf{P}_k}
\big) \\
&=\mrm{length}_{S_{\mf{P}_k}}\big(
(R_{m_k,N'_k,\chi}/(\text{the image of\ }\mf{C}_{i,\chi}))
\otimes_{\mca{O}[[\Gamma]]} S_{\mf{P}_k}
\big) \\
& \ge \mrm{length}_{S_{\mf{P}_k}}\big(
(R_{m_k,N'_k,\chi}/(\text{the image of\ }\mf{C}_{i,m_k,N''_k,\chi}))
\otimes_{\mca{O}[[\Gamma]]} S_{\mf{P}_k}
\big)
\end{align*}
Since the ring $R_{m_k,N'_k,\chi}\otimes_{\mca{O}[[\Gamma]]} S_{\mf{P}_k}$ is 
a quotient of the discrete valuation ring $S_{\mf{P}_k}$, 
the image of $\mf{C}_{i,m_k,N''_k,\chi}$ in 
$R_{m_k,N'_k,\chi}\otimes_{\mca{O}[[\Gamma]]} S_{\mf{P}_k}$
is a principal ideal. 
So, there exist 
\begin{itemize}
\item a circular unit
$$\eta(n_k)=\eta^{(k)}_{m_k}(n_k):=
\prod_{d|\mf{f}_K}\eta_{m_k}^{d}(n)^{u_d}\times
\prod_{i=1}^r \eta_{m_k}^{1,a_i}(n)^{v_i} 
\in F_{m_k}(\mu_{n_k})^\times,$$
where $r \in \bb{Z}_{>0}$, $u_d$ and $v_i$ 
are elements of $\bb{Z}[\Gal(F_m/\bb{Q})]$ 
for each positive integers $d$ and $i$ 
with $d|\mf{f}_K$ and $1 \le i \le r$, 
and $a_1,..., a_r$ are integers prime to $p$, 
\item an element $n_k \in \mca{N}_{N''_k}^{\mrm{w.o.}}$, 
\item a homomorphism 
$\xymatrix{h_k\colon F_m^\times/p^{N''_k} \ar[r] & R_{m_k, N''_k,\chi}},$
\end{itemize}
such that the ideal of $R_{m_k,N'_k,\chi}\otimes_{\mca{O}[[\Gamma]]} 
S_{\mf{P}_k}$
is generated by the image of $h(\kappa_{m_k,N''_k}(\eta;n_k))$. 
Therefore, we obtain 
\begin{equation}\label{inequarityoflength}
Be_{\mf{P}}k \succ \mrm{length}_{S_{\mf{P}_k}}\big(
R_{m_k,N'_k,\chi}/h(\kappa_{m_k,N''_k}(\eta;n_k))
S_{\mf{P}_k}
\big).
\end{equation}
We denote by 
$\xymatrix{\bar{h}_k 
\colon F_m^\times/p^{N'_k} \ar[r] & R_{m_k, N'_k,\chi}},$
the $R_{m_k,N'_k,\chi}$-linear homomorphism 
induced by $h_k$.

For a moment, we fix an integer $k \ge N(\mf{P})$, 
and put $N':=N_k$, $N'':=N''_k$, $m:=m_k$, 
$n=n_k$ and $\bar{h}_k:=\bar{h}$ for simplicity. 
We put 
$$N_{H_n}:=\sum_{\sigma \in H_n}\sigma 
\in \bb{Z}[H_n].$$
Let $\xymatrix{\nu_{H_n}\colon R_{m,N',\chi} \ar[r] 
& R_{m,N,\chi}[H_n]}$ be the isomorphism of 
$R_{m,N',\chi}[H_n]$-module defined by 
$1 \mapsto N_{H_n}$. 
Note that $R_{m,N',\chi}[H_n]$ is a injective 
$R_{m_k,N'_k,\chi}$-module, so there exist an 
$R_{m_k,N'_k,\chi}$-linear homomorphism 
$\xymatrix{
\tilde{h} \colon 
(F_m(\mu_n)^\times/p^{N'})_\chi \ar[r] & 
R_{m,N',\chi}[H_n] 
}$
which makes the diagram 
$$\xymatrix{
(F_m^\times/p^{N'})_\chi \ar@{^{(}->}[d] \ar[r]^{h} &  
R_{m,N',\chi} \ar[d]^{\nu_{H_n}} \\
(F_m(\mu_n)^\times/p^{N'})_\chi \ar[r]^{\tilde{h}} & 
R_{m,N',\chi}[H_n] 
}$$
commute.

Note that by Shapiro's lemma, 
we have a natural isomorphism 
$$\xymatrix{
H^1(\bb{Q}(\mu_n),\mathbf{T}_\chi/p^{N'}) 
\ar[r]^(0.47){\simeq} &
\displaystyle\varprojlim_{m'}(F_{m'}(\mu_n)^\times/p^{N'})_\chi.
}$$
Then, by Lemma \ref{liftinghom} (ii), 
we obtain the following lemma. 

\begin{lem}
There exists a homomorphism 
$$\xymatrix{{\tilde{h}_\infty}\colon
H^1(\bb{Q}(\mu_n),\mathbf{T}_\chi/p^{N'}) 
\ar[r]  & 
\mca{O}[[\Gamma]][H_n]/p^{N''}=
\displaystyle\varprojlim_{m'} R_{m',N',\chi}[H_n]}$$ 
of $\mca{O}[[\Gamma]][H_n]/p^{N'}$-modules 
which makes the diagram 
$$\xymatrix{
H^1(\bb{Q}(\mu_n),\mathbf{T}_\chi/p^{N'})
\ar[rr]^(0.56){\tilde{h}_\infty} \ar[d] & &
\mca{O}[[\Gamma]][H_n]/p^{N'} \ar[d]^{\mod{(\gamma^{p^m}-1)}}  \\
(F_m(\mu_n)^\times/p^{N'})_\chi \ar[rr]^(0.56){\tilde{h}} & & 
R_{m,N',\chi}[H_n] 
}$$
commute. 
\end{lem}

Recall that we define a non-negative integer $s$ 
by 
$$p^s=(S_{\mf{P}}:\mca{O}[[\Gamma]]/\mf{P}).$$
Let us show the following proposition.  

\begin{prop}\label{big commutative diagram}
There exists a $S_{\mf{P}_k}[H_n]$-linear homomorphism 
$$\xymatrix{\tilde{h}_{\mf{P}_k,N'}\colon
H^1(\bb{Q}(\mu_n),(T_\chi
\otimes_{\mca{O}}S_{\mf{P}_N})/p^{N'})
\ar[r] & S_{\mf{P}_k}[H_n]/p^{N'}}$$
which makes the diagram 
$$\xymatrix{
H^1(\bb{Q}(\mu_n),\mathbf{T}_\chi/p^{N'}) \ar[d] 
\ar[rr]^(0.6){p^{3s}\tilde{h}_\infty} & &
\mca{O}[[\Gamma]][H_n]/p^{N'} \ar[d] \\
H^1(\bb{Q}(\mu_n),(T_\chi
\otimes_{\mca{O}}S_{\mf{P}_N})/p^{N'}) 
\ar@{-->}[rr]^(0.6){\tilde{h}_{\mf{P}_k}} & & 
S_{\mf{P}_k}[H_n]/p^{N'}} 
$$
commute. 
Here, the vertical maps in this diagram  are the natural map.
\end{prop}

\begin{proof}
We divide the vertical maps in some short steps, 
and we will construct suitable homomorphisms 
step by step. 
Recall that define a non-negative integer  $s$ by 
$$p^s=(S_{\mf{P}}:\mca{O}[[\Gamma]]/\mf{P}).$$
In order to prove the proposition, 
we need the following 
Lemma \ref{elementary lemmaSP} and its corollary.
(Note that the following lemma is proved
by similar arguments to Lemma \ref{elementary lemma}, 
so we omit the proof.) 
\begin{lem}\label{elementary lemmaSP}
Let $M$ be a $\mca{O}[[\Gamma]]$-module. 
Then, the kernel and the cokernel of the natural 
$\Lambda_{\chi}/\mf{P}_k$-linear map 
$$\xymatrix{M \ar[r] & 
M\otimes_{\mca{O}[[\Gamma]]/{\mf{P}_k}}S_{\mf{P}_k}}$$
is annihilated by $p^s$. 
\end{lem}
The following corollary follows from 
Lemma \ref{elementary lemmaSP} by 
the similar arguments to Corollary \ref{kernel}. 
\begin{cor}\label{kernelSP}
Let $\xymatrix{f\colon M \ar[r] & N}$ be a homomorphism of 
$\mca{O}[[\Gamma]]/\mf{P}_k$-modules. 
Consider the $S_{\mf{P}_k}$-linear map 
$$\xymatrix{f\otimes S_{\mf{P}_k} \colon 
M\otimes_{\mca{O}[[\Gamma]]/\mf{P}_k}S_{\mf{P}_k}
\ar[r] & N\otimes_{\mca{O}[[\Gamma]]/\mf{P}_k}}S_{\mf{P}_k}$$ 
induced by $f$. 
Then, the $S_{\mf{P}_k}$-submodule
$p^s \Ker (f\otimes S_{\mf{P}_k})$ of $M\otimes_{\mca{O}[[\Gamma]]/\mf{P}_k}$
is contained in the image of $\Ker f$. 
\end{cor}

Here, we return to the proof of the proposition. 
From the exact sequence 
$$\xymatrix{ 0 \ar[r] & \mathbf{T}_\chi/p^{N'}  
\ar[rr]^{\times f_k(T)} & &
\mathbf{T}_\chi/p^{N'} \ar[r] & 
\mathbf{T}_\chi/(p^{N'}\mathbf{T}_\chi+\mf{P}_k\mathbf{T}_\chi) 
\ar[r] & 0, }$$
it follows that the natural homomorphism 
$$\xymatrix{H^1(\bb{Q}(\mu_n),\mathbf{T}_\chi/p^{N'})
\otimes_{\mca{O}[[\Gamma]]}(\mca{O}[[\Gamma]]/\mf{P}_k) \ar[r] & 
H^1\big( \bb{Q}(\mu_n),
\mathbf{T}_\chi/(p^{N'}\mathbf{T}_\chi+\mf{P}_k\mathbf{T}_\chi) 
 \big) }$$ 
is injective.  So, by Corollary \ref{kernelSP}, 
the kernel of the $S_{\mf{P}_k}[H_n]$-linear map 
$$\xymatrix{H^1(\bb{Q}(\mu_n),\mathbf{T}_\chi/p^{N'})
\otimes_{\mca{O}[[\Gamma]]}S_{\mf{P}_k} \ar[r] & 
H^1\big( \bb{Q}(\mu_n),
\mathbf{T}_\chi/(p^{N'}\mathbf{T}_\chi+\mf{P}_k\mathbf{T}_\chi) 
 \big)\otimes_{\mca{O}[[\Gamma]]}S_{\mf{P}_k} }$$ 
is annihilated by $p^s$. 

The $\mca{O}[[\Gamma]][H_n]$-linear homomorphism $\tilde{h}'_{\infty}$ 
induces an $S_{\mf{P}_k}[H_n]$-linear homomorphism 
$$\xymatrix{
\tilde{h}^{(0)} 
\colon H^1(\bb{Q}(\mu_n),\mathbf{T}_\chi)\otimes_{\mca{O}[[\Gamma]]} 
S_{\mf{P}_k} \ar[r] & S_{\mf{P}_k}[H_n]/p^{N'}.}$$
For simplicity, we put 
\[
\overline{\mathbf{T}}_{\chi,k}:=
\mathbf{T}_\chi/(p^{N'}\mathbf{T}_\chi+\mf{P}_k\mathbf{T}_\chi). 
\]
Let $\mrm{Im}_{S_{\mf{P}_k}}$ 
be the image of $S_{\mf{P}_k}[H_n]$-linear map
$$\xymatrix{H^1(\bb{Q}(\mu_n),\mathbf{T}_\chi/p^{N'})
\otimes_{\mca{O}[[\Gamma]]}S_{\mf{P}_k} \ar[r] & 
H^1( \bb{Q}(\mu_n),
\overline{\mathbf{T}}_{\chi,k})\otimes_{\mca{O}[[\Gamma]]}S_{\mf{P}_k} 
}$$ 
Then, there exists an $S_{\mf{P}_k}[H_n]$-linear homomorphism 
$$\xymatrix{\tilde{h}^{(1)} \colon
\mrm{Im}_{S_{\mf{P}_k}}
\ar[r] & S_{\mf{P}_k}[H_n]/p^{N'}}$$
which makes the diagram 
$$\xymatrix{H^1(\bb{Q}(\mu_n),\mathbf{T}_\chi)\otimes_{\mca{O}[[\Gamma]]} 
S_{\mf{P}_k} \ar[rr]^(0.56){p^s\tilde{h}^{(0)}}  \ar[d] & & 
S_{\mf{P}_k}[H_n]/p^{N'} \\
\mrm{Im}_{S_{\mf{P}_k}} \ar[rru]_{\tilde{h}^{(1)}} & & 
}$$ 
commute.
Note that $S_{\mf{P}_k}[H_n]/p^{N'}$ is an injective 
$S_{\mf{P}_k}[H_n]/p^{N'}$-module
since $S_{\mf{P}_k}/p^{N'}$ is a quotient of 
a complete discrete valuation ring $S_{\mf{P}_k}$ 
with finite residue field. 
So, we can extend $\tilde{h}^{(1)}$ to 
an $S_{\mf{P}_k}[H_n]/p^{N'}$-linear map 
$$\xymatrix{\tilde{h}^{(1)} \colon
H^1( \bb{Q}(\mu_n),
\overline{\mathbf{T}}_{\chi,k} )\otimes_{\mca{O}[[\Gamma]]}S_{\mf{P}_k} 
\ar[r] & S_{\mf{P}_k}[H_n]/p^{N'},}$$
and we obtain the commutative diagram 
$$\xymatrix{
H^1(\bb{Q}(\mu_n),\mathbf{T}_\chi/p^{N'}) \ar[d] 
\ar[rr]^(0.6){p^{s}\tilde{h}_\infty} & &
\mca{O}[[\Gamma]][H_n]/p^{N'} \ar[d] \\
H^1( \bb{Q}(\mu_n),
\overline{\mathbf{T}}_{\chi,k} )\otimes_{\mca{O}[[\Gamma]]}S_{\mf{P}_k} 
\ar@{-->}[rr]^(0.6){\tilde{h}^{(1)}} & & 
S_{\mf{P}_k}[H_n]/p^{N'}.} 
$$
Note that 
$\mrm{Tor}_1^{\Lambda_{\chi}/\mf{P}_k}\big(
\overline{\mathbf{T}}_{\chi,k},
S_{\mf{P}_k}/(\mca{O}[[\Gamma]]/{\mf{P}_k})\big)$ and  
$\overline{\mathbf{T}}_{\chi,k}\otimes_{\mca{O}[[\Gamma]]/\mf{P}_k}S_{\mf{P}_k}/
\overline{\mathbf{T}}_{\chi,k}$ are annihilated by $p^s$, 
the kernel of the natural homomorphism 
$$\xymatrix{
H^1( \bb{Q}(\mu_n),
\overline{\mathbf{T}}_{\chi,k} )
\otimes_{\mca{O}[[\Gamma]]}S_{\mf{P}_k} 
\ar[r] & H^1\big( \bb{Q}(\mu_n),
({T}_\chi\otimes_{\mca{O}}S_{\mf{P}_k})/p^{N'}\big)}$$
is annihilated by $p^{2s}$. 
Then, from the injectivity of $S_{\mf{P}_k}[H_n]/p^{N'}$, 
there exists an $S_{\mf{P}_k}[H_n]/p^{N'}$-linear map 
$$\xymatrix{\tilde{h}_{\mf{P}_k} \colon
H^1\big( \bb{Q}(\mu_n),
({T}_\chi\otimes_{\mca{O}}S_{\mf{P}_k})/p^{N'}\big)
\ar[r] & S_{\mf{P}_k}[H_n]/p^{N'},}$$
which makes the diagram 
$$\xymatrix{
H^1( \bb{Q}(\mu_n),
\overline{\mathbf{T}}_{\chi,k})
\otimes_{\mca{O}[[\Gamma]]}S_{\mf{P}_k} \ar[d] 
\ar[rr]^(0.6){p^{2s}\tilde{h}^{(1)}} & &
S_{\mf{P}_k}[H_n]/p^{N'} \\
H^1\big( \bb{Q}(\mu_n),
({T}_\chi\otimes_{\mca{O}}S_{\mf{P}_k})/p^{N'}\big)
\ar@{-->}[rru]_(0.56){\tilde{h}_{\mf{P}_k}} & & } $$
commute. 
The homomorphism ${\tilde{h}_{\mf{P}_k}}$ is what we want to construct, 
and this completes the proof of Proposition \ref{big commutative diagram}. 
\end{proof}

We identify $S_{\mf{P}_k}/p^{N'}$ with 
$S_{\mf{P}_k}/p^{N'}[H_n]^{H_n}$ 
as an $S_{\mf{P}_k}/p^{N'}[H_n]$-module by the isomorphism 
$$\xymatrix{\nu_{H_n}\colon
S_{\mf{P}_k}/p^{N'}
\ar[r] & S_{\mf{P}_k}/p^{N'}[H_n]^{H_n};\ 
1 \ar[r] & N_{m+1/m},}$$
and let 
\[
\xymatrix{
{h}_{\mf{P}_k}\colon H^1(\bb{Q},(T_\chi
\otimes_{\mca{O}}S_{\mf{P}_N})/p^{N'}) 
\ar[r] & S_{\mf{P}_k}/p^{N'}
}
\]
be the homomorphism induced by $\tilde{h}_{\mf{P}_k}$. 
Note that since we assume the ideal $\mf{P}_k+p^{N'}\mca{O}[[\Gamma]]$ 
contains $\gamma^{p^m}-1$, the natural homomorphism
$$\xymatrix{
H^1(\bb{Q}(\mu_n),\mathbf{T}_\chi/p^{N'}) \ar[r] &  
H^1(\bb{Q}(\mu_n),(T_\chi
\otimes_{\mca{O}}S_{\mf{P}_N})/p^{N'}) } $$
factors through 
$$H^1\big( \bb{Q}(\mu_n),\mathbf{T}_{\chi}
/((\gamma^{p^{m'}}-1)\mathbf{T}_{\chi}
+p^{N''}\mathbf{T}_{\chi}) \big)
\simeq (F_m(\mu_n)^\times/p^{N'})_\chi.$$
We denote the image of 
$H^1(\bb{Q}(\mu_n),\mathbf{T}_\chi/p^{N'})$
in $(F_m(\mu_n)^\times/p^{N'})_\chi$ 
by $\mathrm{Im}_F$. 
Then, by Proposition \ref{big commutative diagram}, 
we obtain the commutative diagram 
$$\xymatrix{
H^1(\bb{Q}(\mu_n),\mathbf{T}_\chi/p^{N'}) \ar[d] 
\ar[rr]^(0.6){p^{3s}\tilde{h}_\infty} & &
\mca{O}[[\Gamma]][H_n]/p^{N'} \ar[d] \\
\mathrm{Im}_F 
\ar[rr]^(0.6){p^{3s}\tilde{h}} \ar[d] & &  
R_{m,N',\chi}[H_n] \ar[d] \\
H^1(\bb{Q}(\mu_n),(T_\chi
\otimes_{\mca{O}}S_{\mf{P}_N})/p^{N'}) 
\ar[rr]^(0.6){\tilde{h}_{\mf{P}_k}} & & 
S_{\mf{P}_k}[H_n]/p^{N'} \\
H^1(\bb{Q},(T_\chi
\otimes_{\mca{O}}S_{\mf{P}_N})/p^{N'}) 
\ar[rr]^(0.6){{h}_{\mf{P}_k}} \ar@{^{(}->}[u] & & 
S_{\mf{P}_k}/p^{N'} \ar@{^{(}->}[u]_{\nu_{H_n}} .} 
$$

By the norm compatibility of circular units, we can define the element 
$$\eta_{\infty}^{D_n}:=(\eta_m(n)^{D_n}) 
\in H^1(\bb{Q}(\mu_n),\mathbf{T}_\chi/p^{N'})
=\displaystyle\varprojlim_{m'}
(F_{m'}(\mu_n)^\times/p^{N'})_\chi.$$
In particular, we have 
\[
\eta_m(n)^{D_n} \in 
\mrm{Im}_F= (F_m(\mu_n)^\times/p^{N'})_\chi.
\] 
Let $\mca{O}_S$ be a ring  
which is isomorphic to $S_{\mf{P}_k}$ as 
a $\mca{O}_\chi$-algebra, 
and we assume that the Galois group $G_{\bb{Q}}$ 
acts on $\mca{O}_S$ trivially. 
The action of $G_{\bf{Q}}$ on $S_{\mf{P}_k}$ defines 
a continuous character 
$$\xymatrix{\rho\colon \Gamma \ar[r] & \mca{O}_S^{\times}. }$$
We regard both $T_\chi\otimes_{\mca{O}}\mca{O}_S$ and 
$T_\chi\otimes_{\mca{O}}S_{\mf{P}_k}$ as free 
$\mca{O}_S$-modules of rank one. 
Let 
$$\eta\otimes \rho:=
\{\eta\otimes \rho_{F} \in H^1
\big(F,(T_\chi \otimes_{\mca{O}}{S_{\mf{P}_k}}) \big) \}_{F 
\subset \mca{K}}$$ 
be the Euler system for 
$(T_\chi \otimes_{\mca{O}}{S_{\mf{P}_k}},\mca{K}/\bb{Q},\Sigma)$
which is the twist of the Euler system 
$$\eta^{(k)}:=\{\eta_{m'}^{(k)}(n') 
\in H^1
\big(F_{m'}(\mu_{n'}),T_\chi \big)
\}_{F_{m'}(\mu_{n'}) \subset \mca{K}}$$ for 
$(T_\chi\otimes_{\mca{O}}\mca{O}_S,\mca{K}/\bb{Q},\Sigma)$
by the character $\rho$ 
in the sense of \cite{Ru5}. 
Since we assume $n \in \mca{N}_{N''}^{\mrm{w.o.}}\subseteq 
\mca{N}^{\mrm{w.o.}}(\mca{S}_{N'}(\mf{P}_k))$, 
we can define 
the Kolyvagin derivative class 
$$\kappa_{0,N'}(\eta\otimes \rho ;n) 
\in H^1\big(\bb{Q}, 
(T_\chi \otimes_{\mca{O}}{S_{\mf{P}_k}})/p^{N'} 
\big)$$
of the Euler system $\eta\otimes \rho$, 
whose image in $H^1(\bb{Q},
(T_\chi \otimes_{\mca{O}}{S_{\mf{P}_k}})/p^{N'})$
coincides with the image of $\eta_m(n)^{D_n}$. 
By Proposition \ref{Comparisonofkappa}, we have 
$$\kappa_{0,N'}(\eta\otimes \rho;n)=
\kappa(\eta\otimes \rho)_n
\in H^1\big(\bb{Q}, 
(T_\chi \otimes_{\mca{O}}{S_{\mf{P}_k}})/p^{N'} 
\big),$$
where $\kappa(\eta\otimes \rho)_n$ is the $n$-component of 
the Kolyvagin system defined by the Euler system $\eta \otimes \rho$. 
Therefore, we obtain 
\begin{align*}
\beta e_{\mf{P}}k & \succ 
\mrm{length}_{S_{\mf{P}_k}}\big( 
S_{\mf{P}_k}/(p^{N'}S_{\mf{P}_k}+
h(\kappa_{m_k,N''_k}(\eta;n_k))S_{\mf{P}_k}) \big)  \\
& \sim 
\mrm{length}_{S_{\mf{P}_k}}\big( 
S_{\mf{P}_k}/(p^{N'}S_{\mf{P}_k}+
p^{3s}\tilde{h}
(\eta_{m_k}(n_k)^{D_{n_k}})S_{\mf{P}_k}) \big)  \\
& =
\mrm{length}_{S_{\mf{P}_k}}\big( 
S_{\mf{P}_k}/(p^{N'}S_{\mf{P}_k}+
{h}_{\mf{P}_k}
(\kappa(\eta\otimes \rho)_n)S_{\mf{P}_k} \big) \\
& \ge \min\{ \partial_i({\mf{P}_k}),N'\} 
= \partial_i({\mf{P}_k})  \sim a_k  \\
& \sim \alpha e_{\mf{P}}k
\end{align*}
Thus, we obtain $\beta \ge \alpha$, 
and this completes the proof of Theorem \ref{local lower bounds}.  
\end{proof}

\end{document}